\documentclass[11pt,reqno]{amsart}

\numberwithin{equation}{section}
\usepackage{times}
\usepackage{amsmath,amsfonts,amstext,amssymb,amsbsy,amsopn,amsthm,eucal}
\usepackage{mathtools}
\usepackage{dsfont}
\usepackage{graphicx}   
\usepackage{hyperref}
\usepackage{accents}
\usepackage{enumerate}
\usepackage{xcolor}
\usepackage{verbatim}
\usepackage{esint}

\usepackage[normalem]{ulem}
\usepackage{cancel}

\newcommand{\NN}{\mathds{N}}
\newcommand{\RR}{\mathds{R}}

\newcommand{\dC}{\mathds{C}}

\newcommand{\ba}{\text{\bf{a}}}
\newcommand{\bb}{\text{\bf{b}}}

\newcommand{\bp}{\text{\bf{p}}}

\newcommand{\bx}{\text{\bf{x}}}

\newcommand{\Bw}{\Bar{w}}
\newcommand{\Bz}{\Bar{z}}

\newcommand{\cA}{\mathcal{A}}

\newcommand{\cC}{\mathcal{C}}

\newcommand{\cQ}{\mathcal{Q}}

\newcommand{\cR}{\mathcal{R}}

\newtheorem{theorem}[equation]{Theorem}
\newtheorem*{theorem*}{Theorem}

\newtheorem{proposition}[equation]{Proposition}
\newtheorem{lemma}[equation]{Lemma}
\newtheorem{corollary}[equation]{Corollary}
\theoremstyle{definition}
\newtheorem{definition}[equation]{Definition}

\theoremstyle{remark}
\newtheorem{remark}[equation]{Remark}
\theoremstyle{remark}

\theoremstyle{remark}

\theoremstyle{remark}
\theoremstyle{remark}

\allowdisplaybreaks

\begin{document}

\thanks{}
\thanks{}

\title[]{General regularity obstructions for the arrival time equation}

\author{Yiqi Huang, Jingze Zhu}
\address[Yiqi Huang]{Department of Mathematics, MIT, 77 Massachusetts Avenue, Cambridge, MA 02139-4307, USA}
 \email{yiqih777@mit.edu}
 \address[Jingze Zhu]{Institute of Geometry and Physics, USTC, No. 99 Xiupu Road, Shanghai 201315, China}
 \email{zhujz3@ustc.edu.cn}

\begin{abstract}
We study the regularity of the arrival time function associated with mean curvature flow, formulated as a degenerate elliptic equation encoding the singular structure of the flow. We identify a systematic mechanism obstructing higher regularity, arising from higher-order asymptotic expansions near spherical singularities. We construct uncountably many low-regularity arrival time functions in all dimensions. In particular, we settle the question of smoothness in the planar case by showing that even for convex curve shortening flow the arrival time need not be $C^{25}$. Our companion paper \cite{hz} proves that every such arrival time is $C^{24,\alpha}$ for any $0<\alpha<1$. Hence the obstruction at order $25$ gives the optimal regularity threshold. 

Our approach introduces new analytic ingredients, including a precise correspondence between elliptic asymptotic expansions and parabolic long-time asymptotics of the associated rescaled mean curvature flow, a complexification of the arrival time equation, and a method for prescribing higher-order asymptotic expansions.
\end{abstract}
\maketitle

\setcounter{tocdepth}{1}
\tableofcontents

\section{Introduction}

A smooth family of hypersurfaces $\{M_t\}_{t\in [0,T)}$ in $\RR^{n+1}$ is called a mean curvature flow if it satisfies the equation $\partial_t x = \vec{H}$ where $x$ denotes the position vector and $\vec{H}$ denotes the mean curvature vector. Mean curvature flow can be viewed as the gradient flow of the area functional and is one of the central geometric evolution equations in differential geometry and geometric analysis. It has deep connections with material science, minimal surfaces and topology, and has become a fundamental tool for understanding the geometry and topology of hypersurfaces.

The arrival time function associated with mean curvature flow provides a unified analytic framework for geometric evolution, encoding the entire flow --- including its singularities --- into a single nonlinear partial differential equation. Introduced through the level set method, the arrival time formulation replaces the evolving mean convex hypersurface with a scalar function
\begin{equation*}
    U : \Omega \subset \mathbb{R}^{n+1} \to \mathbb{R}
\end{equation*}
solving the degenerate elliptic boundary value problem
\begin{equation} \label{eq:arrival}
\begin{cases}
|D U| \cdot \mathrm{div}\!\left(\frac{D U}{|D U|}\right) = -1  & \text{in } \Omega,\\
U = 0 & \text{on } \partial\Omega,
\end{cases}
\end{equation}
in the viscosity sense. The function $U(x)$ represents the time at which the evolving hypersurface passes through the point $x$. This perspective originated in numerical analysis through the work of Osher and Sethian \cite{os}, and was subsequently developed in the viscosity framework by Chen--Giga--Goto \cite{cgg} and Evans--Spruck \cite{es1,es2,es3,es4}. The arrival time formulation also admits a game-theoretic interpretation, developed by Kohn and Serfaty \cite{ks}; see also \cite{sp,e,koh,gl,ks09,ks10} for related developments.

\subsection{Regularity background}

While formulation \eqref{eq:arrival} provides global existence and uniqueness beyond the classical regime, it raises subtle questions concerning regularity. Away from critical points, equation \eqref{eq:arrival} is non-degenerate and solutions are smooth. All analytic difficulties therefore concentrate at critical points, where $|D U| = 0$ and the equation becomes degenerate.

Although viscosity theory initially yields only Lipschitz regularity \cite{cgg,es1}, the true regularity is substantially stronger. In the convex setting, Huisken \cite{huisken} proved that the flow shrinks smoothly to a round point and that the arrival time is $C^2$. This reflects the asymptotic spherical structure of the flow near extinction. Colding and Minicozzi \cite{cm16a,cm18} later established that for general mean convex flows the arrival time is twice differentiable everywhere with uniformly bounded Hessian, and satisfies \eqref{eq:arrival} in the classical sense even at critical points, building on their work on uniqueness of tangent flows and structure of the singular set \cite{cm15,cm16b}. Their work shows that the quadratic expansion of $U$ at a critical point is determined by the tangent flow, reflecting a strong rigidity at second order. Beyond quadratic order, however, the behavior of $U$ is no longer rigid: higher regularity depends on finer asymptotic information and is not determined solely by the tangent flow.

This naturally leads to the question: \emph{how regular can the arrival time function be?}

In higher dimensions $n\geq 2$, \v{S}e\v{s}um \cite{se} and later Strehlke \cite{str} constructed convex flows for which the arrival time fails to be $C^3$. More precisely, for each $k\geq 2$, they proved that there exist arrival time functions in the form of
\begin{align}\label{Intro;arrival-expansion;11}
	U(x) = T_0 - \frac{|x|^2}{2n} + |x|^{\frac{k(k-1)}{n}} H_k + o\left(|x|^{\frac{k(k-1)}{n}+k}\right)
\end{align}
where $H_k$ is a homogeneous harmonic polynomial of degree $k$. When $k = 2$, the term $|x|^{\frac{2}{n}} H_2$ prevents $C^3$ regularity. Furthermore, Sun-Xue \cite{sx24} proved that expansion \eqref{Intro;arrival-expansion;11} is generic with respect to initial perturbation of the corresponding mean curvature flow when $n\geq 2$, therefore proving generic low regularity. More generally, whenever $\frac{k(k-1)}{n} \notin 2\mathbb{Z}$, the presence of non-even-integer power in  the term $|x|^{\frac{k(k-1)}{n}} H_k$ creates obstruction to smoothness.

In contrast, this mechanism breaks down in dimension $n=1$. Indeed, when $n=1$ the exponent reduces to $k(k-1)$, which is always an even number. Hence the corresponding term $|x|^{k(k-1)} H_k$ is a polynomial and cannot obstruct smoothness. Moreover, Kohn-Serfaty \cite{ks} proved that the arrival time function of convex curve shortening flow is $C^3$ with vanishing third derivatives at the extinction point, which played a key role in the game-theoretic interpretation of the flow. These facts lead to the question whether planar arrival time function might enjoy substantially higher regularity. In fact, it was even unknown whether or not the arrival time functions for convex curves are smooth.

\subsection{Main results}

Our first result settles the question of smoothness in the planar case.

\begin{theorem}[]\label{thm:curve25}
There exist arrival time functions arising from convex compact curve shortening flows that are not $C^{25}$.
\end{theorem}

Thus, although planar arrival time functions exhibit improved low-order regularity, this phenomenon does not persist at higher order.  Theorem \ref{thm:curve25} is a special case of a more general theorem below, where we identify a systematic obstruction mechanism governing the regularity of the arrival time in all dimensions.

Together with our companion paper \cite{hz}, Theorem \ref{thm:curve25} identifies the optimal regularity for arrival time in the plane. In \cite{hz}, we prove that every arrival time function on convex $\Omega\subset \RR^2$ is $C^{24,\alpha}$ for any $0<\alpha<1$. Moreover, its twenty-fourth derivatives satisfy a logarithmic Lipschitz estimate and thus the examples constructed here fail to be $C^{24,1}$. Hence the order-25 obstruction is the first one that can occur, and it does occur for convex flows by Theorem \ref{thm:curve25}. Its location is striking: the leading expansions in the plane are polynomials, and the compatibility conditions at the lower critical degrees 4, 9, and 16 vanish.

\begin{theorem}\label{thm:general}
Suppose $n \ge 1$. For each $k \ge 2$ and each nonzero homogeneous harmonic polynomial $H_k$ of degree $k$, there exist uncountably many asymptotically distinct arrival time functions $U$ on a bounded convex domain $\Omega \subset \mathbb{R}^{n+1}$ that are not $C^N$
and satisfy the asymptotic expansion near $0$:
\begin{equation}\label{Intro;arrival-expansion;12}
    U(x) = T_0 - \frac{|x|^2}{2n} + |x|^{\frac{k(k-1)}{n}} H_k + o\!\left(|x|^{\frac{k(k-1)}{n}+k}\right),
\end{equation}
where
$$
N =
\begin{cases}
(k^2+k-1)^2, & n=1, \\
\left\lceil \frac{j(j-1)}{n}+j \right\rceil,
& n\ge2,
\end{cases}
$$
and $j \ge k$ is the smallest integer such that $\frac{j(j-1)}{2n} \notin \mathbb{Z}$.
\end{theorem}

\begin{remark}
    A stronger statement than uncountability can be proved up to minor modification. The examples constructed in the proof arise from finite-dimensional families of prescribed higher-order asymptotic data. In these families, the condition that the arrival time have the corresponding higher regularity imposes nontrivial algebraic equations on the parameters. As a comparison, Sun--Xue \cite{sx24} proved a global genericity result in dimensions $n\ge2$: generically, the rescaled flow near a spherical singularity is governed by the slowest nontrivial mode, which yields the leading term $|x|^{2/n}H_2$ in the arrival time expansion and hence obstructs $C^3$ regularity. This leading-order mechanism does not apply to convex curve shortening flow, since when $n=1$ all exponents $k(k-1)$ are even and the leading terms are polynomial. 
\end{remark}

Theorem \ref{thm:curve25} is a direct consequence of Theorem \ref{thm:general} by taking $k=2$ and $n=1$. By the work of Strehlke \cite{str}, any nontrivial convex arrival time function necessarily admits a leading-order asymptotic expansion of the form \eqref{Intro;arrival-expansion;11} with nonzero $H_k$. Consequently, the asymptotic profiles appearing in Theorem \ref{thm:general} exhaust all possible nontrivial leading-order behaviors. Therefore, Theorem \ref{thm:general} is in full generality in the sense that it covers all the possible nontrivial leading order behaviors $|x|^{\frac{k(k-1)}{n}} H_k$. Theorem \ref{thm:general}  is new not just when dimension $n=1$, but also in the case of higher dimension $n\geq 2$ and $\frac{k(k-1)}{n}\in 2\mathbb{Z}$. 

The lack of smoothness obstruction from $|x|^{\frac{k(k-1)}{n}}H_k$ causes major difficulties. To address the difficulties, we introduce several new ideas, which we outline now.

\subsection{Strategy and new ingredients}

We now outline the main ideas.

\emph{(1) Algebraic obstruction via complexification.}

To understand the possible regularity of $U$ near a spherical singularity, we analyze the recursive structure imposed by \eqref{eq:arrival}. Assuming sufficient regularity --- for clarity, we temporarily assume $U$ is analytic --- we expand $U$ near the singular point $0$
\begin{equation*}
    U = T_0 - \frac{|x|^2}{2n} + \sum_{m \ge 3} P_m,
\end{equation*}
where each $P_m$ is a homogeneous polynomial of degree $m$.

Substituting this expansion into \eqref{eq:arrival} yields a recursive system of the form
\begin{equation*}
    L_1 P = L_2 P + L_3 P,
\end{equation*}
where
\begin{equation*}
    L_1 P = |x|^2 \Delta P - D^2P(x,x)
\end{equation*}
preserves degree, while $L_2$ and $L_3$ increase degree, sending $P_k$ to homogeneous polynomials of degrees $2k-2$ and $3k-4$, respectively. Equating homogeneous components of each degree transfers \eqref{eq:arrival} into a hierarchy of algebraic equations for the coefficients of the expansion.

The degeneracy of the arrival time equation is reflected in the fact that $L_1$ is not invertible. Indeed, one verifies that a homogeneous polynomial $P$ lies in $\ker L_1$ if and only if
\begin{equation*}
    P = |x|^{\frac{k(k-1)}{n}} H_k
\end{equation*}
for some harmonic homogeneous polynomial $H_k$ of degree $k$, where $\frac{k(k-1)}{n}$ is an even number; see Lemma \ref{Lem;linear-term-equation}.  We remark that this extends to any smooth homogeneous function, but we do not use it in this paper.

At degrees for which $L_1$ is invertible, the recursive equation uniquely determines $P_m$ from lower-order terms.
However, at the critical degrees $m = \frac{k(k-1)}{n}+k$, the operator $L_1$ has nontrivial kernel, and the corresponding homogeneous component is determined only up to an element of $ \{ |x|^{\frac{k(k-1)}{n}} H_k \} $. These kernel elements therefore serve as free parameters in the expansion. Solvability at these degrees requires that the right-hand side be orthogonal to the cokernel of $L_1$, yielding a compatibility condition. This compatibility condition constitutes the \emph{obstruction equation}.

A central difficulty is to determine whether this obstruction equation is genuinely nontrivial. Because the recursive system couples multiple lower-order terms, its structure is highly nonlinear and a priori it is not clear that any genuine obstruction occurs at finite order. In dimension $n=1$, the possible critical degrees are $k^2$. A detailed analysis shows that the obstruction equations for $k=2,3,4$ are automatically satisfied, while the first nontrivial obstruction appears at $k=5$, corresponding to order $25$. This explains the appearance of $C^{25}$ as the threshold of regularity.

To make the computation tractable in dimension one, we introduce a complexification of the arrival time equation. Since the space of harmonic homogeneous polynomials of fixed degree in two variables is two-dimensional, it can be naturally identified with the complex numbers. Under this identification, the recursive compatibility condition reduces to an explicit algebraic equation involving the coefficients corresponding to $k=2,3,4$. We prove that this algebraic obstruction is nonvanishing, and hence provides a genuine compatibility condition.

Consequently, if the coefficients of degrees $4,9,16$ fail to satisfy this algebraic relation, then the recursive equation at degree $25$ cannot be solved, contradicting the assumed Taylor expansion. It follows that $U$ cannot be $C^{25}$. The detailed derivation of the obstruction equation
is carried out in Section 4.

\emph{(2) Elliptic-parabolic correspondence.}

Having identified the algebraic obstruction at the level of the elliptic expansion, the next problem is to construct solutions realizing prescribed asymptotics near the singular point.

In earlier work, \cite{str} adapted the stable manifold theorem for geometric evolution equations \cite{nai} to the rescaled mean curvature flow, constructing invariant manifolds of solutions that converge to the round sphere with a prescribed exponential rate. Similar ideas were employed by \v{S}e\v{s}um \cite{se}. These results establish a correspondence between the leading asymptotic mode of the graph function of the rescaled flow over the sphere, and the leading-order term $|x|^{\frac{k(k-1)}{n}}H_k$ in the elliptic expansion of $U$.

However, this correspondence controls only the first asymptotic mode. In the present setting, the obstruction arises from higher-order interactions, and therefore depends essentially on finer asymptotic information. To address this limitation, we establish a refined elliptic-parabolic correspondence: we relate the full higher-order expansion of $P$ to the higher order asymptotics of the graph function
of the associated rescaled mean curvature flow over the sphere.

This refined correspondence allows us to translate the algebraic obstruction for the elliptic expansion into a dynamical obstruction for the higher-order asymptotics of the rescaled flow. The construction of this correspondence is carried out in Section 3, and the transfer of the obstruction equation to the parabolic setting
is developed in Section 5.

\emph{(3) Prescription of higher-order asymptotics.}

The final step is to construct solutions of the rescaled mean curvature flow whose asymptotic expansions violate the obstruction equation.

While Strehlke’s work \cite{str} allows the prescription of the leading asymptotic mode, this is insufficient in the planar case, where the obstruction arises from the interactions among multiple higher-order asymptotic modes.  Controlling only the leading order mode does not determine the finer asymptotic structure required to detect or violate the obstruction.

To overcome this difficulty, we develop in Section \ref{Sec;prescribe-higher-asymptotic} a refined stable manifold analysis for the rescaled mean curvature flow.  More precisely, we analyze the evolution of the difference of graph functions and construct a stable manifold for this difference equation. This allows us to prescribe higher order mode while preserving slower modes (which includes leading order mode). By iterating the argument, we are able to prescribe multiple higher-order asymptotic modes simultaneously. 

The proof of Theorem \ref{thm:general} is constructive in every dimension. Starting from a convex rescaled flow with the prescribed leading mode, we use ingredient (3) to construct finite-dimensional families in which selected higher-order modes vary independently while all slower modes remain fixed. Ingredient (2) transfers $C^N$ regularity into compatibility conditions on these parabolic coefficients, and ingredient (1) shows that these conditions
are nontrivial.

\subsection*{Disclosure of AI tools} The authors use ChatGPT 5.6 Pro for assistance in identifying mathematical and linguistic typos.
\subsection*{Acknowledgements}
The authors are grateful to Prof. Toby Colding for his continual support and encouragement. The authors would like to thank Prof. Natasa \v{S}e\v{s}um for some helpful conversations. Yiqi Huang is supported by a Simons Dissertation Fellowship.

\section{Preliminaries and setup}\label{section;pre}

\subsection{Background on mean curvature flow}

Let   $\{M_t\}_{t\in I}$ denote a mean curvature flow (MCF) in $\RR^{n+1}$. The rescaled mean curvature flow (RMCF) centered at $(x_0,T_0)$ is defined by
\begin{align}\label{Def;RMCF;1}
	\bar{M}_{\tau} = e^{\frac{\tau}{2}}(M_{T_0-e^{-\tau}}-x_0) = (T_0-t)^{-\frac{1}{2}}(M_t - x_0)
\end{align}
where $\tau = -\ln(T_0-t)$.

We say that the MCF has a spherical singularity at $(x_0,T_0)$ if the corresponding RMCF centered at $(x_0,T_0)$ converges to $\mathbb{S}^n = \mathbb{S}^n(\sqrt{2n})$. Equivalently,  there is a $\tau_0$ such that for all $\tau \geq \tau_0$ one can write $\bar{M}_{\tau}$ as a graph $v$ over $\mathbb{S}^n$:
\begin{align}\label{Def;RMCF-graph;1}
	\bar{M}_{\tau} = \{ y + v(y, \tau) \vec{n} :  y \in \mathbb{S}^{n}(\sqrt{2n}), \ \vec{n} \text{ is the outward unit normal on }\mathbb{S}^n(\sqrt{2n}) \}
\end{align}
with $||v||_{C^l}\rightarrow 0$ as $\tau\rightarrow \infty$ for all $l\in \mathbb{N}$.

The evolution equation of $v$ is 
\begin{align}
	\partial_{\tau} v = Lv+ N(v, \nabla v, \nabla^2 v)
\end{align}
where $L = \Delta + 1$ is the linearized operator on $\mathbb{S}^n(\sqrt{2n})$ and $N$ collects the nonlinear terms. For a derivation of this equation see for example
\cite{se,str}.

The eigenvalues of the linear operator $\Delta + 1$ on $\mathbb{S}^n(\sqrt{2n})$  are given by 
\begin{align}\label{eigenvalue;lambda-k;1}
	\lambda_k = \frac{k(k+n-1)}{2n}-1,\quad k=0,1,2,....
\end{align}
The corresponding eigenspace is denoted by $E_k$ with dimension $\binom{n+k}{n}  - \binom{n+k-2}{ n}$ for all $k\geq 2$ and with dimension $\binom{n+k}{n}$ when $k=0,1$. Note that any element in $E_k$ is the restriction of a homogeneous harmonic polynomial in $\RR^{n+1}$ of degree $k$.   When $n=1$, the dimension of $E_k$ is 2 for all $k\geq 1$.

Given any MCF with spherical singularity at $(x_0,T_0)$, one has the leading order asymptotic behavior of $v$ by the work of Strehlke \cite{str}:
\begin{theorem}\cite[Theorem 2.2]{str} If the flow is not a family of self-similarly shrinking round spheres, then there exist $k\geq 2$, a nonzero $P_k\in E_k$, and $\sigma>0$ such that, for each $l\geq0$,
\begin{align}
	\Big|v - e^{-\lambda_k \tau}P_k(y) \Big|_{C^l(\mathbb{S}^{n})}\leq Ce^{-(\lambda_k + \sigma)\tau}
\end{align}
for some $C = C(l)$ and for all large $\tau$. 

\end{theorem}

Now we turn to the discussion of the arrival time function. Suppose that we have a convex MCF $\{M_t\}_{t\in [0,T_0)}$ with a spherical singularity at $(x_0,T_0)$. Let $\Omega\subset\RR^{n+1}$ be the convex domain bounded by $M_{0}$. By the work of \cite{cgg,es1}, there is a unique arrival time function $U(x)$ in $\Omega$ satisfying
\begin{align}\label{arrival-time-eqn-1}
	|DU|\cdot \textrm{div}(\frac{DU}{|DU|}) = -1
\end{align}
with Dirichlet boundary condition $U\equiv 0$ on $\partial\Omega$. In addition the level set
\begin{align}\label{Def;arrival-time;1}
	\{x\in \Omega: U(x) = t\}
\end{align}
moves by the mean curvature, and in particular coincides with $M_t$ for all $t\in [0,T_0)$.

Equivalently, if can define $U(x) = t$ when $x\in M_t$, then $U$ is an  arrival time function in $\Omega$, where $\Omega$ is the domain bounded by $M_0$. 

Therefore, convex MCF has a canonical correspondence with arrival time functions. By abuse of notation we might say $(x_0, T_0)$ to be the singular point of $U$ without mentioning the corresponding MCF.

By \cite{hui93}, $U$ is  $C^2$ in $\Omega$, analytic away from $x_0$ and $DU(x_0) = 0$, \ $D_{ij}U(x_0) = -\frac{1}{n}\delta_{ij}$. In particular, $U$ has the expansion:
\begin{align}\label{Taylor;1}
	U(x) = T_0 - \frac{|x-x_0|^2}{2n} + o(|x-x_0|^2) \quad \quad \text{ as } x\rightarrow x_0.
\end{align}
Furthermore, in a small neighborhood of $x_0$, $DU(x) = 0$ if and only if $x = x_0$.  If $\Omega$ is convex, then $x_0$ is the only critical point on $\Omega$, not just near $x_0$.

\begin{remark}
In the case of  $n= 1$,  $U$ is $C^3$ by \cite{ks} with vanishing third derivatives and therefore the expansion can be improved  to be:
\begin{align*}
	U(x) = T_0 - \frac{|x-x_0|^2}{2n} + o(|x-x_0|^3)
\end{align*}
\end{remark}

We summarize these results as follows.
\begin{theorem}\cite{huisken,cgg,es1,hui93}\label{maxp;1}
	For a bounded  convex   domain $\Omega$, there exists a unique arrival time function $U$ with $U \equiv 0$ on $\partial \Omega$. Then $U > 0$ in $\Omega$, $U$ is $C^2$ in $\Omega$ and is analytic away from a unique critical point. 
\end{theorem}

\subsection{Notations}

We will frequently use multi-index notation. All the multi-indices throughout the paper are assumed to contain only \textbf{nonnegative} integers. We list a few definitions we might need:
\begin{definition}\label{multiindex-def;1}
	Consider \textbf{nonnegative}  multi-index $\alpha = (\alpha_1,...,\alpha_p)$.
 \begin{itemize}
 	\item Define $|\alpha| = \sum_{k=1}^p \alpha_k$ and $\alpha! = \alpha_1! \alpha_2!\cdots \alpha_p!$
 	\item Define $Zero(\alpha) = \#\{i: \alpha_i = 0\}$  to be the number of 0 element in $\alpha$. 
 	\item Define $D(\alpha)$ to be the number of distinct elements in $\alpha$. 
 	
 	For example, if $\alpha = (3,1,1,2)$ then $D(\alpha) = 3$; if $\alpha = (4,1)$ then $D(\alpha) = 2$; if $\alpha = (2,2,2)$ then $D(\alpha) = 1$.
 	
 	\item Define $\hat{\alpha}$ to be a multi-index derived from $\alpha$, consisting of the multiplicities of the distinct entries of $\alpha$, listed in descending order.
 
 		For example, if $\alpha = (3,1,1,2)$ then $\hat{\alpha} = (2,1,1)$; if $\alpha = (4,1)$ then $\hat{\alpha} = (1,1)$; if $\alpha = (2,2,2)$ then $\hat{\alpha} = (3)$.
 \end{itemize}
\end{definition}

\noindent\textbf{Convention A:} All constants in the form of $C^{(A)}_{\beta}$ or $C^{(A)}_{\ba}$ in this paper depend only on the multi-index $\beta$ or double multi-index $\ba$ (to be defined in Section \ref{Sec;elliptic}) respectively and dimension $n$. All polynomials in the form of $\cQ^{(A)}_{m}$ or $\cQ^{(A)}_{i,j}$ depend only on their index $m$ or $(i,j)$ respectively and dimension $n$.  Indeed, they all can be defined universally and inductively for all $\beta, \ba, m$ or $(i,j)$  respectively, similar to the fashion of Definition \ref{Def;C1C2}, although we will only use them for small $\beta, \ba, m$ or $(i,j)$. These constants should be considered as being associated with the structure of the PDE, not with the actual regularity of the solution. In other words, they are algebraic constants, not analytic constants. Sometimes we suppress subscripts when this causes no ambiguity. This rule applies to all such constants or polynomials, either in the statement of theorem or in the proof. \\
Note that in the Section \ref{Sec;elliptic} and \ref{Sec;parabolic-obstruction} we specialize in the case $n+1=2$, where these constants will be even independent of dimension.

\noindent\textbf{Convention B}: We reserve $\bx$ exclusively for polynomial variables (along with superscripts and subscripts). This is different from the coordinate $x$ in $\RR^{n+1}$.

\section{Elliptic-parabolic correspondence}\label{Sec;elliptic-to-parabolic}
Consider a convex MCF $\{M_t\}_{t\in [0,T_0)}$ with a spherical singularity at $(x_0,T_0)$. Suppose that $M_0$ bounds $\Omega$. By Section \ref{section;pre}, there is a corresponding arrival time function $U$ in $\Omega$ that is analytic away from $x_0$ and a corresponding RMCF centered at $(x_0, T_0)$ with a graph function $v$. The regularity of $U$ then boils down to the regularity at $x_0$ only.

The goal of this section is to establish the relation between the asymptotic expansion of $U$ near $x_0$  and the fine asymptotic behavior of the corresponding graph function $v$:

\begin{theorem}\label{Taylor-to-asymptotic}
Suppose that $N\geq3$ and that an arrival time function $U$ on $\Omega$ satisfies
\begin{align}\label{Taylor;2}
	\Big|T_0-U  - \sum_{k=0}^{N} \frac{P_k(x-x_0)}{n}\Big| = o(|x-x_0|^{N}) \quad \text{ as } x\rightarrow x_0
\end{align} 
where each $P_k$ is a homogeneous polynomial of degree $k$,  and $P_2(x) = \frac{|x|^2}{2}$. Let $v$ be the graph function of the corresponding RMCF centered at $(x_0,T_0)$, then there exist a unique collection of  functions $f_1,..., f_{N-2}\in C^{\infty}(\mathbb{S}^n(\sqrt{2n}))$ such that $v$ satisfies the following \textbf{half-integer} asymptotics:
\begin{align}\label{v;asymp;0}
	\Big|v(y,\tau)-\sum_{k=1}^{N-2} f_k(y) e^{-\frac{k\tau}{2}}\Big|_{C^0(\mathbb{S}^n)} = o(e^{-\frac{N-2}{2}\tau}) \quad \text{ uniformly in } y, \text{ as } \tau \rightarrow \infty
\end{align}
Moreover, $f_1,...,f_{N-2}$ are given by the induction formula:
\begin{align}\label{fm;induction;3}
	f_0 := \sqrt{2n},\quad \quad f_m = - \frac{1}{2P_2f_0 }\sum_{\substack{\alpha_1+\cdots+\alpha_k +k =m+2\\ 0\leq \alpha_1,...,\alpha_k\leq m-1}} (2n)^{1-k/2} P_{k}f_{\alpha_1} \cdots f_{\alpha_k}  \quad (1\leq m \leq N-2)
\end{align}
where we restrict $P_k$ on the sphere $\mathbb{S}^n(\sqrt{2n})$ and set $P_0 = P_1\equiv 0$. In general, the functions $f_0, f_1,...,f_{N-2}$ are restrictions of homogeneous polynomials in $\RR^{n+1}$ onto $\mathbb{S}^n(\sqrt{2n})$. 
\end{theorem}

\begin{remark}
The expansion \eqref{Taylor;2} is weaker than $C^N$ regularity. In the proof, we do not use differentiability of $U$; the expansion itself is sufficient. The arrival-time structure is used to identify the level sets of $U$ with the evolving hypersurfaces, and, in the convex setting, to ensure that for each direction $\theta$ the function $r\mapsto \sqrt{T_0-U(x_0+r\theta)}$ is strictly increasing for small $r>0$.  This radial monotonicity allows us to invert the expansion along each ray.
\end{remark}
\begin{remark}
	Note that $P$ is slightly different from those in the next sections. In particular, $P_2$ is not zero here but $P_2 =  0$ in the next sections.  
\end{remark}
\begin{remark}
The converse is also true, namely fine asymptotic of $v$ implies the expansion of $U$. Since we do not need this direction in the present paper, we omit it.
\end{remark}

The proof of Theorem \ref{Taylor-to-asymptotic} will rely on the following elementary Lemma \ref{Taylor;inverse-fuction}.

\begin{lemma}\label{Taylor;inverse-fuction}
	Suppose that $N\geq 3$ and $s = s(r): [0,\delta_1]\rightarrow [0,\delta_2]$ is a continuous, strictly increasing function satisfying the expansion:
	\begin{align}\label{s-r;11}
		\Big|s(r)^2 - \sum_{k=2}^{N} a_k r^k\Big| \leq E_1(r)\cdot r^{N} 
	\end{align}
	where $a_2 > 0$ and $E_1$ is a continuous increasing function with $E_1(0)=0$. Set  $M_1 = \max\{a_2, |a_3|,..., |a_N|, a_2^{-1}\}$ and $a_1 =0$.  Then there is another continuous increasing function $E_2$ with $E_2(0) =0$ and $s_2>0$, both depending only on $E_1, M_1 ,\delta_1, N$, such that the inverse function $r = r(s)$ is an increasing function satisfying the expansion:
	\begin{align}\label{r-s;12}
		\Big|r(s) - \sum_{k=0}^{N-2} b_k s^{k+1} \Big| \leq  E_2(s)\cdot s^{N-1}  \quad \quad \text{ when } 0\leq s \leq s_2
	\end{align}
	where $b_0, b_1, ...,b_{N-2}$ satisfy:
	\begin{align}\label{bk;induction;1}
	\sum_{\substack{|\alpha | + k  = p \\ \alpha_1,\cdots,\alpha_k \geq 0}} a_{k} b_{\alpha_1}...b_{\alpha_k} = \begin{cases}
		1, \quad \quad p = 2\\
		0, \quad \quad 3\leq p\leq N
	\end{cases}   
\end{align}
Equivalently, $b_0,...,b_{N-2}$ satisfy the  induction formula:
\begin{align}\label{bk;induction;2}
	b_0 = \sqrt{\frac{1}{a_2}}, \quad  \quad    b_m = - \frac{1}{2a_2b_0}\cdot\sum_{\substack{|\alpha| + k=m+2 \\ 0\leq \alpha_1,...,\alpha_k\leq m-1}} a_{k} b_{\alpha_1}...b_{\alpha_k} \quad (1\leq m\leq N-2)
\end{align}
\end{lemma}
\begin{proof}
The existence of inverse function $r(s)$ is clear since $s(r)$ is strictly increasing. Note that our assumption implies that $r(0) = 0$. Throughout the proof, $C > 1, r_0 >0 , s_0>0$  denote constants and $E$ denote a continuous increasing function with $E(0) =0$. All of them \textbf{may vary from line to line but only depend on}  $E_1, M_1,\delta_1, N$.  (Indeed, it can be arranged so that $C, E$ do not depend on $\delta_1$.)

We start from the most rough estimate. Our assumption gives 
\begin{align}\label{s-r;9}
	|s^2 - a_2r^2| \leq C\cdot r^3, \quad  \forall r \in [0,r_0] 
\end{align}
this implies that $r$ and $s$ are comparable up to a multiplicative constant; more precise:
\begin{align}\label{s-r;10}
	\sqrt{\frac{a_2}{2}}\cdot r \leq s\leq \sqrt{2a_2}\cdot r,   \quad  \forall r \in [0,r_0]
\end{align}
The proof will be inductive. By  \eqref{s-r;10}, the rough expansion \eqref{s-r;9} implies that:
\begin{align*}
	|s-\sqrt{a_2}r| \leq C\cdot r^2,  \quad  \forall r \in [0,r_0]
\end{align*}
If we let $b_0 = \sqrt{\frac{1}{a_2}}$ and use \eqref{s-r;10}, the above can be rewritten as:
\begin{align*}
	\Big|r - \sqrt{\frac{1}{a_2}}s\Big|= |r - b_0 s | \leq C\cdot r^2  \leq C \cdot s^2,  \quad  \quad \forall s\in [0, s_0]
\end{align*}
this establishes the induction base. Suppose that we already have the following hold for some $m\leq N -2$:
\begin{align}\label{r-s;13}
	r(s) = \sum_{k=0}^{m-1} b_k s^{k+1} + \mathcal{E}_{m+1}, 
\end{align}
where 
\begin{itemize}
	\item $b_0,...,b_{m-1}$ are computed by induction formula \eqref{bk;induction;2}, which is equivalent to \eqref{bk;induction;1} up to $p = m+1$.
	\item $\mathcal{E}_{m+1} = \mathcal{E}_{m+1}(s)$ satisfies $\big|\mathcal{E}_{m+1}(s)\big|\leq C\cdot s^{m+1}$ when $s\in [0,s_0]$.
\end{itemize}
Plugging into \eqref{s-r;11} and using \eqref{s-r;10}:
\begin{align*}
	\Big| s^2 - \sum_{k=2}^{N} a_k r^k \Big| = \Big| s^2 - \sum_{k=2}^{N} a_k \Big(\sum_{l=0}^{m-1} b_l s^{l+1} + \mathcal{E}_{m+1}\Big)^k \Big| \leq& E(r)\cdot r^{N} \quad \quad \forall  r\in [0, r_0]\\
	\leq & E(s)\cdot s^N  \quad  \quad \forall s\in [0, s_0]
\end{align*}
By direct expansion and the fact that $|\mathcal{E}_{m+1}(s)| \leq  C\cdot s^{m+1}$ for all $s\in [0,s_0]$, we find that: 
\begin{align*}
	\Big|\Big(\sum_{l=0}^{m-1} b_l s^{l+1} + \mathcal{E}_{m+1}\Big)^k  -  \Big(\sum_{\substack{0\leq \alpha_1,\cdots,\alpha_k\leq m-1}} b_{\alpha_1}\cdots b_{\alpha_k}s^{\alpha_1+\cdots+\alpha_k + k} + k b_0^{k-1}\mathcal{E}_{m+1}s^{k-1} \Big)\Big| \leq C\cdot s^{m+k+1}, 
\end{align*}
for any $s\in [0,s_0]$. Putting them together we get:
\begin{align*}
	 \Big| s^2 - \sum_{k=2}^{N}a_k \Big(\sum_{\substack{0\leq \alpha_1,\cdots,\alpha_k\leq m-1}} b_{\alpha_1}\cdots b_{\alpha_k}s^{\alpha_1+\cdots+\alpha_k + k} + k b_0^{k-1}\mathcal{E}_{m+1}s^{k-1}\Big)\Big| \leq E(s)\cdot s^N + C\cdot s^{m+k+1}  
\end{align*}
for all $s\in [0,s_0]$. Note that it is possible ${\alpha_1+\cdots+\alpha_k} \geq m+k +1$. Thus $C\cdot s^{m+k+1}$ may not be much smaller than some of $s^{\alpha_1+\cdots+\alpha_k}$. 

Continuing the computation, after rearranging summation:
\begin{align*}
	\Big| s^2 -&  \Big(\sum_{p=2}^{m+1}s^p \sum_{\substack{|\alpha| + k=p}} a_{k} b_{\alpha_1}...b_{\alpha_k}\Big) - s^{m+2} \cdot \Big(\sum_{\substack{|\alpha| + k=m+2\\ 0\leq \alpha_1,...,\alpha_k\leq m-1}} a_{k} b_{\alpha_1}...b_{\alpha_k}\Big) - 2a_2b_0\mathcal{E}_{m+1}s \Big| \leq  C\cdot s^{m+3} + E(s)\cdot s^{N}   
\end{align*}
for all $s\in [0,s_0]$. 
Note that all indices $\alpha_i's$ appeared here are at most $m-1$.  By the induction hypothesis, the first bracket cancels $s^2$. The leftover is:
\begin{align*}
	\Big| -s^{m+2} \sum_{\substack{|\alpha| + k=m+2\\ 0\leq \alpha_1,...,\alpha_k\leq m-1}}a_{k} b_{\alpha_1}\cdots b_{\alpha_k} - 2a_2b_0\mathcal{E}_{m+1}s \Big| \leq  C\cdot s^{m+3} + E(s)\cdot s^{N},    \quad  \quad \forall s\in [0, s_0]
\end{align*}
Then we can write:
\begin{align}\label{Em;1}
	\Big|\mathcal{E}_{m+1} - b_m s^{m+1}\Big| \leq  C\cdot s^{m+2} + E(s)\cdot s^{N -1},  \quad  \quad \forall s\in [0, s_0]
\end{align}
where $b_{m}$ is defined by 
\begin{align}\label{bk;induction;3}
	b_m = - \frac{1}{2a_2b_0}\cdot\sum_{\substack{|\alpha| + k=m+2\\ 0\leq \alpha_1,...,\alpha_k\leq m-1}}a_{k} b_{\alpha_1}...b_{\alpha_k} 
\end{align}
But this is precisely \eqref{bk;induction;2}, which is also equivalent to \eqref{bk;induction;1} when $p = m+2$.
Putting \eqref{Em;1} into \eqref{r-s;13}:
\begin{align}\label{r-s;14}
	\Big|r(s) - \sum_{k=0}^{m} b_k s^{k+1}\Big|\leq C\cdot s^{m+2} + E(s)\cdot s^{N -1},     \quad  \quad \forall s\in [0, s_0]
\end{align}
When $m \leq N-3$ we have $N -1 \geq m+2$, hence $\big|r(s) - \sum_{k=0}^{m} b_k s^{k+1}\big|\leq C\cdot s^{m+2}$ for all $s\in [0,s_0]$. In these cases,  \eqref{r-s;13} is verified for $m+1$. By induction, \eqref{r-s;13} is true for $m = N-2$.

Finally, we carry out the above argument one last time for $m = N-2$. In this case $N -1 < m+2$, thus \eqref{r-s;14} becomes 
\begin{align}\label{r-s;15}
	\Big|r(s) - \sum_{k=0}^{N-2} b_k s^{k+1}\Big|\leq E(s)\cdot  s^{N-1},   \quad  \quad \forall s\in [0, s_0]
\end{align}
with $b_0,...,b_{N-2}$ specified by \eqref{bk;induction;1} or equivalently by \eqref{bk;induction;2}. Setting $E_2, s_2$ to be $E, s_0$ in \eqref{r-s;15}, the assertion then follows. 
\end{proof}

Now we proceed to the proof of Theorem \ref{Taylor-to-asymptotic}:
\begin{proof}[Proof of Theorem \ref{Taylor-to-asymptotic}]
Throughout the proof, the constant $C$ is always independent of $x, r, s, \tau, \theta, y$.
By spatial translation, we may assume that $x_0 = 0$. Then our assumption implies that 
\begin{align}\label{Taylor;3} 
	\Big|T_0-U  - \sum_{k=2}^{N} \frac{P_k(x)}{n}\Big| \leq E_1(|x|)\cdot |x|^{N}
\end{align}
for some continuous increasing function $E_1$ with $E_1(0) = 0$,
where $x\in \RR^{n+1}$ and each $P_k$ is a degree $k$ homogeneous polynomial. Moreover by \eqref{Taylor;1} we know that
\begin{align}\label{P2;1}
	P_2(x) = \frac{|x|^2}{2}
\end{align}
We can set $P_1\equiv 0$ so that the summation starts from $k=1$. 
Note that $T_0 - U$ can be regarded as "remaining time" to the spherical singularity.

We need two substitutions:
\begin{itemize}
	\item Polar coordinate $x = r\theta$, where $r = |x|$ and $\theta\in \mathbb{S}^n(1)$. 
	\item $T_0 - U = s^2 = s(r, \theta)^2$
\end{itemize}
Then we can rewrite \eqref{Taylor;3} and \eqref{P2;1} as
\begin{align}\label{s-r;1}
	\Big| s(\theta,r)^2 - \sum_{k=2}^{N} \frac{P_k(\theta)}{n}\cdot r^k \Big| \leq E_1(r)\cdot  r^{N}
\end{align}
\begin{align}
	P_2(\theta)\equiv \frac{1}{2}
\end{align}

By the convexity, we know that  for each $\theta \in \mathbb{S}^n$, $s(\theta,r)$ is an increasing function of $r$ on the interval $[0,\delta]$. To see this, it suffices to show that $\langle \theta, D U|_{r\theta} \rangle < 0$ whenever it is defined and $r>0$. Since $x_0 = 0$ is in the interior of $M_t$ for all $t>0$ and $Du$ is parallel to normal vector, this is true by strict convexity. Moreover, there is a small $\delta$ independent of $\theta$ such that $s(\theta,r)$ can be defined on $r\in [0,\delta]$ for each $\theta$.

 We can now apply Lemma \ref{Taylor;inverse-fuction} on \eqref{s-r;1} for each fixed $\theta$. 
Note that  $P_2,P_3,\cdots, P_N, P_2^{-1}$ are uniformly bounded on unit sphere,  we can find a continuous increasing function $E_2$ with $E_2(0) =0$ and a constant $s_2 > 0$, both of which \textbf{independent} of $\theta$, such that $r$ can be regarded as a function of $s,\theta$ that is increasing in $s$ variable and satisfies:
\begin{align}\label{r-s;1}
	\Big| r(\theta, s) - \sum_{k=0}^{N-2} \tilde{f}_k(\theta) s^{k+1}\Big| \leq E_2 (s)\cdot s^{N-1}  \quad \forall s\in [0, s_2]
\end{align}
where
\begin{align}\label{tilde-P0}
			 \tilde{f}_0(\theta) = \sqrt{\frac{n}{P_2(\theta)}} \equiv  \sqrt{2n}
\end{align} 
and
\begin{align}\label{tilde-Pk}
	\tilde{f}_m(\theta) = - \frac{1}{2P_2(\theta)\tilde{f}_0(\theta)}\sum_{\substack{|\alpha| + k=m+2\\ 0\leq \alpha_1,...,\alpha_k\leq m-1}} P_{k}(\theta) \tilde{f}_{\alpha_1}(\theta)\cdots\tilde{f}_{\alpha_k}(\theta) \quad (1\leq m \leq N-2)
\end{align}

So far we are working with each individual $\theta$, but we reached the same formula \eqref{tilde-P0}, \eqref{tilde-Pk} for all $\theta$. Therefore, if we view $\tilde{f}_{m}$ as a function over $\mathbb{S}^n(1)$, then \eqref{tilde-P0} and \eqref{tilde-Pk} extends to induction formula for functions. Furthermore, we can extend $\tilde{f}_m$ to $\RR^{n+1}$ to be $m$-homogeneous for all $0\leq m \leq N-2$, i.e. 
\begin{align*}
	\tilde{f}_m(x) = \tilde{f}_m(r\theta) := r^m\tilde{f}_m(\theta)
\end{align*} 
under polar coordinate $x= r\theta$. Clearly \eqref{tilde-P0} still holds in $\RR^{n+1}$ under the extension.  Since each $P_k$ is $k$-homogenous, comparing the homogeneous degree on both sides of \eqref{tilde-Pk} we find that \eqref{tilde-Pk} continue hold for all extended $\tilde{f}_m$ on $\RR^{n+1}\backslash\{0\}$.
Since $P_1 \equiv 0$ and $P_2\neq 0$ away from 0, the induction formula \eqref{tilde-Pk} along with \eqref{tilde-P0} implies that $\tilde{f}_0,...,\tilde{f}_{N-2}$ are all smooth functions on $\RR^{n+1}\backslash \{0\}$.

\begin{remark}

Indeed each $P_k$ is divisible by $|x|^2$ (see e.g Corollary \ref{Cor;divisble} for $n=1$ case). Therefore, the induction formula implies that each $\tilde{f}_k$ is a homogeneous polynomial of degree $k$ and thus is smooth on $\RR^{n+1}$ across 0, but we don't need this now. We do not prove this fact for $n\geq 2$ as it will not be used in the current paper.
\end{remark}

Under the polar coordinate, one can write $x$ in terms of $v, \theta, s$. To this end, recall that definition of the arrival time function implies that $x \in M_{U(x)}$. Using substitution
\begin{align}
	y = \sqrt{2n}\theta\in \mathbb{S}^n(\sqrt{2n}) \quad and\quad   \tau = -\ln(T_0-U(x)) = -2\ln s
\end{align} we can convert into RMCF:
\begin{align}
	|x| =& e^{-\frac{\tau}{2}}(\sqrt{2n} + v(y,\tau)) 
\end{align}
This means
\begin{align}
	r(\theta, s) = s\cdot \Big(\sqrt{2n} + v\big(y, -2\ln s\big)\Big)
\end{align}
Plugging into \eqref{r-s;1} we obtain:
\begin{align}
	\Big| s\cdot \Big(\sqrt{2n} + v\big(y, -2\ln s\big)\Big)  - \tilde{f}_0\Big(\frac{y}{\sqrt{2n}}\Big)\cdot s - \sum_{k=1}^{N-2} \tilde{f}_k\Big(\frac{y}{\sqrt{2n}}\Big) s^{k+1}\Big| \leq E_2(s)\cdot s^{N-1} \quad \forall s\in [0,s_2]
\end{align}
By \eqref{tilde-P0},  $\tilde{f}_0\Big(\frac{y}{\sqrt{2n}}\Big)\cdot s =  \sqrt{2n}\cdot s$, thus giving the following simplification:
\begin{align}
	\Big| v(y,-2\ln s)\cdot s - \sum_{k=1}^{N-2} \tilde{f}_k\Big(\frac{y}{\sqrt{2n}}\Big) s^{k+1}\Big| \leq E_2(s)\cdot s^{N-1} \quad \forall s\in [0,s_2]
\end{align}
By reducing an $s$ factor, using $\tau = -2\ln s$ and that each $\tilde{f}_k$ is a homogeneous of degree $k$ we get:
\begin{align}
	\Big| v(y,\tau) - \sum_{k=1}^{N-2} (2n)^{-\frac{k}{2}} \tilde{f}_{k}(y) e^{-\frac{k\tau}{2}}\Big| \leq E_2(e^{-\frac{\tau}{2}})\cdot  e^{-\frac{N  - 2}{2}\tau} \quad \forall \tau\in [\tau_2,\infty)
\end{align}
where $\tau_2>0$ is some constant independent of $y$. 
Finally we define functions on $\mathbb{S}^n(\sqrt{2n})$:
\begin{align}
	f_k(y) :=(2n)^{-\frac{k}{2}} \tilde{f}_k(y) , \quad  y\in \mathbb{S}^n(\sqrt{2n}), \quad 0\leq k \leq N-2, 
\end{align}
Then 
\begin{align}
	\Big|v(y,\tau)-\sum_{k=1}^{N-2} f_k(y) e^{-\frac{k\tau}{2}}\Big|_{C^0(\mathbb{S}^n(\sqrt{2n}))} \leq E_2(e^{-\frac{\tau}{2}})\cdot e^{-\frac{N-2}{2}\tau}  \quad \forall \tau\in [\tau_2,\infty)
\end{align}
and $f_1,...,f_{N-2}$ satisfies the induction formula
\begin{align}
	f_0 = \sqrt{2n},\quad \quad f_m = - \frac{1}{2P_2f_0 }\sum_{\substack{|\alpha| + k=m+2\\ 0\leq \alpha_1,...,\alpha_k\leq m-1}} (2n)^{1-k/2} P_{k}f_{\alpha_1} \cdots f_{\alpha_k}  \quad (1\leq m \leq N-2)
\end{align}
 
Since $\tilde{f}_k$ are all smooth away from $0$, $f_k$ are smooth on $\mathbb{S}^n({\sqrt{2n}})$, $1\leq k \leq N-2$.  By the structure of the asymptotic expansion, the uniqueness of $f_1,...f_{N-2}$ is clear.
The Theorem \ref{Taylor-to-asymptotic} is now proved.
\end{proof}

Now we take $P_0, P_1, P_2, P_3,...,P_{N}$ and $f_0, f_1 ,...,f_{N-2}$ from Theorem \ref{Taylor-to-asymptotic}. Keep in mind the special values: $P_0 = P_1\equiv 0, P_2= \frac{|x|^2}{2}$ in $\mathbb{R}^{n+1}$ and $f_0 = \sqrt{2n}$ on $\mathbb{S}^n({\sqrt{2n}})$. 

In the  following corollaries, we aim to express $f_m$ directly in terms of $P_k$ and vice versa:

\begin{corollary}\label{fm;Lemma;1}
There exist universal polynomials $\cQ^{(11)}_{m} = \cQ^{(11)}_{m}(\bx_1,...,\bx_{m-1})$ depending only on $m$ and $n$ with the following property.
Under the same setup as Theorem \ref{Taylor-to-asymptotic}, for all $1\leq m\leq N-2$ we can write:
\begin{align}\label{fm;formula;3}
	\frac{P_{m+2}}{2P_2}  = - (2n)^{-\frac{1}{2}}f_m + \cQ^{(11)}_{m} ( f_{1},\cdots,f_{m-1})
\end{align}
when restricted to $\mathbb{S}^n(\sqrt{2n})$. Moreover, $\cQ^{(11)}_m$ is in the following form:
\begin{align}\label{Q11;def}
	 \cQ^{(11)}_m(\bx_1,...,\bx_{m-1}) = \sum_{\substack{|\beta|=m \\ 1\leq \beta_1 \leq\cdots \leq \beta_p\leq m-1}} C^{(11)}_{\beta}\cdot \bx_{\beta_1}\cdots \bx_{\beta_p}
\end{align}

\end{corollary}
 \begin{proof}
	Using \eqref{fm;induction;3}, this follows from the standard induction. Details are in the Appendix \ref{App: A1}.
\end{proof}

\begin{corollary}\label{fm;Lemma;2}
There exist universal polynomials $\cQ^{(21)}_{m} = \cQ^{(21)}_{m}(\bx_1,...,\bx_{m-1})$ depending only on $m\geq 1$ and $n$ with the following property. Under the same setup as Theorem \ref{Taylor-to-asymptotic}, for all $1\leq m\leq N-2$ we can write $f_m$ as:
\begin{align}\label{fm;formula;4}
	f_m = -\frac{(2n)^{\frac{1}{2}}}{2}\cdot \frac{P_{m+2}}{P_2} + \cQ^{(21)}_{m} \Bigg(\frac{P_{1+2}}{P_2},\cdots,\frac{P_{(m-1)+2}}{P_2}\Bigg)
\end{align}
when restricted to $\mathbb{S}^n(\sqrt{2n})$. Moreover, $\cQ^{(21)}_m$ is in the following form:  
\begin{align}\label{Q21;def}
	\cQ^{(21)}_m(\bx_1,...,\bx_{m-1}) = \sum_{\substack{|\beta|=m \\ 1\leq \beta_1 \leq\cdots \leq \beta_p\leq m-1}} C^{(21)}_{\beta}\cdot \bx_{\beta_1}\cdots\bx_{\beta_p}
\end{align}
where $C^{(21)}_{\beta}$ are universal constants depending only on $\beta$ and $n$. 
\end{corollary}
\begin{proof}
	Using \eqref{fm;induction;3}, this follows from the standard induction. Details are in the Appendix \ref{App: A2}.
\end{proof}

\section{Elliptic obstruction and complexification}\label{Sec;elliptic}

In this section we analyze the Taylor expansion of the arrival time function near a spherical singularity. The arrival time equation induces a recursive algebraic system for the Taylor coefficients. The key feature is that the linearized operator
\begin{align*}
L_1P = |x|^2\Delta P - D^2P(x,x)
\end{align*}
fails to be invertible at certain critical degrees. At these resonant degrees, solvability of the recursive system requires nonlinear compatibility conditions, which we call obstruction equations.

The main result of this section is that in dimension $n=1$ these obstruction equations are genuinely nontrivial. More precisely, the first obstruction occurs at degree $25$ when $k=2$, and more generally the obstruction at degree $(k^2+k-1)^2$ is nonvanishing for every $k\geq2$. See Theorems \ref{Thm;11} and \ref{Thm;12}.

The section is organized as follows. In Subsection 4.1 we derive the recursive elliptic system and identify the kernel of the linearized operator. In Subsection 4.2 we specialize to dimension $n=1$ and introduce a complex formulation which diagonalizes the linearized operator and isolates these monomials. In Subsections 4.3 we construct the obstruction equations recursively and prove their nontriviality.

\subsection{Taylor expansion and the recursive elliptic system}

Let $\{M_t\}_{t\in[0,T_0)}$ be a convex MCF in $\RR^{n+1}$ with a spherical singularity $(x_0,T_0)$ and consider the corresponding arrival time function $U$. By spatial translation we may assume that $x_0 = 0$ and that 
\begin{align*}
	U \in C^{N} \text{ for some } N\geq 2.
\end{align*}
We also assume that $T_0 - U\neq \frac{|x|^2}{2n}$, otherwise we are in the trivial case where the flow is shrinking spheres. 

By \eqref{Taylor;1} we can write:
\begin{align}\label{ansatz-U}
	T_0-U = \frac{1}{n}\Big(\frac{|x|^2}{2} + F(x)\Big)
\end{align}
where the remaining term $F(x) = o(|x|^{2})$ is  $C^{N}$. By our assumption, $F \not\equiv 0$.

We always consider a small neighborhood near 0. Then $|DU| = 0$ if and only if $x = 0$. Therefore, when $x\neq 0$, we can rewrite the equation \eqref{arrival-time-eqn-1} as
\begin{align}\label{arrival-time-eqn-2}
	\Delta U - \frac{D^2U(DU, DU)}{|DU|^2} = -1
\end{align}

Using \eqref{ansatz-U} we compute 
\begin{align}
	\Delta U = -\frac{1}{n}(n +1 + \Delta F),\quad DU = -\frac{1}{n}(x+ DF)
	,\quad D^2_{ij} U =-\frac{1}{n}(\delta_{ij} + D^2_{ij}F),
\end{align}
Plugging into \eqref{arrival-time-eqn-2} we get
\begin{align*}
	n+ 1+\Delta F - (\delta_{ij} + D^2_{ij}F)\cdot \frac{(x_i+D_iF)(x_j+D_jF)}{|x+DF|^2} = n
\end{align*}
Rearranging terms:
\begin{align*}
	n+ 1 + \Delta F - 1 - \frac{D^2_{ij}F\cdot (x_i+D_iF)\cdot (x_j+D_jF)}{|x+DF|^2} = n
\end{align*}
Cancelling $n$ from both sides we get:
\begin{align*}
	\Delta F - \frac{D^2F (x+D F,  x+DF)}{|x+DF|^2} = 0
\end{align*}
This is equivalent to:
\begin{align*}
	\Delta F\cdot  |x+DF|^2 - D^2F (x + DF, x + DF) = 0 
\end{align*}
Now we expand everything and get the following equation:
\begin{align}\label{f-Eqn}
	\Delta F \cdot |x|^2 + 2\Delta F \cdot \langle x,DF\rangle + \Delta F\cdot |DF|^2 = D^2F(x,x)+ 2D^2F(x,DF) + D^2F(DF,DF)
\end{align}

Since $F$ is $C^{N}$, it admits the Taylor expansion
\begin{align}\label{Taylor-expansion-f}
	F(x) =& \sum_{|\beta|\leq N} \frac{1}{\beta!} D^{\beta}F(0)x^{\beta}   + o(|x|^{N}) \quad \quad \text{ as } x\rightarrow 0
\end{align}
Moreover, $DF$ is $C^{N-1}$ and $D^2F$ is $C^{N-2}$, thus:
\begin{align}\label{Taylor-expansion-Df}
	D_i F(x) =& \sum_{|\beta|\leq N-1} \frac{1}{\beta!} D^{\beta}D_iF(0)x^{\beta}  + o(|x|^{N-1 }) \quad \quad \text{ as } x\rightarrow 0
\end{align}
\begin{align}\label{Taylor-expansion-DDf}
	 \quad D_{ij} F(x) = \sum_{|\beta|\leq N-2} \frac{1}{\beta!} D^{\beta}D_{ij}F(0)x^{\beta}  + o(|x|^{N-2})  \quad \quad \text{ as } x\rightarrow 0
\end{align}
If we let $P(x)$ to be the Taylor polynomial of degree at most $N$:
\begin{align}\label{e:taylor F}
	P(x) =& \sum_{|\beta|\leq N} \frac{1}{\beta!} D^{\beta}F(0)x^{\beta}  
\end{align}
Then we can rewrite the above as
\begin{align}\label{Taylor-expansion-Df;2}
	F(x) =& P(x)  + o(|x|^{N}) \quad \quad \text{ as } x\rightarrow 0 \\
	D_i F(x) =& D_iP(x)  + o(|x|^{N-1})  \quad \quad \text{ as } x\rightarrow 0\\
	D_{ij} F(x) =& D_{ij}P(x) + o(|x|^{N-2})  \quad \quad \text{ as } x\rightarrow 0
\end{align}
Moreover, since $F= o(|x|^2)$,  expansion \eqref{Taylor-expansion-f} implies that the polynomial $P$ has no terms of degree $0,1$, or $2$; equivalently, its lowest degree is at least $3$.

By evaluating the error, we can turn the \eqref{f-Eqn} into an equation of $P$ up to a certain degree:
\begin{lemma}
Let $P$ be the Taylor polynomial of $F$ at $0$ up to degree $N$ as in \eqref{e:taylor F}. Then up to degree $N$ we have
\begin{align}\label{P-Eqn;1}
	\Delta P \cdot |x|^2 + 2\Delta P \cdot \langle x,DP\rangle + \Delta P\cdot |DP|^2 
	= D^2P(x,x)+ 2D^2P(x,DP) + D^2P(DP,DP). 	
\end{align}

\end{lemma}
\begin{proof}

Note that $P$ has lowest degree at least 3 and thus $P = O(|x|^{3})$, $DP = O(|x|^{2})$, $D^2P = O(|x|)$. Therefore:
\begin{align}
	\Delta F\cdot |x|^2 =& \Delta P \cdot |x|^2 + o(|x|^{N})\\
	D^2F(x,x) =& D^2 P(x,x) + o(|x|^{N})
\end{align}
Since $N\geq 2$ we have $2N-2 \geq N  $ and therefore:
\begin{align}
	|DF|^2 =& |DP|^2 + o(|x|^{N})\\
	\Delta F\cdot \langle x,DF\rangle  =& \Big(\Delta P + o(|x|^{N-2})\Big)\Big(\langle x,DP\rangle + o(|x|^{N})\Big)\nonumber\\
	 =& \Delta P \langle x,DP\rangle + o(|x|^{N})\\
	 D^2F(x,DF)=&  \Big(D^2 P + o(|x|^{N-2})\Big)\big(x, DP + o(|x|^{N-1})\big)\nonumber\\
	 =& D^2 P (x,DP) + o(|x|^{N})
\end{align}
Using the above and the fact that $3N-4 \geq 2N-2 \geq N $ we can further derive:
\begin{align}
	 \Delta F\cdot |DF|^2 =& \Big(\Delta P + o(|x|^{N-2})\Big)\Big(|DP|^2 + o(|x|^{N})\Big)\nonumber\\
	 =&\Delta P \cdot |DP|^2 + o(|x|^{N})\\
	 D^2F(DF,DF)=&  \Big(D^2 P + o(|x|^{N-2})\Big)\Big(DP + o(|x|^{N-1}), DP + o(|x|^{N-1})\Big)\nonumber\\
	 =& D^2 P (DP,DP) + o(|x|^{N})
\end{align}
 Plugging all above into \eqref{f-Eqn} we get
\begin{align}
	&\Delta P \cdot |x|^2 + 2\Delta P \cdot \langle x,DP\rangle + \Delta P\cdot |DP|^2 \nonumber\\
	=& D^2P(x,x)+ 2D^2P(x,DP) + D^2P(DP,DP) 	+ o(|x|^{N})
\end{align}

Consequently, polynomial $\Delta P \cdot |x|^2 + 2\Delta P \cdot \langle x,DP\rangle + \Delta P\cdot |DP|^2$ and polynomial $D^2P(x,x)+ 2D^2P(x,DP) + D^2P(DP,DP)$ must match up to degree $N$. This finishes the proof.
\end{proof}

The equation naturally separates into a linear term, a quadratic term,
and a cubic term in the Taylor coefficients.
This motivates the following decomposition. 
\begin{align}
    L_1P: =& \Delta P \cdot |x|^2 - D^2 P(x, x), \\ 
    L_2P :=& 2D^2P(x,DP) - 2\Delta P \cdot \langle x,DP\rangle, \\
    L_3P :=& D^2P(DP,DP) -  \Delta P\cdot |DP|^2.
\end{align}
Clearly, \eqref{P-Eqn;1} becomes
\begin{align}\label{P-Eqn;2}
	L_1 P = L_2 P + L_3 P 
\end{align}
up to degree $N$. If $P$ is a homogeneous polynomial of degree $k$, then
$
\deg(L_1P)=k,\deg(L_2P)=2k-2,\deg(L_3P)=3k-4.
$
In particular, $L_1$ captures the lowest-order contribution, while $L_2$ and $L_3$ are higher-order terms. Consequently, the leading-order term in the Taylor expansion of $P$ must lie in the kernel of $L_1$, which explains the algebraic form of the leading-order term in \eqref{Intro;arrival-expansion;12}. 

We now analyze the kernel of the leading operator $L_1$. 

\begin{lemma}\label{Lem;linear-term-equation}
Let $P$ be a nonzero homogeneous polynomial in $\RR^{n+1}$.
Then $ L_1 P= 0 $ if and only if $ P = |x|^{\frac{k(k-1)}{n}} H_k, $ where $k\in \mathbb N$, $\frac{k(k-1)}{n}\in 2\mathbb Z_{\ge 0}$, and $H_k$ is a homogeneous harmonic polynomial of degree $k$.
\end{lemma}

\begin{remark}
More generally, without assuming that $P$ is a polynomial, the homogeneous solutions of $ |x|^2 \Delta P = D^2P(x,x) $ on $\RR^{n+1}\setminus\{0\}$ are precisely of the form $P = |x|^{\frac{k(k-1)}{n}}H_k$, where $k\in\mathbb N$ and $H_k$ is a homogeneous harmonic polynomial of degree $k$. The polynomial assumption is exactly what forces $\frac{k(k-1)}{n}$ to be a nonnegative even integer.
\end{remark}

\begin{proof}
Let $P$ be homogeneous of degree $m$. By Euler's identity,
\begin{align}
    D^2P(x,x)=m(m-1)P.
\end{align}
Thus $L_1P=0$ is equivalent to
\begin{align}\label{eq:L1-eigen-reduction}
    |x|^2\Delta P=m(m-1)P.
\end{align}

We use the standard harmonic decomposition of homogeneous polynomials:
\begin{align}
    P(x)=\sum_{j=0}^{\lfloor m/2\rfloor}|x|^{2j}H_{m-2j}(x),
\end{align}
where each $H_{m-2j}$ is a homogeneous harmonic polynomial of degree
$m-2j$. For a homogeneous harmonic polynomial $H_k$ of degree $k$, one has
\begin{align}
    |x|^2\Delta\big(|x|^{2j}H_k\big) =    2j(2j+2k+n-1)|x|^{2j}H_k.
\end{align}
Since the harmonic decomposition is unique, \eqref{eq:L1-eigen-reduction} holds if and only if each nonzero summand satisfies
\begin{align}
    2j(2j+2k+n-1)=m(m-1),
    \qquad m=2j+k.
\end{align}
Substituting $m=2j+k$ gives
\begin{align}
    2jn=k(k-1).
\end{align}
Hence $2j=\frac{k(k-1)}{n}$, and therefore every nonzero summand has the
form
\begin{align}
    |x|^{\frac{k(k-1)}{n}}H_k.
\end{align}

It remains to note that only one such summand can occur for a fixed degree $m$. Indeed, if $    m=k+\frac{k(k-1)}{n}$,  then the right-hand side is strictly increasing in $k\ge 0$. Thus the harmonic decomposition contains at most one summand satisfying the above relation. Consequently, $ P=|x|^{\frac{k(k-1)}{n}}H_k$ for some homogeneous harmonic polynomial $H_k$ of degree $k$.

Conversely, if $P=|x|^{\frac{k(k-1)}{n}}H_k$ with $H_k$ homogeneous harmonic of degree $k$, then the preceding computation gives
\begin{align}
    |x|^2\Delta P = \frac{k(k-1)}{n}
    \left(\frac{k(k-1)}{n}+2k+n-1\right)P.
\end{align}
Since $P$ is homogeneous of degree $m=k+\frac{k(k-1)}{n}$, the identity
\begin{align}
    \frac{k(k-1)}{n} \left(\frac{k(k-1)}{n}+2k+n-1\right) = m(m-1)
\end{align}
is equivalent to $n\frac{k(k-1)}{n}=k(k-1)$, and hence holds. Therefore
$L_1P=0$.

Finally, since $P$ is assumed to be a polynomial, the exponent $\frac{k(k-1)}{n}$ must be a nonnegative even integer. This completes the proof.
\end{proof}

The critical degrees $ m = k + \frac{k(k-1)}{n} $ are precisely the degrees at which the linearized operator $L_1$ develops a kernel. Away from these degrees, the recursive equation uniquely determines the Taylor coefficients from lower-order terms. At the critical degrees, however, solvability requires nonlinear compatibility conditions.
These are the obstruction equations.

In higher dimensions, the structure of these equations is complicated by the increasing dimension of the space of harmonic homogeneous polynomials. In dimension $n=1$, however, the harmonic spaces are two-dimensional, which allows a substantial simplification through complexification. This will be carried out in the next subsection.

\subsection{Complexification}

\textbf{In the rest of the section we work exclusively in $\RR^2$}, that is when $n = 1$. In dimension two, the kernel of $L_1$ is particularly simple: each critical degree contributes only two resonant monomials, which are conjugate to one another. Passing to complex coordinates diagonalizes the action of $L_1$ on monomials and makes the recursive structure essentially explicit. 

We introduce complex coordinates
\begin{align}
    z= x_1+ix_2; \quad \Bz= x_1-ix_2.
\end{align}

Recall the formula:
\begin{align}
    \partial_z = \frac{1}{2}(\partial_{x_1} - i \partial_{x_2}); &\quad \partial_{\Bz} = \frac{1}{2}(\partial_{x_1} +i \partial_{x_2}).\\
    \partial_{x_1}= \partial_z + \partial_{\Bz}; &\quad \partial_{x_2} = i (\partial_z - \partial_{\Bz}).
\end{align}

Let us always assume that the function is \textbf{real} so that conjugate on function will not change anything. We shall notice that all the complex expression below will always come with a conjugate and therefore making the result real.  Let us record some elementary differential identities on polynomials:

\begin{lemma}\label{Lem;complex-diff}
For any polynomial $P$, we have the following identities:
\begin{enumerate}
    \item $\langle DP , x\rangle = z P_z + \Bz P_{\Bz}$.
    \item $\Delta P= 4 P_{z \Bz}$.
    \item $|DP|^2 = 4P_z P_{\Bz}$
    \item $D^2P(x,DP) = 2zP_{\Bz}P_{zz}  + 2zP_zP_{z\Bz} + 2\Bz P_{\Bz }P_{z\Bz} + 2\Bz P_z P_{\Bz\Bz}$
    \item $D^2P(x,x) = P_{zz}\cdot z^2 + 2P_{z\Bz}\cdot |z|^2 + P_{\Bz\Bz}\cdot \Bz^2$
    \item $D^2P(DP,DP) = 4 (P_{\Bz})^2 P_{zz} + 8 P_z P_{\Bz} P_{z \Bz} + 4 (P_z)^2 P_{\Bz \Bz}$.
\end{enumerate}
\end{lemma}
\begin{proof}
	One can check the first two items when $P = z^a\Bz ^b$ for all $a,b \in \NN$. Then they hold for all polynomial by linearity.
	The last two items follows from the following fact:
	Let $\vec{v} = (v_1,v_2)$ and $w=v_1 + iv_2$. Then $D^2 f(\vec{v},\vec{v}) = f_{zz} w^2 + 2 f_{z \Bz} |w|^2 + f_{\Bz \Bz} \Bw^2$.
	The third identity follows from direct computation.
\end{proof}
 
We now rewrite the operators $L_1, L_2, L_3$ in complex form. Recall that
\begin{align}
    L_1P: =& \Delta P \cdot |x|^2 - D^2 P(x, x), \\ 
    L_2P :=& 2D^2P(x,DP) - 2\Delta P \cdot \langle x,DP\rangle, \\
    L_3P :=& D^2P(DP,DP) -  \Delta P\cdot |DP|^2.
\end{align}

Using Lemma \ref{Lem;complex-diff} we can rewrite $L_1, L_2, L_3$ as
\begin{align}
	L_1P = 2 P_{z \Bz} \cdot |z|^2 -  P_{zz} \cdot z^2 - P_{\Bz \Bz} \cdot \Bz^2. 
\end{align}
\begin{align}
    L_2P &= 2D^2P(x,DP) - 2\Delta P \cdot \langle x,DP\rangle \nonumber \\
    &= 2 \Big( 2z P_{\Bz} P_{zz}  + 2 z P_z P_{z \Bz} +2 \Bz P_{\Bz} P_{z \Bz} + 2 \Bz P_z P_{\Bz \Bz}    \Big) - 8 P_{z \Bz} (P_z \cdot z + P_{\Bz} \cdot \Bz) \nonumber\\
    &= 4 \Big( z P_{\Bz} P_{zz}  - z P_z P_{z \Bz} - \Bz P_{\Bz} P_{z \Bz}  +  \Bz P_z P_{\Bz \Bz}   \Big).
\end{align}
and
\begin{align}
    L_3P &= D^2P(DP,DP) -  \Delta P\cdot |DP|^2 \nonumber\\
    &= 4 (P_{\Bz})^2   P_{zz} + 8 P_z P_{\Bz} P_{z \Bz} + 4 (P_z)^2 P_{\Bz \Bz} - 16 P_z P_{\Bz} P_{z \Bz} \nonumber \\
    &= 4 \Big( (P_{\Bz})^2   P_{zz} - 2 P_z P_{\Bz} P_{z \Bz} +  (P_z)^2 P_{\Bz \Bz} \Big).
\end{align}

This complex formulation allows us to reduce the elliptic equation
to explicit algebraic relations among the coefficients of $P$.
 
\begin{definition}Suppose that $P$ is a polynomial:
\begin{itemize}
	\item $P_n$ is the sum of all degree $n$ monomials in $P$,
	\item $P_{<d}$ is the sum of all degree $<d$ monomials in $P$.
	\item $P_{\leq d}$ is the sum of all degree $\leq d$ monomials in $P$.
	\item $P_{i,j}$ is the sum of all monomials of the form $C z^i \Bz^j$ in $P$.
\end{itemize}
\end{definition}

The following lemma follows from direct computations and Lemma \ref{Lem;linear-term-equation}.

\begin{lemma}\label{Lem;LP;2}
We can compute $L_1, L_2, L_3$ on polynomial $P$:
\begin{align}
	L_1 P =& \sum_{i, j} [(i+j)-(i-j)^2] P_{i,j}\\
	L_2 P =& 4\sum_{\substack{i+k\geq 1 \\ j+l \geq 1}}\Big[jk(k-1) + il(l-1) - (i+j)kl\Big] \frac{P_{i,j}P_{k,l}}{z\Bz} \\
	L_3 P =& 4\sum_{\substack{i+k+p\geq 2 \\ j+l+q \geq 2}}\Big[jlp(p-1) + ikq(q-1) - (il+jk)pq\Big] \frac{P_{i,j}P_{k,l}P_{p,q}}{z^2\Bz^2} 
\end{align}
Moreover, the kernel of $L_1$ restricted on the set of polynomials is  :
\begin{align}
	\text{span} \Big\{z^{\frac{k(k-1)}{2}}\Bz^{\frac{k(k+1)}{2}},\ z^{\frac{k(k+1)}{2}}\Bz^{\frac{k(k-1)}{2}} \big| \ k\geq 0\Big\} 
\end{align}

\end{lemma}

\begin{remark} 
	Note that when $i,j\geq 1$ (which will be justified in the next Theorem), each $P_{i,j}$ is divisible by $z\Bz = |z|^2$. In this case, the quotients above are indeed polynomials. Moreover, the summation is finite because there are only finitely many nonzero $P_{i,j}$.   Note also that by symmetry $\sum_{i,j,k,l\geq 1}(il+jk)P_{i,j}P_{k,l} = \sum_{i,j,k,l\geq 1}2ilP_{i,j}P_{k,l} $.
\end{remark}

Since $L_1$ acts diagonally on monomials, its cokernel can be naturally identified with its kernel. Therefore, at the critical degree $k^2$, solvability of the recursive equation $   L_1P = L_2P + L_3P$ requires that the degree-$k^2$ component of the nonlinear forcing
$L_2P+L_3P$ have vanishing projection onto
\begin{align}
    \mathrm{span}\left\{ z^{\frac{k(k+1)}{2}}\bar z^{\frac{k(k-1)}{2}}, \, z^{\frac{k(k-1)}{2}}\bar z^{\frac{k(k+1)}{2}} \right\}.
\end{align}
These compatibility conditions are precisely the obstruction equations. The explicit formulas in Lemma \ref{Lem;LP;2} reduce the obstruction equations to a purely algebraic recursive system among the coefficients of the monomials $z^a\bar z^b$.  We now organize these recursive interactions systematically.

\subsection{Construction of the obstruction equations}

We continue to work in dimension $n=1$ and use the complexification introduced in the previous subsection. The purpose of this subsection is to organize the recursive structure of these equations systematically. To this end, we introduce a collection of universal algebraic constants which encode the nonlinear interactions among lower-order coefficients. These constants depend only on the combinatorics of the equation and are independent of the particular solution under consideration.

First of all, Lemma \ref{Lem;LP;2} inspired us to define the following constants:

\begin{definition}\label{Def;C3C4}
We define $C^{(3)}$ to be constant depending on a pair of indices:
\begin{align}
	C^{(3)}_{(i,j)} = i + j - (i-j)^2 
\end{align}
Note that this is symmetric about $i,j$. 

We next define $C^{(4)}$ to be constant depending on two pairs of indices and $C^{(s4)}$ be its symmetrization:
	\begin{align} 
		C^{(4)}_{(i,j),(k,l)} :=& jk(k-1) + il(l-1) - (i+j)kl\\
		C^{(s4)}_{(i,j),(k,l)} :=& C^{(4)}_{(i,j),(k,l)} + C^{(4)}_{(k,l),(i,j)} = (j+k-i-l)(jk-il) -2(jk+il)
	\end{align}
	Let $C^{(5)}$ be a constant depending on three pairs of indices and $C^{(s5)}$ be its symmetrization:
	\begin{align}
		C^{(5)}_{(i,j),(k,l), (p,q)} :=& jlp(p-1) + ikq(q-1) - (il+jk)pq\\
		C^{(s5)}_{(i,j),(k,l), (p,q)} :=& C^{(5)}_{(i,j),(k,l), (p,q)}  + C^{(5)}_{(i,j),(p,q), (k,l)}  + C^{(5)}_{(k,l),(i,j), (p,q)}  + C^{(5)}_{(k,l),(p,q), (i,j)} \\
		&+C^{(5)}_{(p,q),(i,j), (k,l)} + C^{(5)}_{(p,q),(k,l), (i,j)}  \nonumber
	\end{align} 
\end{definition}

Using Definition \ref{Def;C3C4}, we can rewrite the operator $L_1,L_2,L_3$ on the set of polynomial $P$ as:
\begin{align}
	L_1 P =& \sum_{i, j} C^{(3)}_{(i,j)}P_{i,j}\\
	L_2 P =& \sum_{\substack{i+k\geq 1 \\ j+l \geq 1}} 4C^{(4)}_{(i,j),(k,l)}\frac{P_{i,j}P_{k,l}}{z\Bz}\\
	L_3 P =& \sum_{\substack{i+k+p\geq 2 \\ j+l+q \geq 2}} 4C^{(5)}_{(i,j),(k,l),(p,q)}\frac{P_{i,j}P_{k,l}P_{p,q}}{z^2\Bz^2}
\end{align}

To organize the recursive structure of the obstruction equations, we introduce a bookkeeping system which records how nonlinear interactions among lower-order resonant modes generate higher-order monomials. We will explain what those constants mean after the definitions.

\begin{definition}
For a double multi-index of \textbf{nonnegative} integers $ \ba  = (\ba^{-}_2, \ba^{+}_2, \ba^{-}_3, \ba^{+}_3, \ba^{-}_4, \ba^{+}_4...)$ with finitely many nonzero elements,  define
\\

\begin{align}\label{Def;S;1}
	S(\ba) =& (S_1(\ba),S_2(\ba)) \nonumber \\
    =&  (1,1) +  \sum_{k\geq2} \Bigg[ \ba^{-}_k \cdot \Big(\frac{k(k-1)}{2}-1 ,\  \frac{k(k+1)}{2}-1\Big)  + \ba^{+}_k \cdot \Big(\frac{k(k+1)}{2}-1 ,\  \frac{k(k-1)}{2}-1\Big) \Bigg].
\end{align}
\end{definition}

\begin{remark}
The quantity $S(\mathbf a)$ should be viewed as the total exponent of the monomial generated by the nonlinear interaction encoded by $\mathbf a$.
\end{remark}

\begin{remark}

Note that the lower index of $\ba$ starts from 2, not 1. Since $\ba$ only has finite nonzero elements, the above is indeed a finite summation. It is clear that we have
\begin{align}\label{Ta-Ineq;2}
	|S(\ba)| = S_1(\ba) + S_2(\ba) = 2 + \sum_{k\geq 2}(\ba^{-}_k + \ba^{+}_k)(k^2-2) \geq 2 + 2|\ba|
\end{align}

\end{remark}

Roughly speaking, each pair
\begin{align*}
\left( \frac{k(k-1)}{2}-1, \frac{k(k+1)}{2}-1 \right), \qquad \left( \frac{k(k+1)}{2}-1, \frac{k(k-1)}{2}-1 \right)
\end{align*}
corresponds to the two resonant monomials $z^{\frac{k(k-1)}{2}}\bar{z}^{\frac{k(k+1)}{2}}$ and $z^{\frac{k(k+1)}{2}}\bar{z}^{\frac{k(k-1)}{2}}$  associated with the critical degree $k^2$. The shift by $-1$ in the definition above comes from the derivative structure of the nonlinear operators $L_2$ and $L_3$. The double multi-index $\mathbf a$ records how many times each resonant mode appears in a nonlinear interaction in order the generate the target high-order monomials. The quantity $S(\mathbf a)$ records the total exponent produced by these interactions, while the constants
\begin{align*}
C_{\mathbf a}^{(1)}, \quad C_{\mathbf a}^{(2)}
\end{align*}
to be defined below, encode the universal coefficients of the target monomials generated recursively by the nonlinear operators $L_2$ and $L_3$.

For example, as we will see, the first nontrivial obstruction occurs at degree $25$ (when $k=5)$ and is generated by the interaction among resonant monomials of degrees $4$, $9$, and $16$ (when $k=2,3, 4$):
\begin{align*}
    z^{10}\bar z^{15} = \frac{(z^3\bar z(z^3\bar z^6)(z^6\bar z^{10})}{z^2\bar z^2}.
\end{align*}
Here $\mathbf a_2^+ = \mathbf a_3^- = \mathbf a_4^- =1$ and all other entries are zero. Hence $\mathbf a=(0,1,1,0,1,0,0,...)$ and $S(\mathbf a)=(10,15)$ records the resulting exponent of $z$ and $\bar z$. $C^{(2)}_{\mathbf a}$ records the coefficient of this term $z^{10} \bar z^{15}$ in
$L_2P+L_3P$.

We now define the coefficients associated with these nonlinear interactions recursively. These constants are universal algebraic quantities depending only on the structure of the equation and not on the particular solution.

\begin{definition}\label{Def;C1C2}
We define constants $C^{(1)}, C^{(2)}$ inductively. First define
	\begin{align}\label{Def;C1C2;base}
		C^{(1)}_{\bb} := 1, \quad C^{(2)}_{\bb} := 0 \quad \text{ whenever } |\bb| = 1.
	\end{align} 
	Suppose that we have defined $C^{(1)}_{\bb}, C^{(2)}_{\bb}$ for all $|\bb|\leq L$. Then for any double multi-index $\bb$ with $|\bb| = L+1\geq 2$, we define:
	\begin{align}\label{Def;C2;induction}
		C^{(2)}_{\bb} :=& 4\sum_{\substack{\ba,\ba'\\ \ba+\ba' =\bb \\ |\ba|, |\ba'|\geq 1}} C^{(4)}_{S(\ba),S(\ba')} C^{(1)}_{\ba}C^{(1)}_{\ba'} + 4 \sum_{\substack{\ba,\ba', \ba''\\ \ba+\ba' + \ba'' =\bb \\|\ba|, |\ba'|, |\ba''|\geq 1}}C^{(5)}_{S(\ba),S(\ba'),S(\ba'')}  C^{(1)}_{\ba}C^{(1)}_{\ba'}C^{(1)}_{\ba''}
	\end{align} 
and
	\begin{align}
		C^{(1)}_{\bb} := \begin{cases}
			\frac{1}{C^{(3)}_{S(\bb)}}C^{(2)}_{\bb} \quad &\text{ if }S(\bb) \neq \Big(\frac{k(k-1)}{2},  \frac{k(k+1)}{2}\Big) \text{ and } S(\bb) \neq \Big(\frac{k(k+1)}{2},  \frac{k(k-1)}{2}\Big)  \\ 
			0 \quad &\text{ otherwise }
		\end{cases}
	\end{align}	
\end{definition}
\begin{remark}
	Note that in the RHS of \eqref{Def;C2;induction}, $|\ba|, |\ba'|,|\ba''| \leq L$ and thus the inductive definition is valid.
\end{remark}

\begin{remark}
 From the definition it is clear that $C^{(1)}_{\bb}, C^{(2)}_{\bb}$ depends only on $\bb$ and are explicitly computable. Moreover, Definition \ref{Def;C3C4} gives explicit formulas for $C^{(4)}_{S(\ba),S(\ba')}$ and $C^{(5)}_{S(\ba),S(\ba'),S(\ba'')}$:
	\begin{align}
		C^{(4)}_{S(\ba),S(\ba')} =& S_2(\ba)S_1(\ba')[S_1(\ba')-1] +  S_1(\ba)S_2(\ba')[S_2(\ba')-1] - |S(\ba)|\cdot S_1(\ba')S_2(\ba') \\
		C^{(5)}_{S(\ba),S(\ba'), S(\ba'')} =& S_2(\ba)S_2(\ba')S_1(\ba'')[S_1(\ba'')-1] + S_1(\ba)S_1(\ba')S_2(\ba'')[S_2(\ba'')-1] \\
		&- [S_1(\ba)S_2(\ba') + S_2(\ba)S_1(\ba')]\cdot S_1(\ba'')S_2(\ba'')\nonumber
	\end{align} 
\end{remark}

Now we are able to state the main decomposition theorem for this section:
\begin{theorem}\label{Thm;11}
We take constants $C^{(1)}, C^{(2)}$ from Definition \ref{Def;C1C2}. For all integers $i,j\geq 0$ with $i+j\leq N$,
\begin{align}\label{L2+L3;expression;1}
	(L_2P + L_3P)_{i,j} =  \sum_{\substack{S(\ba) = (i,j) \\ |\ba|\geq 2}} C^{(2)}_{\ba}\cdot z\Bz \prod_{k\geq2} \Bigg(\frac{P_{\frac{k(k-1)}{2} ,\frac{k(k+1)}{2}}}{z\Bz}\Bigg)^{\ba^{-}_k}\Bigg(\frac{P_{\frac{k(k+1)}{2} ,\frac{k(k-1)}{2}}}{z\Bz}\Bigg)^{\ba^{+}_k}
\end{align}
Moreover, in the case that $(i,j) \neq (\frac{k(k-1)}{2} , \frac{k(k+1)}{2}) $ and $(i,j)\neq(\frac{k(k+1)}{2} , \frac{k(k-1)}{2})$, for any $k\geq 2$, 
\begin{align}\label{Pij;expression;2}
	P_{i,j} =  \sum_{\substack{S(\ba) = (i,j) \\ |\ba|\geq 2}} C^{(1)}_{\ba}\cdot z\Bz \prod_{k\geq2} \Bigg(\frac{P_{\frac{k(k-1)}{2} ,\frac{k(k+1)}{2}}}{z\Bz}\Bigg)^{\ba^{-}_k}\Bigg(\frac{P_{\frac{k(k+1)}{2} ,\frac{k(k-1)}{2}}}{z\Bz}\Bigg)^{\ba^{+}_k}
\end{align}
Finally, in the case that  $(i,j) = (\frac{k(k-1)}{2} , \frac{k(k+1)}{2}) $ or $(\frac{k(k+1)}{2} , \frac{k(k-1)}{2})$ for some $k \geq 2$,  we have the following \textbf{obstruction equation}: $(L_2P + L_3P)_{i,j} = 0$, i.e.
\begin{align}\label{L2+L3;expression;2}
	 \sum_{\substack{S(\ba) = (i,j) \\ |\ba|\geq 2}} C^{(2)}_{\ba}\cdot \prod_{k\geq2} \Bigg(\frac{P_{\frac{k(k-1)}{2} ,\frac{k(k+1)}{2}}}{z\Bz}\Bigg)^{\ba^{-}_k}\Bigg(\frac{P_{\frac{k(k+1)}{2} ,\frac{k(k-1)}{2}}}{z\Bz}\Bigg)^{\ba^{+}_k} = 0
\end{align}
\end{theorem}
\begin{remark}\label{Rem;12}
	Since $N\geq |S(\ba)|\geq 2 + 2|\ba|$, all the above multiplications and summations are finite. By \eqref{Ta-Ineq;2} we can further deduce that all $k$ appeared in the Theorem \ref{Thm;11} must satisfy $k^2\leq i+j \leq N$. 
\end{remark}
\begin{remark}
	If we allow $|\ba| = 1$ under the summation, then \eqref{Pij;expression;2} is indeed true for all $i,j$. But \eqref{Pij;expression;2} is trivial when $(i,j) = (\frac{k(k-1)}{2} , \frac{k(k+1)}{2}) $ or $(i,j) = (\frac{k(k+1)}{2} , \frac{k(k-1)}{2}) $ in view of Defintion \ref{Def;C1C2}.  
\end{remark}

\begin{corollary}\label{Cor;divisble}
 $P_{i,j} \equiv 0$ when $i=0$ or $j=0$. In particular, $P_{i,j}$  are divisible by $z\Bz$ and thus $\displaystyle \frac{P_{i,j}}{z\Bz}$ are polynomials for all $i,j\geq 0$ with $i+j\leq N$. 
\end{corollary}
\begin{proof}
	Note that $P_{\frac{k(k-1)}{2} ,\frac{k(k+1)}{2}}$ and $P_{\frac{k(k+1)}{2} ,\frac{k(k-1)}{2}}$ are divisible by $z\Bz$ when $k\geq 2$. Then the corollary follows from Theorem \ref{Thm;11}.
\end{proof}

\begin{proof}[Proof of Theorem \ref{Thm;11}]
The proof proceeds by induction on the total degree $i+j$. The key observation is that the nonlinear operators $L_2$ and $L_3$ only involve lower-order coefficients, so that the recursive structure
is triangular with respect to degree.

We will repeatedly use the identity \eqref{P-Eqn;2} throughout the proof:
	\begin{align}		
	L_1 P = L_2 P + L_3 P \nonumber 
	\end{align}
	Suppose that the lowest degree in $P$ is $d_0 \geq 3$, i.e $P = \sum_{k\geq d_0} P_k$. If $d_0 > N$, then there is nothing to prove. Therefore we may assume that $d_0 \leq N$.  By Lemma \ref{Lem;LP;2}, we find that the lowest degree of $L_2P + L_3P$ is at least $min\{2d_0-2, 3d_0 -4\} > d_0$. Thus comparing the degree $d_0 $ term in the identity \eqref{P-Eqn;2} we get
	\begin{align*}
		(L_1P)_{d_0} = (L_2P + L_3P)_{d_0} = 0
	\end{align*}
	This shows that $P_{d_0}$ is in the kernel of $L_1$. By Lemma \ref{Lem;LP;2}, this forces $d_0 = k_0^2$ for some $k_0 \geq 2$ and $P_{d_0}$ be in the form:
	\begin{align}\label{P-Eqn-12}
		P_{d_0} = P_{k_0^2} = P_{\frac{k_0(k_0-1)}{2},\frac{k_0(k_0+1)}{2}} +  P_{\frac{k_0(k_0+1)}{2},\frac{k_0(k_0-1)}{2}} 
	\end{align}
	which consists of only two parts that are conjugate to each other. To prove the Theorem for $i+j \leq d_0$, it suffices to prove that the right hand side of \eqref{L2+L3;expression;1} and \eqref{Pij;expression;2} are both 0. Indeed, combining the fact that $d_0$ is the lowest degree, \eqref{P-Eqn-12} and \eqref{Ta-Ineq;2}, we see that when $|\ba|\geq 2$:
	\begin{align}
		 |S(\ba)| \geq 2 + 2(k_0^2-2) > d_0 \geq i+j
	\end{align} 
	Thus the summation is empty and RHS of \eqref{L2+L3;expression;1} and \eqref{Pij;expression;2} are both 0. This establishes the induction base.

	Suppose that the Theorem is true for all $i+j = 1,2,...,d$ for some $d_0 \leq d\leq N-1$. Then we will work with $i+j = d+1$.
	For each $\ba = (\ba^{-}_2, \ba^{+}_2, \ba^{-}_3, \ba^{+}_3, \ba^{-}_4, \ba^{+}_4...)$ with $|S(\ba)| \leq d$ and $|\ba|\geq 1$ we define
	\begin{align}\label{Def;Pa;1}
		P^{(\ba)} :=  C^{(1)}_{\ba}\cdot z\Bz \prod_{k\geq2} \Bigg(\frac{P_{\frac{k(k-1)}{2} ,\frac{k(k+1)}{2}}}{z\Bz}\Bigg)^{\ba^{-}_k}\Bigg(\frac{P_{\frac{k(k+1)}{2} ,\frac{k(k-1)}{2}}}{z\Bz}\Bigg)^{\ba^{+}_k}
	\end{align}
	Then $P^{(\ba)}$ is a monomial of the form $cz^{S_1(\ba)}\Bz^{S_2(\ba)}$. In particular,  $P^{(\ba)} = P^{(\ba)}_{S(\ba)}$, where $P^{(\ba)}_{S(\ba)}$ is the abbreviation for $P^{(\ba)}_{S_1(\ba),S_2(\ba)}$. Moreover, 	
	\begin{align}
		\text{degree}(P^{(\ba)}) = |S(\ba)| 
	\end{align}
	
	By the induction hypothesis, we can decompose degree $\leq d$ terms of $P$ as: 
	\begin{align}\label{P-decomp;1}
		P_{\leq d} = \sum_{|S(\ba)|\leq d, \ |\ba|\geq 1} P^{(\ba)} = \sum_{|S(\ba)|\leq d, \ |\ba|\geq 1} P^{(\ba)}_{S(\ba)} 
	\end{align}
	To see this, it suffices to check for all $i'+j'\leq d$ that:
	\begin{align}\label{P-decomp;2}
		P_{i',j'} = \big(\sum_{|S(\ba)|\leq d, \ |\ba|\geq 1} P^{(\ba)}\big)_{i',j'} =  \sum_{S(\ba) = (i',j')} P^{(\ba)}
	\end{align}
	When $(i',j')\neq (\frac{k(k-1)}{2},\frac{k(k+1)}{2})$ and $(i',j')\neq (\frac{k(k+1)}{2},\frac{k(k-1)}{2})$, we note that $S(\ba) = (i',j')$ forces $|\ba|\geq 2$. Then \eqref{P-decomp;2} is true by the induction hypothesis and definition of $P^{(\ba)}$.  When $(i',j') = (\frac{k(k-1)}{2},\frac{k(k+1)}{2})$ or $(i',j') = (\frac{k(k+1)}{2},\frac{k(k-1)}{2})$, by Definition \ref{Def;C1C2} we find that $(P^{(\ba)})_{i',j'} = \begin{cases}
		0 &\text{ when } |\ba|\geq 2 \text{ and } S(\ba) = (i',j') \\
		P_{i',j'} &\text{ when } |\ba|=1 \text{ and } S(\ba) = (i',j') 	\end{cases} $.  
	Note that there is only one $\ba$ with $|\ba| = 1$ and $S(\ba) = (i',j')$, then \eqref{P-decomp;2} is true by direct computation.  Combining both cases, \eqref{P-decomp;2} holds for all $i'+j'\leq d$.

	Moving forward, we use \eqref{P-decomp;1}, Lemma \ref{Lem;LP;2} and compare the coefficient of the monomial $z^i\Bz^j$ to get:
	\begin{align}\label{Eqn;32}
		(L_2P + L_3P)_{i,j}={}&\sum_{\substack{S(\ba) + S(\ba') = (i+1,j+1)\\ |\ba|, |\ba'|\geq 1}}4C^{(4)}_{S(\ba),S(\ba')} \frac{P^{(\ba)}_{S(\ba)}P^{(\ba')}_{S(\ba')}}{z\Bz}\nonumber\\
		&+ \sum_{\substack{S(\ba) + S(\ba') + S(\ba'') = (i+2,j+2)\\ |\ba|, |\ba'|, |\ba''|\geq 1}} 4C^{(5)}_{S(\ba),S(\ba'),S(\ba'')} \frac{P^{(\ba)}_{S(\ba)}P^{(\ba')}_{S(\ba')}P^{(\ba'')}_{S(\ba'')}}{z^2\Bz^2}.
	\end{align}

	We can check that all $P^{(\ba)}_{S(\ba)}, P^{(\ba')}_{S(\ba')}, P^{(\ba'')}_{S(\ba'')}$ appeared on the RHS of \eqref{Eqn;32} has degree at most $d-1$, thus is valid (since we only defined $P^{(\ba)}$ when degree is at most $d$). Indeed, by \eqref{Ta-Ineq;2} we have $|S(\ba)|\geq 4$ when $|\ba|\geq 1$. Then in \eqref{Eqn;32} each degree $|S(\ba)|, |S(\ba')|, |S(\ba'')|$ must be at most $i+j-2 = d-1$. 
	
	By the definition of $S$ we can directly check the following facts:
	\begin{align}
		S(\ba) + S(\ba') =& S(\ba + \ba') + (1,1)\\
		S(\ba) + S(\ba') +  S(\ba'') =& S(\ba + \ba' + \ba'') + (2,2)
	\end{align}
	\begin{align}
		\frac{P^{(\ba)}P^{(\ba')} }{z\Bz} =  C^{(1)}_{\ba}C^{(1)}_{\ba'}\cdot z\Bz \prod_{k\geq2} \Bigg(\frac{P_{\frac{k(k-1)}{2} ,\frac{k(k+1)}{2}}}{z\Bz}\Bigg)^{(\ba+\ba')^{-}_k}\Bigg(\frac{P_{\frac{k(k+1)}{2} ,\frac{k(k-1)}{2}}}{z\Bz}\Bigg)^{(\ba+\ba')^{+}_k}
	\end{align}
	\begin{align}
		\frac{P^{(\ba)}P^{(\ba')} P^{(\ba'')} }{z^2\Bz^2} =  C^{(1)}_{\ba}C^{(1)}_{\ba'}C^{(1)}_{\ba''}\cdot z\Bz \prod_{k\geq2} \Bigg(\frac{P_{\frac{k(k-1)}{2} ,\frac{k(k+1)}{2}}}{z\Bz}\Bigg)^{(\ba+\ba'+\ba'')^{-}_k}\Bigg(\frac{P_{\frac{k(k+1)}{2} ,\frac{k(k-1)}{2}}}{z\Bz}\Bigg)^{(\ba+\ba'+\ba'')^{+}_k}
	\end{align}
	
	Thus \eqref{Eqn;32} can be rewritten as
	\begin{align}
		&(L_2P + L_3P)_{i,j} =  \sum_{\substack{\ba,\ba'\\S(\ba+\ba') = (i,j)\\ |\ba|, |\ba'|\geq 1}}4 C^{(4)}_{S(\ba),S(\ba')} C^{(1)}_{\ba}C^{(1)}_{\ba'}\cdot z\Bz \prod_{k\geq2} \Bigg(\frac{P_{\frac{k(k-1)}{2} ,\frac{k(k+1)}{2}}}{z\Bz}\Bigg)^{(\ba+\ba')^{-}_k}\Bigg(\frac{P_{\frac{k(k+1)}{2} ,\frac{k(k-1)}{2}}}{z\Bz}\Bigg)^{(\ba+\ba')^{+}_k}\nonumber\\
		+&\sum_{\substack{\ba,\ba',\ba''\\S(\ba+\ba'+\ba'') = (i,j)\\ |\ba|, |\ba'|, |\ba''|\geq 1}} 4 C^{(5)}_{S(\ba),S(\ba'),S(\ba'')}  C^{(1)}_{\ba}C^{(1)}_{\ba'}C^{(1)}_{\ba''}\cdot z\Bz \prod_{k\geq2} \Bigg(\frac{P_{\frac{k(k-1)}{2} ,\frac{k(k+1)}{2}}}{z\Bz}\Bigg)^{(\ba+\ba'+\ba'')^{-}_k}\Bigg(\frac{P_{\frac{k(k+1)}{2} ,\frac{k(k-1)}{2}}}{z\Bz}\Bigg)^{(\ba+\ba'+\ba'')^{+}_k}
	\end{align}
	In the first line, $\ba+\ba'$ appeared both under summation and in the exponenet. Similarly for the second line. Therefore, rearranging the summation and using Definition \ref{Def;C1C2} we get the simplification:
	\begin{align}\label{L2+L3;expression;3}
		(L_2P + L_3P)_{i,j} =  \sum_{\substack{S(\bb) = (i,j)\\|\bb|\geq 2}}C^{(2)}_{\bb}\cdot z\Bz \prod_{k\geq2} \Bigg(\frac{P_{\frac{k(k-1)}{2} ,\frac{k(k+1)}{2}}}{z\Bz}\Bigg)^{\bb^{-}_k}\Bigg(\frac{P_{\frac{k(k+1)}{2} ,\frac{k(k-1)}{2}}}{z\Bz}\Bigg)^{\bb^{+}_k}
	\end{align}
	By definition $C^{(2)}_{\bb}$ depends only on $\bb$, thus \eqref{L2+L3;expression;1} is proved.

	To prove \eqref{Pij;expression;2}, we use Lemma \ref{Lem;LP;2} to write
	\begin{align}\label{L1;expression;1}
		(L_1P)_{i,j} = [i+j - (i-j)^2]P_{i,j}
	\end{align} 
	Since  $i+j\leq N$, \eqref{P-Eqn;2} still holds, namely: $(L_1P)_{i,j} = (L_2P+L_3P)_{i,j}$. Then \eqref{L2+L3;expression;1}and \eqref{L1;expression;1} imply
	\begin{align}
		P_{i,j} = \frac{1}{i+j - (i-j)^2}\sum_{\substack{S(\bb) =  (i,j)}} C^{(2)}_{\bb}\cdot z\Bz \prod_{k\geq2} \Bigg(\frac{P_{\frac{k(k-1)}{2} ,\frac{k(k+1)}{2}}}{z\Bz}\Bigg)^{\bb^{-}_k}\Bigg(\frac{P_{\frac{k(k+1)}{2} ,\frac{k(k-1)}{2}}}{z\Bz}\Bigg)^{\bb^{+}_k} 
	\end{align}
	By Definition \ref{Def;C3C4} and Definition \ref{Def;C1C2},  this can be rewritten as:
	\begin{align}
		P_{i,j} =  \sum_{\substack{S(\bb) =  (i,j)}} C^{(1)}_{\bb}\cdot z\Bz \prod_{k\geq2} \Bigg(\frac{P_{\frac{k(k-1)}{2} ,\frac{k(k+1)}{2}}}{z\Bz}\Bigg)^{\bb^{-}_k}\Bigg(\frac{P_{\frac{k(k+1)}{2} ,\frac{k(k-1)}{2}}}{z\Bz}\Bigg)^{\bb^{+}_k} 
	\end{align}
	whenever $(i,j) \neq (\frac{k(k-1)}{2} , \frac{k(k+1)}{2}) $ and $(i,j)\neq(\frac{k(k+1)}{2} , \frac{k(k-1)}{2})$. 
	Thus \eqref{Pij;expression;2} is proved.

	Finally in the case that $(i,j) = (\frac{k(k-1)}{2} , \frac{k(k+1)}{2}) $ or $(\frac{k(k+1)}{2} , \frac{k(k-1)}{2})$ for some $k \geq 2$, by Lemma \ref{Lem;LP;2} again we find that
	$(L_1P)_{i,j} = 0$. This forces $(L_2P+L_3P)_{i,j} =0$, thus proving \eqref{L2+L3;expression;2}. 
	
	Since we proved all the statements when $i+j = d+1$. By the induction, the Theorem is proved for all $i+j \leq N$. 	
\end{proof}

We now prove that the obstruction equations are genuinely nontrivial. Since the obstruction polynomial contains many nonlinear interaction terms, a direct analysis of the full expression would be difficult. Instead, we isolate a distinguished monomial whose coefficient can be computed explicitly.

The key observation is that certain interaction patterns among the resonant modes are rigid: they cannot be generated in multiple ways through the recursive nonlinear interactions. Consequently, the corresponding coefficient in the obstruction equation cannot cancel with contributions from other terms. To exploit this rigidity, we consider the interaction generated by:
\begin{itemize}
    \item one resonant mode at degree $k^2$,
    \item one resonant mode at degree $(k+1)^2$,
    \item one resonant mode at degree $(k^2+k-2)^2$.
\end{itemize}
The resulting interaction contributes to the obstruction equation at degree $(k^2+k-1)^2$. For example, when $k=2$, the obstruction occurs at degree $(2^2+2-1)^2=25$. In this case, there is a distinguished monomial arising from the interaction among resonant terms of degrees $4$, $9$, and $16$:
\begin{align*}
    z^{10}\bar{z}^{15} = \frac{\big( z^3 \bar{z} \big) \cdot \big( z^3 \bar{z}^6 \big) \cdot\big( z^6 \bar{z}^{10} \big)}{z^2 \bar{z}^2}.
\end{align*}
The following theorem computes the coefficient of this monomial explicitly and proves that it is nonzero. In Appendix \ref{app:compute}, we also compute the cases $k=2,3$ explicitly for the reader's reference.

\begin{theorem}\label{Thm;12}
	Let $k\geq 2$ be a fixed integer.  Suppose that $\ba^{1},\ba^{2}, \ba^{3}$ are double multi-index, each with only \textbf{one} nonzero elements defined by: $(\ba^1)^{+}_k = 1;  \  (\ba^2)^{-}_{k+1} = 1; \  (\ba^3)^{-}_{k^2+k-2} = 1$. Let $\ba = \ba^1 + \ba^2 + \ba^3$,  then 
	\begin{align}
		C^{(2)}_{\ba} = -\frac{2\mathbf{K}^2(\mathbf{K}-1)(\mathbf{K}-2)^2(\mathbf{K}-3)}{(k^3+k^2-2k-1)(k^3+2k^2-k-1)}
	\end{align} 
	where $\mathbf{K}:= k^2+k$. In particular, $C^{(2)}_{\ba} \leq -k^4 < 0$ for all $k\geq 2$.
\end{theorem}

\begin{proof}
 Note that
\begin{align}
	S(\ba^1) =& \Big(\frac{k(k+1)}{2}, \frac{k(k-1)}{2}\Big),\quad \quad   S(\ba^2) = \Big(\frac{k(k+1)}{2}, \frac{(k+1)(k+2)}{2})\Big) \nonumber \\
	 S(\ba^3) =& \Big(\frac{(k^2+k-3)(k^2+k-2)}{2}, \frac{(k^2+k-2)(k^2+k-1)}{2}\Big) \nonumber\end{align} 
We can rewrite the above under the abbreviation $\mathbf{K} = k^2+ k$:
\begin{align}
	S(\ba^1) =& \Big(\frac{\mathbf{K}}{2}
	, \frac{\mathbf{K}}{2} -k\Big),\quad  \ S(\ba^2) = \Big(\frac{\mathbf{K}}{2} ,\frac{\mathbf{K}}{2} +k+1\Big), \quad S(\ba^3) = \Big(\frac{(\mathbf{K}-3)(\mathbf{K}-2)}{2}, \frac{(\mathbf{K}-2)(\mathbf{K}-1)}{2}\Big)
\end{align}
Using definition of $S$, we can continue our computation:
\begin{align}
	S(\ba^1 + \ba^2) =&  (\mathbf{K}-1, \mathbf{K})\\
	  S(\ba^1 + \ba^3) =&   \Big(\frac{(\mathbf{K}-2)^2}{2}, \frac{\mathbf{K}(\mathbf{K}-2)}{2}-k\Big)\\
	 S(\ba^2+\ba^3) =&   \Big(\frac{(\mathbf{K}-2)^2}{2}, \frac{\mathbf{K}(\mathbf{K}-2)}{2} + k+1\Big)
\end{align}
Using definition $\ba = \ba^1 + \ba^2 + \ba^3$, we can further compute:
\begin{align}
	S(\ba^1 + \ba^2 + \ba^3) =  \Big(\frac{(\mathbf{K}-2)(\mathbf{K}-1)}{2}, \frac{\mathbf{K}(\mathbf{K}-1)}{2}\Big)
\end{align}

Then we get an explicit formula for $C^{(2)}_{\ba}$ using \eqref{Def;C2;induction}:
\begin{align*}
	C^{(2)}_{\ba} =& 4C^{(s4)}_{S(\ba^{1} + \ba^{2}), \ S(\ba^{3})}C^{(1)}_{\ba^{1} + \ba^{2}} + 4C^{(s4)}_{S(\ba^{1} + \ba^{3}), \ S(\ba^{2})}C^{(1)}_{\ba^{1} + \ba^{3}} + 4C^{(s4)}_{S(\ba^{2} + \ba^{3}), \ S(\ba^{1})}C^{(1)}_{\ba^{2} + \ba^{3}} + 4C^{(s5)}_{S(\ba^1), S(\ba^2), S(\ba^3)} 
\end{align*}

To justify this formula, we note that $|\ba| = 3$ and thus there are only three ways to break $\ba$ into two parts and only one way to break it into three parts, each part being not identically 0.

Using Definition \ref{Def;C3C4} and Definition \ref{Def;C1C2} we have: 
\begin{align}\label{Comp;C1;2}
	C^{(1)}_{\ba^i+\ba^j} = \frac{C^{(2)}_{\ba^i+\ba^j}}{C^{(3)}_{S(\ba^i+\ba^j)}} = \frac{4C^{(s4)}_{S(\ba^i),\ S(\ba^j)}}{C^{(3)}_{S(\ba^i+\ba^j)}}
\end{align} 
Therefore, we reached the final expansion before plugging in:
\begin{align}\label{Comp;C2;1}
	C^{(2)}_{\ba}=& \frac{4C^{(s4)}_{S(\ba^{1} + \ba^{2}), \ S(\ba^{3})}\cdot 4C^{(s4)}_{S(\ba^{1}), \ S(\ba^{2})}}{C^{(3)}_{S(\ba^1+\ba^2)}} + \frac{4C^{(s4)}_{S(\ba^{1} + \ba^{3}), \ S(\ba^{2})}\cdot 4C^{(s4)}_{S(\ba^{1}), \ S(\ba^{3})}}{C^{(3)}_{S(\ba^1+\ba^3)}}+ \frac{4C^{(s4)}_{S(\ba^{2} + \ba^{3}), \ S(\ba^{1})}\cdot 4C^{(s4)}_{S(\ba^{2}), \ S(\ba^{3})}}{C^{(3)}_{S(\ba^2+\ba^3)}} \nonumber\\
	  &+4C^{(s5)}_{S(\ba^1), S(\ba^2), S(\ba^3)}  
\end{align}

\begin{remark} In order to save spaces, we chose not to repeat symmetric subscripts above. Instead, we encode the symmetrization in the constants $C^{(s4)}, C^{(s5)}$.  That's why our constants are changed from $C^{(4)}, C^{(5)}$ to $C^{(s4)}, C^{(s5)}$.  Therefore, the above form is  equivalent to \eqref{Def;C2;induction} even though it looks different at first glance.
\end{remark}

Now we compute the constant $C^{(3)}, C^{(s4)}, C^{(s5)}$ explicitly:
\begin{align}
	C^{(s4)}_{S(\ba^{1} + \ba^{2}), \ S(\ba^{3})} =& -\frac{(\mathbf{K}-1)(\mathbf{K}-2)(3\mathbf{K}-5)}{2}, \quad C^{(s4)}_{S(\ba^{1}), \ S(\ba^{2})} = \frac{(2\mathbf{K}-1)\mathbf{K}}{2},\quad C^{(3)}_{S(\ba^1+\ba^2)} = 2\mathbf{K}-2 \nonumber
\end{align}
\begin{align}
	C^{(s4)}_{S(\ba^{1} + \ba^{3}), \ S(\ba^{2})} =& -\frac{(k+1)(k^6+2k^5-k^4-2k^3 +k^2-2k-4)}{2} \nonumber \\
	C^{(s4)}_{S(\ba^{1}), \ S(\ba^{3})} =& \frac{k(k-1)(k+2)(k^4+2k^3-2k^2-4k+2)}{2} \nonumber\\
	C^{(3)}_{S(\ba^1+\ba^3)} =& 2k^3+2k^2-4k-2 \nonumber
\end{align}
\begin{align}
	C^{(s4)}_{S(\ba^{2} + \ba^{3}), \ S(\ba^{1})} =& \frac{k(k^6+4k^5+4k^4-2k^3-4k^2+2k-1)}{2} \nonumber \\
	C^{(s4)}_{S(\ba^{2}), \ S(\ba^{3})} =& -\frac{(k-1)(k+1)(k+2)(k^4 + 2k^3 -2k^2-2k+3)}{2} \nonumber\\
	C^{(3)}_{S(\ba^2+\ba^3)} =& -2k^3-4k^2+2k+2 \nonumber
\end{align}
\begin{align*}
	C^{(s5)}_{S(\ba^1), S(\ba^2), S(\ba^3)} = -\frac{k(k-1)(k+1)(k+2)(k^6+3k^5-6k^4-17k^3+15k^2+24k-10)}{2}
\end{align*}

Plugging all above into the formula \eqref{Comp;C2;1} we get:
\begin{align}\label{Comp;C2;3}
	C^{(2)}_{\ba}=& -\frac{2\mathbf{K}^2(\mathbf{K}-1)(\mathbf{K}-2)^2(\mathbf{K}-3)}{(k^3+k^2-2k-1)(k^3+2k^2-k-1)}
\end{align}

To estimate $C^{(2)}_{\ba}$ when $k\geq 2$, we note that  $(k^3+k^2-2k-1)(k^3+2k^2-k-1) < (k^3+k^2)(k^3+2k^2+k) = \mathbf{K}^3$ and $2(\mathbf{K}-3)\geq \mathbf{K}$ . Then $C^{(2)}_{\ba} \leq -(\mathbf{K}-1)(\mathbf{K}-2) < -k^4$. The assertion then follows.
\end{proof}

\section{Parabolic obstruction and nonlinear asymptotics}\label{Sec;parabolic-obstruction}

In the previous section, we derived nonlinear obstruction equations for the Taylor expansion of the arrival time function near a spherical singularity. Those obstruction equations arise from the failure of solvability of the elliptic recursive system at the critical degrees.

The goal of the present section is to transfer these obstruction equations to the parabolic setting. More precisely, using the elliptic-parabolic correspondence established in Section \ref{Sec;elliptic-to-parabolic}, we derive recursive equations for the higher-order asymptotic coefficients of the rescaled mean curvature flow. These equations encode the nonlinear interactions among the exponentially decaying modes of the flow.

A key point is that the obstruction equations obtained in Section \ref{Sec;elliptic} can now be interpreted dynamically: certain combinations of asymptotic modes are incompatible with the existence of a smooth arrival time expansion. Equivalently, the obstruction equations impose nonlinear algebraic constraints on the higher-order asymptotics of the rescaled flow.

In the next section \ref{Sec;prescribe-higher-asymptotic}, we will construct rescaled mean curvature flows with prescribed higher-order asymptotics. By choosing the asymptotic coefficients to violate the obstruction equations derived here, we will obtain solutions whose associated arrival time functions fail to possess the corresponding higher-order Taylor expansions.

\subsection{Setup and definitions}

We retain the setup and notation of the preceding sections.

Consider a convex MCF $\{M_t\}_{t\in[0,T_0)}$ with a spherical singularity $(x_0,T_0)$ and consider the corresponding arrival time function $U\in C^{N}$ for some $N\geq 2$. Assume $x_0 = 0$ and take $P(x)$ from Section \ref{Sec;elliptic}, which is divisible by $z\Bz$ by Corollary \ref{Cor;divisble}.  Then, we consider the corresponding RMCF centered at $(x_0,T_0)$ with graph function $v$ over $\mathbb{S}^n(\sqrt{2n})$.  Then we can use Theorem \ref{Taylor-to-asymptotic} to obtain functions $f_0,...,f_{N-2}$ defined on $\mathbb{S}^n(\sqrt{2n})$.  After that, we will work solely in $\RR^2$, which corresponds to $n=1$.  Consequently,  all the constants and polynomials in this section are dimension independent. Moreover, we will adapt the complex coordinate $z,\Bz$ from Section \ref{Sec;elliptic}.  

\noindent\textbf{Caution:}
 $P$ in this section (and in Section \ref{Sec;elliptic}) do not have degree 2 term, which is different from those appeared in the Section \ref{Sec;elliptic-to-parabolic}. For all other degrees, one should expect that $P$ matches Section \ref{Sec;elliptic-to-parabolic}. To minimize confusion, we avoid writing $P_2$ in this section.

The goal of this section is to obtain obstruction equations for $f$. This will require us to break $f_m$ into finer pieces, as we did for $P_k$. To this end, we first need a Lemma to clarify the structure of $f_m$:
	
\begin{lemma}\label{fm;Lemma;3}
	For each $0\leq m \leq N-2$, there exists a unique degree $m$ homogeneous polynomial $\tilde{f}_m$ such that $f_m$ is the restriction of $\tilde{f}_m$  on $\mathbb{S}^1(\sqrt{2})$. In fact, $\tilde{f}_m$ is given by the RHS of \eqref{fm;formula;4}.
\end{lemma}
\begin{proof}
	By Corollary \ref{Cor;divisble},  $P_k$ is divisible by $z\Bz$ for all $k \geq 0$. Then the RHS of \eqref{fm;formula;4} is indeed a degree $m$ homogeneous polynomial. Since the restriction of two different homogeneous degree $m$ polynomials on sphere must be different, the uniqueness follows.
\end{proof}

\begin{remark}
	Uniqueness fails if we don't require $\tilde{f}_m$ to be homogeneous degree $m$. For instance, one can multiply $\tilde{f}_m$ by $\frac{z\Bz  }{\sqrt{2n}}$. 
\end{remark}
\begin{remark}
	This Lemma is indeed true in any dimension, not just in $\RR^2$. But we don't need it in our paper.
\end{remark}

From this Lemma, the following definition makes senses:
\begin{definition}\label{Def;fij}
	For all $0\leq m\leq N-2$, let $f_m$ be the restriction of the degree $m$ homogeneous polynomial $\tilde{f}_m$ on $\mathbb{S}^1(\sqrt{2})$ as in Lemma \ref{fm;Lemma;3}. Then for all $i,j\geq 0$ with $i+j\leq N-2$,   we define $\tilde{f}_{i,j}$ to be the sum of all monomials  of the form $cz^i\Bz^j$ in $\tilde{f}_{i+j}$. Define $f_{i,j}$ to be the restriction of $\tilde{f}_{i,j}$ onto $\mathbb{S}^1(\sqrt{2})$. In particular, $\displaystyle f_m = \sum_{i+j = m}f_{i,j}$. We then identify $f_m$ and $f_{i,j}$ as degree $m$ homogeneous polynomials on $\mathbb{S}^1(\sqrt{2})$. 
\end{definition}

\begin{definition}\label{Def;S'}
For any double multi-index, we define
	\begin{align}\label{Def;S';2}
	S'(\ba) = (S'_1(\ba),S'_2(\ba)) =   \sum_{k\geq2} \ba^{-}_k\Big(\frac{k(k-1)}{2}-1 ,\  \frac{k(k+1)}{2}-1\Big)  + \ba^{+}_k\Big(\frac{k(k+1)}{2}-1 ,\  \frac{k(k-1)}{2}-1\Big)
\end{align}
Note that $S'(\ba) = S(\ba) - (1,1)$.
\end{definition}

For the sake of simplicity, we adapt the following abbreviations:
\begin{definition}\label{Def;bfp;1}
For all integer $k\geq 2$ with $k^2\leq N$ we define
\begin{align}
	\mathbf{p}^{-}_k:= \frac{P_{\frac{k(k-1)}{2} ,\frac{k(k+1)}{2}}}{z\Bz},\quad \mathbf{p}^{+}_k:= \frac{P_{\frac{k(k+1)}{2} ,\frac{k(k-1)}{2}}}{z\Bz}
\end{align}
\begin{align}
	\mathbf{f}^{-}_k:= f_{\frac{k^2-k-2}{2} ,\frac{k^2+k-2}{2}},\quad \mathbf{f}^{+}_k:= f_{\frac{k^2+k-2}{2} ,\frac{k^2-k-2}{2}}  
\end{align}
For integer $k$ with $k^2 > N$, we simply set $\bp^{-}_k = \bp^{+}_k = \mathbf{f}^{-}_k =\mathbf{f}^{+}_k = 0$. 
Set $\bp^{*} = (\bp^{-}_2,\bp^{+}_2, \bp^{-}_3,\bp^{+}_3, ...)$ and $\mathbf{f}^{*} = (\mathbf{f}^{-}_2,\mathbf{f}^{+}_2, \mathbf{f}^{-}_3,\mathbf{f}^{+}_3, ...)$
\end{definition}

From the definition, it is clear that each $\bp^{\pm}_k$ or $\mathbf{f}^{\pm}_{k}$ is an element of $\dC[z,\Bz]$.  We will see later in this section that the coefficient of any term involving $\bp^{\pm}_k$ or $\mathbf{f}^{\pm}_{k}$ with $k^2 > N$ must be 0, so they are irrelevant. 
We still define infinitely many $\bp^{\pm}_k$ or $\mathbf{f}^{\pm}_k$ here so that we can keep the notation $\bp^{*}$ and $\mathbf{f}^{*}$ independent of $N$.

\begin{definition}
Let $\bx^{*}$ be the abbreviation the collection of double variables $(\bx_2^{-}, \bx_2^{+}, \bx_3^{-}, \bx_3^{+}, ...)$. Define $\dC[\bx^{*}]$ to be the polynomial ring over $\dC$ of double variables $(\bx_2^{-}, \bx_2^{+}, \bx_3^{-}, \bx_3^{+}, ...)$. For a double multi-index $\ba$ with finitely many nonzeros, we abbreviate 
\begin{align}
	(\bx^{*})^\ba := \prod_{k\geq 2} (\bx_k^{-})^{\ba^{-}_k}(\bx_k^{+})^{\ba^{+}_k} 
\end{align}
\end{definition}

Note that $\bx^{*}$ contains infinitely many variables starting from lower index 2, but each $\cQ\in \dC[\bx^{*}]$ is only allowed to have finitely many nonzero terms. 

Suppose that $\cQ\in \dC[\bx^{*}]$. If we write $\cQ(\bx^{*})$ or $\cQ$, we consider it as a polynomial with variable $\bx^{*}$. However, if we take $\bp^{*}$ and $\mathbf{f}^{*}$ from Definition \ref{Def;bfp;1}, then $\cQ(\bp^{*})$ and  $\cQ(\mathbf{f}^{*})$ are considered as an element of $\dC[z,\Bz]$ (a polynomial of $z,\Bz$) by substituting $\bp_k^{\pm}$ or $\mathbf{f}^{\pm}_k$ with the corresponding element in $\dC[z,\Bz]$, specified by Definition \ref{Def;bfp;1}. 

If a polynomial $\cQ\in \dC[\bx^{*}]$ is proved to depend only on part of its variable, we might drop the dependence of remaining variables. For example, if we found that $\cQ(\bx^{*})$ depends only on $\bx^{-}_2,\bx^{+}_2,...,\bx^{-}_k,\bx^{+}_k$, then we might write  $\cQ(\bx^{*}) = \cQ(\bx^{-}_2,\bx^{+}_2,...,\bx^{-}_k,\bx^{+}_k)$.

\begin{definition}\label{Def;S';3}
	For any monomial $\cR \in \dC[\bx^{*}]$ in the form of
\begin{align}
	\cR(\bx) = c(\bx^{*})^{\ba} = c\prod_{k\geq 2} (\bx_k^{-})^{\ba^{-}_k}(\bx_k^{+})^{\ba^{+}_k} 
\end{align}	
where $c\in \dC$ and $\ba$ is a double multi-index with finitely many nonzeros, define the $S'$-degree of $\cR$ to be \textbf{a pair of numbers} given by: 
\begin{align}
	deg_{S'}(\cR) := S'(\ba) 
\end{align}
define the $|S'|$-degree of $\cR$ to be a number given by 
\begin{align}
	deg_{|S'|}\cR := |S'(\ba)|
\end{align}
For any polynomial $\cQ\in \dC[\bx^{*}]$, we define the $|S'|$-degree of $\cQ$ as the highest $|S'|$-degree among all monomials in $\cQ$. 
We define $deg_{S'}(c) := (0,0)$ and $deg_{|S'|}(c) := 0 $ for constant $c \in \dC$.
\end{definition}

\begin{definition}\label{Def;S';4}
	A polynomial $\cQ$ is said to be $S'$-homogenous, if all of the monomials of $\cQ$ has the same $S'$-degree (as a pair of numbers), and we define the homogeneous $S'$-degree of $\cQ$ to be the $S'$-degree of either of its monomial, written as $deg^h_{S'}(\cQ)$.\\
A polynomial $\cQ$ is said to be $|S'|$-homogenous, if all of the monomials of $\cQ$ has the same $|S'|$-degree, and we define the homogeneous $|S'|$-degree of $\cQ$ to be the $|S'|$-degree of either of its monomial, written as $deg^h_{|S'|}(\cQ)$.\\
We will treat the homogeneous $S'$-degree and $|S'|$-degree of 0 as arbitrary number. That is, any equality of the form $deg^h_{S'}0 = \cdot$ or $deg^h_{|S'|} 0 = \cdot$ is always considered to be true. 
\end{definition}
From the definition, we see that an $S'$-homogeneous polynomial of $S'$-degree $(i,j)$ is automatically $|S'|$-homogeneous with $|S'|$-degree $i+j$, but not vice versa.  
We will use this fact frequently in the following rest of the section.

Next, we define how to extract $S'$ or $|S'|$-homogeneous terms out of a general polynomial in $\dC[\bx^{*}]$: 

\begin{definition}\label{Def;extraction;1}
	For any polynomial $\cQ \in \dC[\bx^{*}]$, define $[\cQ]_{i,j}\in \dC[\bx^{*}]$ to be the sum of all monomials in $\cQ$ with $S'$-degree equal to $(i,j)$.
	Define $[\cQ]_{k}\in \dC[\bx^{*}]$ to be the sum of all monomials in $\cQ$ with $|S'|$-degree equal to $k$. In particular, $[\cQ]_{i,j}$ is $S'$-homogeneous with $S'$-degree $(i,j)$ and $[\cQ]_{k}$ is $|S'|$-homogeneous with $|S'|$-degree equal to $k$.
\end{definition}

Since there is no canonical way to order pairs of numbers, we don't define $S'$-degree for non-$S'$-homogeneous polynomial in $\dC[\bx^{*}]$. 

Now with $S'$-degree and $|S'|$-degree defined, we will still frequently use the usual degree for a polynomial as well. The term "degree", when used without qualifiers $S'$ or $|S'|$ in front, refers to the usual polynomial degree.

\begin{lemma}\label{Lem;extraction;2}
Suppose that $\cQ\in \dC[\bx^{*}]$. Take $\bp^{*}, \mathbf{f}^{*}$ from Definition \ref{Def;bfp;1}. Then $[\cQ]_{i,j}(\bp^{*})$ and $[\cQ]_{i,j}(\mathbf{f}^{*})$ are both monomials of $\dC[z,\Bz]$ in the form $cz^{i}\Bz^j$  for any $i,j\geq 0$.
\end{lemma}

\begin{proof}
	Suppose that $\cQ = \sum_{|S'(\ba)|<k} C_{\ba}\cdot(\bx^{*})^{\ba}$ for some $k<\infty$.
	Then $[\cQ]_{i,j}(\bx^{*}) = \sum_{S'(\ba) = (i,j)} C_{\ba}\cdot(\bx^{*})^{\ba}$. If we plug in $\bp^{*}$, then 
	\begin{align}
		[\cQ]_{i,j}(\bp^{*}) = \sum_{S'(\ba) = (i,j)} C_{\ba}\cdot(\bp^{*})^{\ba} = \sum_{S'(\ba) = (i,j)} C_{\ba}\cdot \prod_{l\geq 2}(\bp^{-}_l)^{\ba^{-}_l}(\bp^{+}_l)^{\ba^{+}_l}
	\end{align}  
	It suffices to show that for each $\ba$ with $S'(\ba) =  (S'_1(\ba), S'_2(\ba)) = (i,j)$, the term $\prod_{l\geq 2}(\bp^{-}_l)^{\ba^{-}_l}(\bp^{+}_l)^{\ba^{+}_l}$ must be in the form of $Cz^i\Bz^j$. 
	We use $C$ to represent constants that might change line by line.
	By Definition \ref{Def;bfp;1} we know that $(\bp^{-}_j)^{\ba^{-}_j} = Cz^{\ba^{-}_j\cdot\frac{j(j-1)-2}{2}}\Bz^{\ba^{-}_j\cdot{\frac{j(j+1)-2}{2}}}$ and $(\bp^{+}_j)^{\ba^{+}_j} = Cz^{\ba^{+}_j\cdot\frac{j(j+1)-2}{2}}\Bz^{\ba^{+}_j\cdot{\frac{j(j-1)-2}{2}}}$. Therefore, 
	\begin{align}
		\prod_{l\geq 2}(\bp^{-}_l)^{\ba^{-}_l}(\bp^{+}_l)^{\ba^{+}_l} =& C \cdot z^{\sum_{l\geq 2}\ba^{-}_l\cdot \frac{l(l-1)-2}{2} + \ba^{+}_l\cdot \frac{l(l+1)-2}{2}} \Bz^{\sum_{l\geq 2}\ba^{-}_l\cdot \frac{l(l+1)-2}{2} + \ba^{+}_l\cdot \frac{l(l-1)-2}{2}} \nonumber\\
		=& C\cdot z^{S'_1(\ba)}\Bz^{S'_2(\ba)} \nonumber\\
		=& Cz^i\Bz^j \nonumber
	\end{align}
	The assertion then follows. 
	
\end{proof}

Lemma \ref{Lem;extraction;2} shows that the $S'$-degree represents the $z,\Bz$ degree after plugging $z,\Bz$ polynomial such as $\bp^{*}, \mathbf{f}^{*}$ into $\cQ$. Similar for $|S'|$-degree. However, we want to keep the polynomial $\cQ$ an independent algebraic object without mentioning $\bp^{*}$ or $\mathbf{f}^{*}$, hence we prefer using $S'$ or $|S'|$ degree instead of  $z,\Bz$ degree for most of the time.

 In the next Lemma, we discuss how terms or variables are absent:

\begin{lemma}\label{Lem;Sdeg;1}
Let us assume that $\cQ \in \dC[\bx^{*}]$. The following facts are independent of each other:
\begin{enumerate}
	\item If $\cQ$ is $S'$-homogeneous with $S'$-degree $=(i,j)$. Further assume that  $(i,j)\neq \big(\frac{k^2-k-2}{2},\frac{k^2+k-2}{2}\big)$ and $(i,j)\neq \big(\frac{k^2+k-2}{2},\frac{k^2-k-2}{2}\big)$ for any $k\geq 2$. Then  there is no degree 1 term in $\cQ$.
	\item  If $\cQ$ is $|S'|$-homogeneous with $|S'|$-degree $< k^2-2$ for some $k \geq 2$, then $\cQ(\bx^{*}) = \cQ(\bx^{-}_2,\bx^{+}_2,...,\bx^{-}_{k-1}, \bx^{+}_{k-1})$. In other words, $\cQ$ is independent of $\bx^{\pm}_{l}$ for all $l \geq k$.
	\item If $\cQ$ is $|S'|$-homogeneous with $|S'|$-degree  $\leq k^2-2$ for some $k \geq 2$ and with no degree 1 term, then $\cQ(\bx^{*}) = \cQ(\bx^{-}_2,\bx^{+}_2,...,\bx^{-}_{k-1}, \bx^{+}_{k-1})$. In other words, $\cQ$ is independent of $\bx^{\pm}_{l}$ for all $l \geq k$.
\end{enumerate}
\end{lemma}
\begin{proof}
\begin{enumerate}
	\item If there is a degree 1 term in $\cQ$, it must be $c\bx^{-}_k$ or $c\bx^{+}_k$ for some $k\geq 2$, the $S'$-degree must then be either $\big(\frac{k^2-k-2}{2},\frac{k^2+k-2}{2}\big)$ or $\big(\frac{k^2+k-2}{2},\frac{k^2-k-2}{2}\big)$ , a contradiction.
	\item Suppose that $\cQ$ contains monomial $c\cdot (\bx^{*})^{\ba}$, then its $|S'|$-degree 
	 is equal to $\sum_{l\geq 2} (l^2-2)(\ba^{-}_l + \ba^{+}_l)$, which by assumption is $<k^2-2$. We then deduce that $\ba^{-}_l = \ba^{+}_l = 0$ for all $l \geq k$,  meaning that this monomial is independent of  $\bx^{\pm}_{l}$ for all $l \geq k$. Since this is true for all monomials in $\cQ$, the assertion follows.
	\item  Similar argument. If $\ba^{-}_k \neq 0$ or $\ba^{+}_k \neq 0$, then inequality $\sum_{l\geq 2} (l^2-2)(\ba^{-}_l + \ba^{+}_l) \leq k^2-2$ implies that $|\ba| = 1$. This contradicts the assumption of no degree 1 term. Thus any monomial $c\cdot (\bx^{*})^{\ba}$ in $\cQ$ must have $\ba^{-}_l = \ba^{+}_l = 0$ for all $l \geq k$,  meaning that it is independent of  $\bx^{\pm}_{l}$ for all $l \geq k$. The assertion  then follows.
\end{enumerate}
\end{proof}

We next show that algebraic property of $S'$-degree and $|S'|$-degree are identical to the usual degree:
\begin{lemma}\label{Lem;Sdeg;2}
Let us assume that $\cQ, \cQ' \in \dC[\bx^{*}]$. The following facts are independent of each other:
\begin{enumerate}
	\item If $\cQ, \cQ'$ are both $S'$-homogenous, then so is their product. Moreover
		\begin{align}
			deg^h_{S'}(\cQ \cQ') = deg^h_{S'}(\cQ ) + deg^h_{S'}(\cQ')  
		\end{align}
	\item If $\cQ, \cQ'$ are both $|S'|$-homogenous, then so is their product. Moreover
		\begin{align}
			deg^h_{|S'|}(\cQ \cQ') = deg^h_{|S'|}(\cQ ) + deg^h_{|S'|}(\cQ')  
		\end{align}
\end{enumerate}
\end{lemma}
\begin{proof}
	It suffices to check for the product of monomials. For $\cQ = c(\bx^{*})^{\ba}$ and $\cQ' = c'(\bx^{*})^{\ba'}$, their product is $cc' (\bx^{*})^{\ba + \ba'}$, whose $S'$-degree is $deg^h_{S'}(\cQ\cQ') = S'(\ba + \ba') = S'(\ba) + S'(\ba') = deg^h_{S'}(\cQ ) + deg^h_{S'}(\cQ')$. Similarly we have $deg^h_{|S'|}(\cQ\cQ') = |S'(\ba + \ba')| = |S'(\ba)| + |S'(\ba')| = deg^h_{|S'|}(\cQ ) + deg^h_{|S'|}(\cQ')$
\end{proof}

We then record a lemma about how to compute the $S'$-degree for polynomial composition:
\begin{lemma}\label{Lem;poly-comp}
\begin{enumerate}
	\item Given an $S'$-homogeneous polynomial $\cQ(\bx^{*}) = \cQ(\bx^{-}_2,\bx^{+}_2, ...,\bx^{-}_k,\bx^{+}_k) \in \dC[\bx^{*}]$ depending only on $\bx^{-}_2,\bx^{+}_2, ...,\bx^{-}_k,\bx^{+}_k$. Suppose that $\cQ^{-}_{l}, \cQ^{+}_{l}\in \dC[\bx^{*}]$ are $S'$-homogeneous with $deg^h_{S'}(\cQ^{-}_{l}) = (\frac{l(l-1)-2}{2}, \frac{l(l+1)-2}{2})$ and $deg^h_{S'}(\cQ^{+}_{l}) = (\frac{l(l+1)-2}{2}, \frac{l(l-1)-2}{2})$ for all $2\leq l \leq k$, then 
	\begin{align}
	\tilde{\cQ}(\bx^{*}):= \cQ\Big(\cQ^{-}_{2}(\bx^{*}), \cQ^{+}_{2}(\bx^{*}), ..., \cQ^{-}_{k}(\bx^{*}), \cQ^{+}_{k}(\bx^{*})\Big)  
	\end{align}
	is an $S'$-homogeneous polynomial with $deg^h_{S'} (\tilde{\cQ}) = deg^h_{S'} (\cQ)$ 
	\item Given an $|S'|$-homogeneous polynomial $\cQ(\bx^{*}) = \cQ(\bx^{-}_2,\bx^{+}_2, ...,\bx^{-}_k,\bx^{+}_k) \in \dC[\bx^{*}]$ depending only on $\bx^{-}_2,\bx^{+}_2, ...,\bx^{-}_k,\bx^{+}_k$. Suppose that $\cQ^{-}_{l}, \cQ^{+}_{l}\in \dC[\bx^{*}]$ are $|S'|$-homogeneous with $deg^h_{|S'|}(\cQ^{-}_{l}) = deg^h_{|S'|}(\cQ^{+}_{l}) = l^2-2$ for all $2\leq l \leq k$, then 
	\begin{align}
	\tilde{\cQ}(\bx^{*}):= \cQ\Big(\cQ^{-}_{2}(\bx^{*}), \cQ^{+}_{2}(\bx^{*}), ..., \cQ^{-}_{k}(\bx^{*}), \cQ^{+}_{k}(\bx^{*})\Big)  
	\end{align}
	is an $|S'|$-homogeneous polynomial with $deg^h_{|S'|} (\tilde{\cQ}) = deg^h_{|S'|} (\cQ)$.
	\item Given an $S'$-homogeneous polynomial $\cQ(\bx^{*}) \in  \dC[\bx^{*}]$. Suppose that $\cQ^{-}_{l}, \cQ^{+}_{l}\in \dC[\bx^{*}]$ are $S'$-homogeneous with $deg^h_{S'}(\cQ^{-}_{l}) = (\frac{l(l-1)-2}{2}, \frac{l(l+1)-2}{2})$ and $deg^h_{S'}(\cQ^{+}_{l}) = (\frac{l(l+1)-2}{2}, \frac{l(l-1)-2}{2})$ for all $l \geq 2$, then 
	\begin{align}
	\tilde{\cQ}(\bx^{*}):= \cQ\Big(\cQ^{-}_{2}(\bx^{*}), \cQ^{+}_{2}(\bx^{*}), \cQ^{-}_{3}(\bx^{*}), \cQ^{+}_{3}(\bx^{*}),....\Big)  
	\end{align}
	is an $S'$-homogeneous polynomial with $deg^h_{S'} (\tilde{\cQ}) = deg^h_{S'} (\cQ)$.
\end{enumerate}
\end{lemma}
 \begin{proof}
 All constants $C$ appeared in the proof are local and will not interact with the other part of the paper.
 \begin{enumerate}
 	\item We assume that $deg^h_{S'}(\cQ) = (i,j)$, then we can write
 	\begin{align}
 		\cQ(\bx^{-}_2,\bx^{+}_2, ...,\bx^{-}_k,\bx^{+}_k)  = \sum_{\substack{S'(\ba) = (i,j)\\ \ba^{-}_{k+1} = \ba^{+}_{k+1} = \ba^{-}_{k+2} = \ba^{+}_{k+2} = \cdots = 0}} C_{\ba}\cdot \prod_{l=  2}^{k} (\bx_l^{-})^{\ba^{-}_l}(\bx_l^{+})^{\ba^{+}_l}  
 	\end{align} 
 		Plugging in we get:
 		\begin{align}
 			\tilde{\cQ}(\bx) =& \sum_{\substack{S'(\ba) = (i,j)\\ \ba^{-}_{k+1} = \ba^{+}_{k+1} = \ba^{-}_{k+2} = \ba^{+}_{k+2} = \cdots = 0}} C_{\ba}\cdot \prod_{l= 2}^{k} (\cQ_l^{-})^{\ba^{-}_l}(\cQ_l^{+})^{\ba^{+}_l} 
 		\end{align}
 		By Definition \ref{Lem;Sdeg;2} we see that 
 		\begin{align}
 			deg^h_{S'} \Bigg[ \prod_{l= 2}^{k} (\cQ_l^{-})^{\ba^{-}_l}(\cQ_l^{+})^{\ba^{+}_l} \Bigg] =& \sum_{l= 2}^{k} \ba^{-}_l \cdot deg^h_{S'}(\cQ_l^{-}) + \ba^{+}_l \cdot deg^h_{S'}(\cQ_l^{+}) \\
 			=&\sum_{l= 2}^{k} \ba^{-}_l \cdot \Big(\frac{l(l-1)-2}{2}, \frac{l(l+1)-2}{2}\Big) + \ba^{+}_l \cdot \Big(\frac{l(l+1)-2}{2}, \frac{l(l-1)-2}{2}\Big)\\
 			=& S'(\ba) = (i,j)
 		\end{align}
 		Therefore, each summand of $\tilde{\cQ}$ has $S'$-degree $=(i,j)$. The result then follows
 		\item  We assume that $deg^h_{|S'|}(\cQ) = m$, then we can write
 	\begin{align}
 		\cQ(\bx^{-}_2,\bx^{+}_2, ...,\bx^{-}_k,\bx^{+}_k)  = \sum_{\substack{|S'(\ba)| = m\\ \ba^{-}_{k+1} = \ba^{+}_{k+1} = \ba^{-}_{k+2} = \ba^{+}_{k+2} = \cdots = 0}} C_{\ba}\cdot \prod_{l=  2}^{k} (\bx_l^{-})^{\ba^{-}_l}(\bx_l^{+})^{\ba^{+}_l}  
 	\end{align} 
 		Plugging in we get:
 		\begin{align}
 			\tilde{\cQ}(\bx) =& \sum_{\substack{|S'(\ba)| = m\\ \ba^{-}_{k+1} = \ba^{+}_{k+1} = \ba^{-}_{k+2} = \ba^{+}_{k+2} = \cdots = 0}} C_{\ba}\cdot \prod_{l= 2}^{k} (\cQ_l^{-})^{\ba^{-}_l}(\cQ_l^{+})^{\ba^{+}_l} 
 		\end{align}
 		By Definition \ref{Lem;Sdeg;2} we see that 
 		\begin{align}
 			deg^h_{|S'|} \Bigg[ \prod_{l= 2}^{k} (\cQ_l^{-})^{\ba^{-}_l}(\cQ_l^{+})^{\ba^{+}_l} \Bigg] =& \sum_{l= 2}^{k} \ba^{-}_l \cdot deg^h_{|S'|}(\cQ_l^{-}) + \ba^{+}_l \cdot deg^h_{|S'|}(\cQ_l^{+}) =\sum_{l= 2}^{k} (\ba^{-}_l + \ba^{+}_l) \cdot (l^2-2) \\
 		=& |S'(\ba)| = m
 		\end{align}
 		Therefore, each summand of $\tilde{\cQ}$ has $|S'|$-degree $=m$. The result then follows.
 		\item Any $S'$-homogeneous polynomial $\cQ$ is in the form of $\cQ(\bx^{*})= \sum_{S'(\ba) = (i,j)}C_{\ba} \cdot (\bx^{*})^{\ba}$. 
 			Since there are only finitely many $\ba$ satisfying $S'(\ba) = (i,j)$, we can find $k\geq 2$ such that $\cQ(\bx^{*}) = \cQ(\bx^{-}_2,\bx^{+}_2,\cdots, \bx^{-}_k,\bx^{+}_k)$. Then we are back to the case (1) and hence the result is true. 

 \end{enumerate}
 \end{proof}
 
\subsection{Transferring between $\bp^{*}$ and $\mathbf{f}^{*}$}

We first define some universal polynomials.  All polynomials defined here depends only on their indices, and they are defined for all large indices, including those beyond $N$.
\begin{definition}\label{Q;def;1}
\begin{itemize}
	\item Define $\cQ^{(25)}, \cQ^{(26)} \in \dC[\bx^{*}]$ in the following way:
	\begin{align}\label{Def;Q26}
	\cQ^{(26)}_{i,j}(\bx^{*}) : = \sum_{\substack{S'(\ba) = (i,j)\\|\ba| \ge 1}} C^{(1)}_{\ba}\prod_{k\geq 2} (\bx^{-}_k)^{\ba^{-}_k}(\bx^{+}_k)^{\ba^{+}_k} \in \dC[\bx^{*}],\quad \quad i, j \geq 0
	\end{align}
	\begin{align}
		\cQ^{(25)}_m := \sum_{i+j = m}\cQ^{(26)}_{i,j} \in \dC[\bx^{*}],\quad \quad m \geq 0
	\end{align}

\item Define $\cQ^{(23)}, \cQ^{(24)}, \cQ^{(27)}\in \dC[\bx^{*}]$ by
\begin{align}\label{Def;Q23}
	\cQ^{(23)}_m(\bx^{*}) : = -\sqrt{2}\cQ^{(25)}_m(\bx^{*}) + \cQ^{(21)}_m\Big(2\cQ^{(25)}_1(\bx^{*}), \cdots, 2\cQ^{(25)}_{m-1}(\bx^{*})\Big), \quad \quad m\geq 1
\end{align}
\begin{align}
	\cQ^{(24)}_{i,j} : = [\cQ^{(23)}_{i+j}]_{i,j},\quad \quad i,j \geq 0, \ i+j \geq 1
\end{align}
\begin{align}\label{Def;Q27}
	\cQ^{(27)}_m := \cQ^{(11)}_{m}\Big(\cQ^{(23)}_1(\bx^{*}), \cdots , \cQ^{(23)}_{m-1} (\bx^{*})\Big), \quad \quad m\geq 1
\end{align}

\item Define $\cQ^{(22)}_{i,j}  \in \dC[\bx^{*}]$ for all $i,j\geq 0$ as follows. If $(i,j)\neq (\frac{k^2-k-2}{2}, \frac{k^2+k-2}{2})$ and $(i,j)\neq(\frac{k^2+k-2}{2}, \frac{k^2-k-2}{2})$ for any $k\geq 2$, define
\begin{align}
	\cQ^{(22)}_{i,j}(\bx^{*}) :=  \cQ^{(26)}_{i,j}(\bx^{*})
\end{align}
Otherwise if  $(i,j) = (\frac{k^2-k-2}{2}, \frac{k^2+k-2}{2})$ or $(i,j) = (\frac{k^2+k-2}{2}, \frac{k^2-k-2}{2})$ for some $k\geq 2$,  define
\begin{align}
	\cQ^{(22)}_{i,j}(\bx^{*}) :=   [\cQ^{(27)}_{i+j}]_{i,j}(\bx^{*})  
\end{align}
\textbf{Remark:} In the case of $(i,j)\neq (\frac{k^2-k-2}{2}, \frac{k^2+k-2}{2})$ and $(i,j)\neq(\frac{k^2+k-2}{2}, \frac{k^2-k-2}{2})$ for any $k\geq 2$, one can alternatively define $\cQ^{(22)}_{i,j}(\bx^{*}) :=  -\frac{1}{\sqrt{2}}\cQ^{(24)}_{i,j}(\bx^{*}) + [\cQ^{(27)}_{i+j}]_{i,j}(\bx^{*})$.  This is equivalent to the previous definition, but we will not use it and the equivalence is not obvious at first glance.

\end{itemize}	
\end{definition}

\begin{corollary}\label{Pij;cor;3}[c.f Theorem \ref{Thm;11}] 
We can rewrite Theorem \ref{Thm;11} as 
\begin{align}
	\frac{P_{i+1,j+1}}{z\Bz} =  \cQ^{(26)}_{i,j}(\bp^{*}) \quad \quad \forall i+j\in [0, N-2].
\end{align}
\end{corollary}

\begin{lemma}\label{Q;deg;lem;2}
We analyze the degree of the polynomials we just defined:
\begin{enumerate}
	\item $\cQ^{(25)}_m$ is $|S'|$-homogeneous with $|S'|$-degree $=m$ for all $m\geq 0$ 
	\item  $\cQ^{(23)}_m, \cQ^{(27)}_m$ are $|S'|$-homogeneous with $|S'|$-degree $=m$ for all $m \geq 1$.
	\item $\cQ^{(22)}_{i,j}, \cQ^{(26)}_{i,j}$ is $S'$-homogeneous with $S'$-degree $=(i,j)$ for all $i,j\geq 0$.
	\item $\cQ^{(24)}_{i,j}$ is $S'$-homogeneous with $S'$-degree $=(i,j)$ for all $i+j\geq 1$.
	\item \label{Q;deg;5} There is no (usual) degree 1 term in $\cQ^{(22)}_{i,j}$  ($i,j \geq 0$) or $\cQ^{(27)}_m$ ($m\geq 1$). In particular, when $i+j \leq k^2-2$ ($k\geq 2$), polynomials $\cQ^{(22)}_{i,j}$ and $\cQ^{(27)}_{k^2-2}$  depend only on $\bx^{-}_2,\bx^{+}_2,...,\bx^{-}_{k-1}, \bx^{+}_{k-1}$.
\end{enumerate}
	
\end{lemma}

\begin{proof}
Recall that $\cQ^{(26)}_{i,j}(\bx^{*}) = \sum_{\substack{S'(\ba) = (i,j)\\ |\ba|\ge 1}} C^{(1)}_{\ba}\prod_{k\geq 2} (\bx^{-}_k)^{\ba^{-}_k}(\bx^{+}_k)^{\ba^{+}_k}$ and $\cQ^{(25)}_m = \sum_{i+j = m}\cQ^{(26)}_{i,j}$. 
	By  Definition \ref{Def;S';3} and  \ref{Def;S';4},  $\cQ^{(26)}_{i,j}$  is $S'$-homogeneous with $S'$ degree equal to $(i,j)$ and $\cQ^{(25)}_{m}$  is $|S'|$-homogeneous with $|S'|$ degree equal to $m$.

Recall that $\cQ^{(21)}_m(\bx_1,...,\bx_{m-1}) = \sum_{\substack{|\beta|=m \\ 1\leq \beta_1 \leq\cdots \leq \beta_p\leq m-1}} C^{(21)}_{\beta}\cdot \bx_{\beta_1}\cdots\bx_{\beta_p}$.Then we compute 
\begin{align}
	\cQ^{(21)}_m\Big(2\cQ^{(25)}_1(\bx^{*}), \cdots, 2\cQ^{(25)}_{m-1}(\bx^{*})\Big) = \sum_{\substack{|\beta|=m \\ 1\leq \beta_1 \leq\cdots \leq \beta_p\leq m-1}} 2^{p}C^{(21)}_{\beta}\cdot \cQ^{(25)}_{\beta_1}(\bx^{*})\cdots \cQ^{(25)}_{\beta_p}(\bx^{*})
\end{align}
By Definition \ref{Lem;Sdeg;2}, each monomial has $|S'|$-degree given by:
\begin{align}
 			deg^h_{|S'|} \Bigg[2^{p}C^{(21)}_{\beta}\cdot \cQ^{(25)}_{\beta_1}(\bx^{*})\cdots \cQ^{(25)}_{\beta_p}(\bx^{*})
 \Bigg] =& \sum_{l= 1}^{p} \deg^h_{|S'|}(\cQ^{(25)}_{\beta_l}) =|\beta| = m
\end{align}
From definition \ref{Q;def;1},  $\cQ^{(23)}_m$ is $|S'|$-homogeneous with $|S'|$-degree equal to $m$. 

We argue similarly for $\cQ^{(27)}$. First recall $\cQ^{(11)}_m(\bx_1,...,\bx_{m-1}) = \sum_{\substack{|\beta|=m \\ 1\leq \beta_1 \leq\cdots \leq \beta_p\leq m-1}} C^{(11)}_{\beta}\cdot \bx_{\beta_1}\cdots \bx_{\beta_p}$. Then we compute 
\begin{align}\label{Q27;2}
	 \cQ^{(27)}_m(\bx^{*}) = \cQ^{(11)}_{m}\Big(\cQ^{(23)}_1(\bx^{*}), \cdots , \cQ^{(23)}_{m-1} (\bx^{*})\Big) = \sum_{\substack{|\beta|=m \\ 1\leq \beta_1 \leq\cdots \leq \beta_p\leq m-1}} C^{(11)}_{\beta}\cdot \cQ^{(23)}_{\beta_1}(\bx^{*})\cdots \cQ^{(23)}_{\beta_p}(\bx^{*})
\end{align}
Next, by Definition \ref{Lem;Sdeg;2}, each monomial has $|S'|$-degree given by:
\begin{align}
 			deg^h_{|S'|} \Bigg[C^{(11)}_{\beta}\cdot \cQ^{(23)}_{\beta_1}(\bx^{*})\cdots \cQ^{(23)}_{\beta_p}(\bx^{*})
 \Bigg] =& \sum_{l= 1}^{p} \deg^h_{|S'|}(\cQ^{(23)}_{\beta_l}) =|\beta| = m.
\end{align}
Therefore, $\cQ^{(27)}_m$ is $|S'|$-homogeneous with $|S'|$-degree equal to $m$. 

The $S'$-degree of $\cQ^{(22)}_{i,j}$ and $\cQ^{(24)}_{i,j}$ are equal to $(i,j)$ by definition. 

Next we justify \eqref{Q;deg;5}.

From \eqref{Q27;2} we see that the length of $\beta$ in \eqref{Q27;2} must be at least 2, because $|\beta| = m$ while each $\beta_1,...,\beta_p \leq m-1$. Since $\cQ^{(23)}_{\beta_1},...,\cQ^{(23)}_{\beta_p}$ has no (usual) degree 0 term, the RHS of \eqref{Q27;2} has no (usual) degree 1 term, equivalently $\cQ^{(27)}$ has no (usual) degree 1 term.

Finally we inspect $\cQ^{(22)}_{i,j}$. If $(i,j) = (\frac{k^2-k-2}{2}, \frac{k^2+k-2}{2})$ or $(i,j) = (\frac{k^2+k-2}{2}, \frac{k^2-k-2}{2})$ for some $k\geq 2$, then $\cQ^{(22)}_{i,j} = [\cQ^{(27)}_{i+j}]_{i,j}$, hence has no (usual) degree 1 term. Otherwise, $\cQ^{(22)}_{i,j} = \cQ^{(26)}_{i,j}$. Observe that there does not exist $\ba$ with $|\ba|=1$ and $S(\ba) \neq (\frac{k(k-1)}{2}, \frac{k(k+1)}{2})$ and $S(\ba) \neq (\frac{k(k+1)}{2}, \frac{k(k-1)}{2})$. Then in the summation of \eqref{Def;Q26}, $|\ba|\geq 2$ and therefore $\cQ^{(22)}_{i,j}$ does not have (usual) degree 1 term. 
\end{proof}

\begin{definition}\label{Q;def;3}
We define $\cQ^{(12)}_{i,j}, \cQ^{(14)}_{i,j} \in\dC[\bx^{*}] $ for all $i,j\geq 0$ inductively. First we let $\cQ^{(14)}_{0,2}(\bx^{*}) := \bx^{-}_2, \ \cQ^{(14)}_{2,0} := \bx^{+}_2, \cQ^{(14)}_{1,1} :=0 $ and
$ \cQ^{(14)}_{1,1} =  \cQ^{(12)}_{0,0} =  \cQ^{(12)}_{1,0} = \cQ^{(12)}_{0,1}  = \cQ^{(12)}_{0,2} =\cQ^{(12)}_{1,1}  =  \cQ^{(12)}_{2,0} := 0$. 
Suppose that  $\cQ^{(12)}_{i,j}, \cQ^{(14)}_{i,j}$ have been defined for all $0\leq i+j \leq m$. Then for any $i,j\geq 0$ with $i+j = m+1$, we find the largest integer $k$ such that $k^2 -2 < m+1$. By the last item of Lemma \ref{Q;deg;lem;2}, it is valid to define:
\begin{align}\label{Def;Q12}
	\cQ^{(12)}_{i,j}(\bx^{*}): 	= \cQ^{(22)}_{i,j}\Bigg(\cQ^{(14)}_{\frac{2^2-2-2}{2},\frac{2^2+2-2}{2}}(\bx^{*}),\cQ^{(14)}_{\frac{2^2+2-2}{2},\frac{2^2-2-2}{2}}(\bx^{*}),\cdots ,\cQ^{(14)}_{\frac{k^2 -k -2}{2},\frac{k^2+k-2}{2}}(\bx^{*}),\cQ^{(14)}_{\frac{k^2+k-2}{2},\frac{k^2-k-2}{2}}(\bx^{*})\Bigg)
\end{align}
Lastly we define $\cQ^{(14)}_{i,j}$ as follows.  If $(i,j) = (\frac{k'^2-k'-2}{2}, \frac{k'^2+k'-2}{2})$ for some $k'\geq 2$ we define  
\begin{align}
	\cQ^{(14)}_{\frac{k'^2-k'-2}{2}, \frac{k'^2+k'-2}{2}} (\bx^{*}): = - \frac{1}{\sqrt{2}} \bx^{-}_{k'} + \cQ^{(12)}_{\frac{k'^2-k'-2}{2}, \frac{k'^2+k'-2}{2}}(\bx^{*})
\end{align}
If $(i,j) = (\frac{k'^2+k'-2}{2}, \frac{k'^2-k'-2}{2})$ for some $k'\geq 2$ we define
\begin{align}
	\cQ^{(14)}_{\frac{k'^2+k'-2}{2}, \frac{k'^2-k'-2}{2}}(\bx^{*}): =  - \frac{1}{\sqrt{2}}\bx^{+}_{k'} + \cQ^{(12)}_{\frac{k'^2+k'-2}{2}, \frac{k'^2-k'-2}{2} }(\bx^{*})
\end{align}
Otherwise, we define $\cQ^{(14)}_{i,j} := \cQ^{(12)}_{i,j}$. 
\end{definition}

\begin{lemma}\label{Q;deg;lem;3}
\begin{enumerate}
	\item  $\cQ^{(12)}_{i,j}$ and $\cQ^{(14)}_{i,j}$ are $S'$-homogeneous with $S'$-degree $=(i,j)$ for all $i,j\geq 0$
	\item There is no (usual) degree  0, 1 term in $\cQ^{(12)}_{i,j}$ for all $i,j\geq 0$.
\end{enumerate}
\end{lemma}

\begin{proof}
	Statement (1) follows from Lemma \ref{Lem;poly-comp} and induction.
	Since a (usual) degree 0 term has $S'$-degree $(0,0)$, there is no (usual) degree 0 term in $\cQ^{(12)}_{i,j}$ or  $\cQ^{(14)}_{i,j}$ unless $i=j=0$. But we have defined $\cQ^{(12)}_{0,0}=0$ and $\cQ^{(14)}_{0,0}=0$.  This rules out (usual) degree 0 term for all $i,j\geq 0$. 
	
	By Lemma \ref{Q;deg;lem;2}, $\cQ^{(22)}_{i,j}$ has neither a constant term nor a (usual) degree 1 term, whereas every $\cQ_{a,b}^{(14)}$ has no constant term. Hence every monoial in hte composition \eqref{Def;Q12} has usual degree at least two. This proves the Lemma. 
\end{proof}

\subsection{Expressing $\bp^{*}$ in terms of $\mathbf{f}^{*}$} 
\begin{proposition}\label{Prop;ptof;1}
	For all $k\geq 2$ with $k^2\leq N$:
		\begin{align}
			\mathbf{p}^{-}_k = -\frac{1}{\sqrt{2}}\mathbf{f}^{-}_k + \cQ^{(12)}_{\frac{k^2-k-2}{2},\frac{k^2+k-2}{2}}(\mathbf{f}^{*})\label{Pij;expression;13} \\
			\mathbf{p}^{+}_k = -\frac{1}{\sqrt{2}}\mathbf{f}^{+}_k + \cQ^{(12)}_{\frac{k^2+k-2}{2},\frac{k^2-k-2}{2}}(\mathbf{f}^{*})\label{Pij;expression;14}
		\end{align}
Moreover, for all $i,j \geq 0$ with $i+j\leq N-2$:
		\begin{align}
			\frac{P_{i+1,j+1}}{z\Bz} = \cQ^{(14)}_{i,j}(\mathbf{f}^{*}) \label{Pij;expression;12}
		\end{align}

\end{proposition}
 
\begin{proof}
We will break into several steps. Induction is only needed in the second step.

\noindent\textbf{Step 1:} Express $f_m$ and $f_{i,j}$ in terms of $\bp^{*}$. This step does \textbf{not} use induction.
 
By \eqref{fm;formula;4}, using $P_2=z\bar z/2$ in the planar polynomial convention,  and Definition \ref{Q;def;1}, we can write 
\begin{align}\label{fm;formula;9}
	f_m  =& -\sqrt{2}\cdot \frac{P_{m+2}}{z\Bz} + \cQ^{(21)}_{m} \Bigg(\frac{2P_{1+2}}{z\Bz},\cdots,\frac{2P_{(m-1)+2}}{z\Bz}\Bigg) \nonumber\\
	=&-\sqrt{2}\cdot \cQ^{(25)}_m(\bp^{*}) + \cQ^{(21)}_{m} \Bigg(2\cQ^{(25)}_1(\bp^{*}) ,\cdots,2\cQ^{(25)}_{m-1}(\bp^{*}) \Bigg)\nonumber\\
	=& \cQ^{(23)}_m(\bp^{*})
\end{align}
for each $1\leq m\leq N-2$. In addition, if $1\leq m\leq N-2$ and $i+j =m$, Lemma \ref{Lem;extraction;2} and Definition \ref{Q;def;1} implies: 
\begin{align}\label{fij;expression;3}
	f_{i,j} =[\cQ^{(23)}_m]_{i,j}(\bp^{*}) = \cQ^{(24)}_{i,j}(\bp^{*}) 
\end{align}

\noindent\textbf{Step 2:} Establish \eqref{Pij;expression;13} and \eqref{Pij;expression;14}  by induction.

\textbf{Induction base:}
We first establish the induction base. Set $m= 2$ in Corollary \ref{fm;Lemma;2}, using the explicit form of $\cQ^{(21)}$  we get:
\begin{align}\label{fm;formula;12}
	f_2=-\sqrt{2} \frac{P_{4}}{z\Bz} +  C^{(21)}_{(1,1)} \cdot \Big(\frac{2P_{3}}{z\Bz}\Big)^2 
\end{align}

By Theorem \ref{Thm;11}, we know that $P_3 = 0$ and $P_4 = P_{1,3} + P_{3,1}$. This means that
\begin{align}\label{fm;formula;13}
	f_2 = -\sqrt{2}\frac{P_4}{z\Bz} = -\sqrt{2}\frac{P_{1,3} + P_{3,1}}{z\Bz} 
\end{align}
Comparing the degree, we find that 
\begin{align}
	\mathbf{f}^{-}_2 = f_{0,2} = -\sqrt{2}\frac{P_{1,3}}{z\Bz} = -\sqrt{2}\bp^{-}_2, \quad \quad \mathbf{f}^{+}_2 = f_{2,0} = -\sqrt{2}\frac{P_{3,1}}{z\Bz} = -\sqrt{2}\bp^{+}_2 \quad \quad  f_{1,1} = 0,
\end{align}
By Definition \ref{Q;def;3}, $\cQ^{(12)}_{0,2} = \cQ^{(12)}_{1,1} = \cQ^{(12)}_{2,0} \equiv 0$, therefore \eqref{Pij;expression;13} and \eqref{Pij;expression;14} are true for $k=2$.  

 \textbf{Induction step:}
 Suppose that  \eqref{Pij;expression;13} and \eqref{Pij;expression;14} are true for $2,...,k$. 
 We may assume that $(k+1)^2 \leq N$, otherwise there is nothing to prove. 
By Definition \ref{Q;def;3}, an alternative form of the induction hypothesis is:
\begin{align}
	\mathbf{p}^{-}_l =  \cQ^{(14)}_{\frac{l^2-l-2}{2},\frac{l^2+l-2}{2}}(\mathbf{f}^{*})\label{Pij;expression;15} \\
	\mathbf{p}^{+}_l =  \cQ^{(14)}_{\frac{l^2+l-2}{2},\frac{l^2-l-2}{2}}(\mathbf{f}^{*})\label{Pij;expression;16} 
\end{align}
hold for all $2\leq l \leq k$.

By setting $m= (k+1)^2-2$ in \eqref{fm;formula;3} we get
 \begin{align}\label{Pij;expression;17}
 	\frac{P_{(k+1)^2}}{z\Bz} =& -\frac{1}{\sqrt{2}}f_{(k+1)^2-2} + \cQ^{(11)}_{(k+1)^2-2}(f_1,...,f_{(k+1)^2-3})
 \end{align}
Using Corollary \ref{Pij;cor;3} on the LHS and using \eqref{fm;formula;9}, \eqref{fij;expression;3}  on RHS we get:
\begin{align}
	&\bp^{-}_{k+1} + \bp^{+}_{k+1} + \sum_{\substack{i+j=(k+1)^2-2, \ i,j\geq 0\\ (i,j)\neq ({\frac{k(k+1)-2}{2}, \frac{(k+1)(k+2)-2}{2}})\\(i,j)\neq (\frac{(k+1)(k+2)-2}{2}, \frac{k(k+1)-2}{2})} } \cQ^{(26)}_{i,j}(\bp^{*}) \nonumber\\
	= -&\frac{1}{\sqrt{2}} \mathbf{f}^{-}_{k+1} - \frac{1}{\sqrt{2}} \mathbf{f}^{+}_{k+1} - \frac{1}{\sqrt{2}} \cdot \sum_{\substack{i+j=(k+1)^2-2, \ i,j\geq 0\\ (i,j)\neq ({\frac{k(k+1)-2}{2}, \frac{(k+1)(k+2)-2}{2}})\\(i,j)\neq (\frac{(k+1)(k+2)-2}{2}, \frac{k(k+1)-2}{2})} } \cQ^{(24)}_{i,j}(\bp^{*})+\cQ^{(11)}_{(k+1)^2-2}(\cQ^{(23)}_1(\bp^{*}),..., \cQ^{(23)}_{(k+1)^2-3}(\bp^{*})) \nonumber
\end{align}

Using Definition \ref{Q;def;1} and Lemma \ref{Lem;extraction;2} to extract and compare the terms in the form of $cz^{\frac{k^2+k-2}{2}}\Bz^{ \frac{k^2+3k}{2}}$:
\begin{align}
	\bp^{-}_{k+1}  = -\frac{1}{\sqrt{2}} \mathbf{f}^{-}_{k+1} +\cQ^{(22)}_{{\frac{k(k+1)-2}{2}, \frac{(k+1)(k+2)-2}{2}}}(\bp^{*}) 
\end{align}
Similarly we get:
\begin{align}
	\bp^{+}_{k+1}  = -\frac{1}{\sqrt{2}} \mathbf{f}^{+}_{k+1} +\cQ^{(22)}_{\frac{(k+1)(k+2)-2}{2}, \frac{k(k+1)-2}{2}}(\bp^{*}) 
\end{align}
Since $\frac{(k+1)(k+2)-2+k(k+1)-2}{2}=(k+1)^2-2$, Lemma \ref{Q;deg;lem;2} implies that $\cQ^{(22)}_{\frac{k^2+k-2}{2}, \frac{k^2+3k}{2}}(\bx^{*})$ and $\cQ^{(22)}_{\frac{k^2+3k}{2}, \frac{k^2+k-2}{2}}(\bx^{*})$ depend only on $\bx^{-}_2, \bx^{+}_2,..., \bx^{-}_k, \bx^{+}_k$. This allows us to use our induction hypothesis \eqref{Pij;expression;15}, \eqref{Pij;expression;16} to write:
\begin{align}
	&\bp^{-}_{k+1} = -\frac{1}{\sqrt{2}} \mathbf{f}^{-}_{k+1} \nonumber  \\
    &+\cQ^{(22)}_{{\frac{k(k+1)-2}{2}, \frac{(k+1)(k+2)-2}{2}}}\Big(\cQ^{(14)}_{\frac{2^2-2-2}{2},\frac{2^2+2-2}{2}}(\mathbf{f^{*}}), \cQ^{(14)}_{\frac{2^2+2-2}{2},\frac{2^2-2-2}{2}}(\mathbf{f^{*}}),...,\cQ^{(14)}_{\frac{k^2-k-2}{2},\frac{k^2+k-2}{2}}(\mathbf{f^{*}}), \cQ^{(14)}_{\frac{k^2+k-2}{2},\frac{k^2-k-2}{2}}(\mathbf{f^{*}})\Big) \nonumber 
\end{align}
Then Definition \ref{Q;def;3} gives
\begin{align}
	\bp^{-}_{k+1}=  -\frac{1}{\sqrt{2}} \mathbf{f}^{-}_{k+1}  +\cQ^{(12)}_{{\frac{k(k+1)-2}{2}, \frac{(k+1)(k+2)-2}{2}}}(\mathbf{f^{*}})  
\end{align}
Similarly we get:
\begin{align}
	\bp^{+}_{k+1}=   - \frac{1}{\sqrt{2}} \mathbf{f}^{+}_{k+1} +\cQ^{(12)}_{\frac{(k+1)(k+2)-2}{2}, \frac{k(k+1)-2}{2}}(\mathbf{f^{*}})  
\end{align}
This finishes the induction step. 

\noindent\textbf{Step 3:} Establish \eqref{Pij;expression;12}.

By Definition \ref{Q;def;3}, an alternative form of the Step 2 is $\mathbf{p}^{-}_k =  \cQ^{(14)}_{\frac{k^2-k-2}{2},\frac{k^2+k-2}{2}}(\mathbf{f}^{*})$ and $
	\mathbf{p}^{+}_k =  \cQ^{(14)}_{\frac{k^2+k-2}{2},\frac{k^2-k-2}{2}}(\mathbf{f}^{*})$ for all $k\geq 2$. 
This already covers \eqref{Pij;expression;12} when $(i,j)= (\frac{k^2-k-2}{2}, \frac{k^2+k-2}{2})$ and $(i,j)= (\frac{k^2+k-2}{2}, \frac{k^2-k-2}{2})$ for some $k\geq 2$. 
Therefore, it suffices to only consider the case that $(i,j)\neq (\frac{k^2-k-2}{2}, \frac{k^2+k-2}{2})$ and $(i,j)\neq (\frac{k^2+k-2}{2}, \frac{k^2-k-2}{2})$ for all $k\geq 2$. 
 Recall that Corollary \ref{Pij;cor;3} gives:
\begin{align}
	\frac{P_{i+1,j+1}}{z\Bz} =  \cQ^{(26)}_{i,j}(\bp^{*}) = \cQ^{(22)}_{i,j}(\bp^{*})
\end{align}
Then we can use Definition \ref{Q;def;3} to further compute:
\begin{align}
	\frac{P_{i+1,j+1}}{z\Bz} =&\cQ^{(22)}_{i,j}(\bp^{*}) \nonumber \\
	=& \cQ^{(22)}_{i,j}\Bigg(\cQ^{(14)}_{\frac{2^2-2-2}{2},\frac{2^2+2-2}{2}}(\mathbf{f}^{*}),\cQ^{(14)}_{\frac{2^2+2-2}{2},\frac{2^2-2-2}{2}}(\mathbf{f}^{*}),\cdots ,\cQ^{(14)}_{\frac{k^2 -k -2}{2},\frac{k^2+k-2}{2}}(\mathbf{f}^{*}),\cQ^{(14)}_{\frac{k^2+k-2}{2},\frac{k^2-k-2}{2}}(\mathbf{f}^{*})\Bigg) \nonumber \\
	=& \cQ^{(12)}_{i,j}(\mathbf{f}^{*}) \nonumber\\
	=&  \cQ^{(14)}_{i,j}(\mathbf{f}^{*}) \nonumber
\end{align}
where $k$ is the largest integer such that $k^2-2 < i+j$. The last line is due to the assumption that  $(i,j)\neq (\frac{k^2-k-2}{2}, \frac{k^2+k-2}{2})$ and $(i,j)\neq (\frac{k^2+k-2}{2}, \frac{k^2-k-2}{2})$ for all $k\geq 2$. The formula \eqref{Pij;expression;12} is now established.

Combining Step 1-3, the Proposition then follows.
\end{proof}

\subsection{Obstruction equation}
\begin{lemma}\label{Lem;decomp;1}
Fix any $k\geq 2$. Suppose that $\cQ\in \dC[\bx^{*}]$ is $S'$-homogeneous with:
\begin{align*}
	deg^h_{S'}(\cQ) = \Bigg(\frac{k(k-1)-2}{2},\frac{k(k+1)-2}{2}\Bigg)
\end{align*}
Then there is no (usual) degree 2 term in $\cQ$.
\begin{proof}
If $\cQ$ contains degree 2 term, assume that it is in the form of $c\cdot (\bx^{*})^{\ba}$ where $c\neq 0$,  $|\ba| =2$	and $S'(\ba) = (\frac{(k-1)k-2}{2},\frac{k(k+1)-2}{2})$.

We let $\ba = \ba^1 + \ba^2$ where $|\ba^1| = |\ba^2|=1$.  There exists integers $i,j\geq 2$ such that $(\ba^1)^{-}_i = 1$ or $(\ba^1)^{+}_i = 1$, and  $(\ba^1)^{-}_j = 1$ or $(\ba^1)^{+}_j = 1$. We may assume $i\leq j$.  
First we compute 
\begin{align}\label{S;ineq;5}
	|S'(\ba)| = (i^2-2)\Big[(\ba^1)^{+}_i + (\ba^1)^{-}_i\Big] + (j^2-2)\Big[(\ba^2)^{+}_j + (\ba^2)^{-}_j\Big] = k^2-2
\end{align}
This implies that $i,j<k$. On the other hand, 
\begin{align}\label{S;ineq;6}
	S'_2(\ba) - S'_1(\ba) = \Big[(\ba^1)^{+}_i - (\ba^1)^{-}_i\Big] \cdot i  +  \Big[(\ba^2)^{+}_j - (\ba^2)^{-}_j\Big] \cdot j = k
\end{align}
Note  $(\ba^1)^{+}_i - (\ba^1)^{-}_i$ and $(\ba^2)^{+}_j - (\ba^2)^{-}_j$  can only take values 1 or $-1$. However, none of them can take $-1$ for otherwise the $S'_2(\ba) - S'_1(\ba) $ will be at most $j - i < k$. Therefore, it must be that $(\ba^1)^{+}_i  = (\ba^2)^{+}_j =1 $ and $(\ba^1)^{-}_i = (\ba^2)^{-}_j = 0$. Then \eqref{S;ineq;5} and \eqref{S;ineq;6} imply the equation:
\begin{align*}
	i+j =k ,\quad  \quad  i^2 + j^2 =k^2 + 2
\end{align*}
 However, this equation has no solution when $k\geq 2$, a contradiction. The Lemma then follows.
\end{proof}

\end{lemma}
\begin{lemma}\label{Lem;decomp;2}
Fix any $k\geq 2$. Suppose that $\cQ\in \dC[\bx^{*}]$ is $S'$-homogeneous with:
\begin{align*}
	deg^h_{S'}(\cQ) = \Bigg(\frac{(k^2+k-1)(k^2+k-2)-2}{2},\frac{(k^2+k-1)(k^2+k)-2}{2}\Bigg)
\end{align*} 
and with no (usual) degree 1 term. Then $\cQ$ admits the decomposition
\begin{align}
	\cQ(\bx^{*}) = c\bx^{+}_k \bx^{-}_{k+1}\bx^{-}_{k^2+k-2} + \cQ'(\bx^{*}) + \cQ''(\bx^{*}) 
\end{align}  where $c\in \dC$ is the coefficient of $\bx^{+}_k \bx^{-}_{k+1}\bx^{-}_{k^2+k-2}$ in $\cQ$ and polynomials $\cQ', \cQ''\in \dC[\bx^{*}]$ satisfy: 
\begin{itemize}
	\item $\cQ'(\bx^{*})$ depends only on $\bx^{-}_k, \bx^{+}_k, \cdots, \bx^{-}_{k^2+k-3}, \bx^{+}_{k^2+k-3}$.
	\item $\cQ''(\bx^{*})$ depends only on $\bx^{-}_2, \bx^{+}_2, \cdots, \bx^{-}_{k^2+k-2}, \bx^{+}_{k^2+k-2}$ and each monomial in $\cQ''(\bx^{*})$ contains at least one of $\bx^{-}_2, \bx^{+}_2, \cdots, \bx^{-}_{k-1}, \bx^{+}_{k-1}$. In other words, if $\bx^{-}_2 = \bx^{+}_2 = \cdots = \bx^{-}_{k-1} = \bx^{+}_{k-1} = 0$, then $\cQ''(\bx^{*}) \equiv 0$.
	\item There is no (usual) degree 1, 2 term in  $\cQ'$ and $\cQ''$.
\end{itemize} 
\end{lemma}
\begin{proof}
Throughout the proof, all constants are local to the Lemma and will not interact with any other part of the paper.

The whole matter  is about the index, as we will see in the following.  Since $\cQ$ does not contain (usual) degree 1 term, and Lemma \ref{Lem;decomp;1}, applied with $k^2+k-1$ in place of $k$, implies that $\cQ$ also does not contain (usual) degree 2 term,  one can write:  
\begin{align}
	\cQ(\bx^{*}) =& \sum_{\substack{S'(\ba) = (\frac{(k^2+k-1)(k^2+k-2)-2}{2},\frac{(k^2+k-1)(k^2+k)-2}{2})\\ |\ba|\geq3}} C_{\ba} (\bx^{*})^{\ba}  
\end{align}
Since 
\begin{align}\label{S;ineq;1}
	|S'(\ba)| = \sum_{j\geq 2} (j^2-2)(\ba^{-}_{j} + \ba^{+}_{j}) 
\end{align}
the summation only involves finitely many $\ba$, and so are all summations below.

 Define 
\begin{align}
	\bar{\cQ}(\bx^{*}) = \sum_{\substack{S'(\ba) = (\frac{(k^2+k-1)(k^2+k-2)-2}{2},\frac{(k^2+k-1)(k^2+k)-2}{2})\\ \ba^{-}_2 =\ba^{+}_2 =\cdots = \ba^{-}_{k-1} = \ba^{+}_{k-1} = 0 \\ \ba^{-}_{k^2+k-2} + \ba^{+}_{k^2+k-2} + \ba^{-}_{k^2+k-1} + \ba^{+}_{k^2+k-1} + \cdots \geq 1\\\\ |\ba|\geq3}} C_{\ba}\cdot (\bx^{*})^{\ba} 
\end{align}
\begin{align}
	\cQ'(\bx^{*}) = \sum_{\substack{S'(\ba) = (\frac{(k^2+k-1)(k^2+k-2)-2}{2},\frac{(k^2+k-1)(k^2+k)-2}{2})\\ \ba^{-}_2 =\ba^{+}_2 =\cdots = \ba^{-}_{k-1} = \ba^{+}_{k-1} = 0 \\ \ba^{-}_{k^2+k-2} = \ba^{+}_{k^2+k-2} = \ba^{-}_{k^2+k-1} = \ba^{+}_{k^2+k-1} = \cdots = 0\\ |\ba|\geq 3}} C_{\ba}\cdot (\bx^{*})^{\ba} 
\end{align}
\begin{align}
	\cQ''(\bx^{*}) :=&\sum_{\substack{S'(\ba) = (\frac{(k^2+k-1)(k^2+k-2)-2}{2},\frac{(k^2+k-1)(k^2+k)-2}{2})\\ \ba^{-}_2 +\ba^{+}_2 +\cdots + \ba^{-}_{k-1} + \ba^{+}_{k-1} \geq 1\\ |\ba|\geq 3 }} C_{\ba}\cdot  (\bx^{*})^{\ba} 
\end{align}

By definition we can directly check that $\cQ'$ and $\cQ''$ satisfies the conditions in the Lemma, together with:
\begin{align}
	\cQ(\bx^{*}) = \bar{\cQ}(\bx^{*}) +  \cQ'(\bx^{*}) + \cQ''(\bx^{*}) 
\end{align}
Clearly, $\bar{\cQ}$ contains $c\bx^{+}_k \bx^{-}_{k+1}\bx^{-}_{k^2+k-2}$ since this term satisfies all conditions under the summation. It suffices to show that this is the only term.
 
Consider a nonnegative double multi-index $\ba = (\ba^{-}_2, \ba^{+}_2,\ba^{-}_3, \ba^{+}_3,....)$ with finitely many nonzero satisfying $\ba^{-}_2 =\ba^{+}_2 =\cdots = \ba^{-}_{k-1} = \ba^{+}_{k-1} = 0$, \ $ \ba^{-}_{k^2+k-2} + \ba^{+}_{k^2+k-2} + \ba^{-}_{k^2+k-1} + \ba^{+}_{k^2+k-1} + \cdots \geq 1$ and 
\begin{align}
	S'(\ba) = \Bigg(\frac{(k^2+k-1)(k^2+k-2)-2}{2},\frac{(k^2+k-1)(k^2+k)-2}{2}\Bigg)
\end{align}
Moreover, by the previous step, we may assume that $|\ba| \geq 3$. It suffices to show that such $\ba$ has precisely three nonzero slot: $\ba^{+}_k = \ba^{-}_{k+1} = \ba^{-}_{k^2+k-2} = 1$. This will follow from a sequel of claims below:

\textbf{Claim 1}: $\ba^{-}_m = \ba^{+}_m = 0$ for all $m \geq k^2+k-1$.

Since $|S'(\ba)| = (k^2+k-1)^2-2$, \eqref{S;ineq;1} implies that $\ba^{-}_m = \ba^{+}_m = 0$ for all $m > k^2+k-1$. In addition, if $\ba^{-}_{k^2+k-1} \geq 1$ or $\ba^{+}_{k^2+k-1}\geq 1$, this will be incompatible with $|\ba| \geq 3$ and \eqref{S;ineq;1}. The claim follows.

Now we are forced to have $\ba^{-}_{k^2+k-2} \geq 1$ or $\ba^{+}_{k^2+k-2}\geq 1$. We need to use formula for $S'$ in the next:
\begin{align}\label{S;ineq;3}
	S'(\ba) =& (S'_1(\ba), S'_2(\ba)) \nonumber\\
    =& \Big(\sum_{j\geq 2} \frac{j(j-1)-2}{2} \cdot \ba^{-}_{j} + \frac{j(j+1)-2}{2}\cdot \ba^{+}_{j},\quad  \sum_{j\geq 2} \frac{j(j+1)-2}{2} \cdot \ba^{-}_{j} +  \frac{j(j-1)-2}{2}\cdot \ba^{+}_{j}\Big)
\end{align}

This claim shows that all polynomials throughout the Lemma depends at most on $\bx^{-}_2, \bx^{+}_2, \cdots, \bx^{-}_{k^2+k-2}, \bx^{+}_{k^2+k-2}$.

\textbf{Claim 2}: $\ba^{+}_{k^2+k-2} = 0$.

Otherwise if $\ba^{+}_{k^2+k-2}\geq 1$, then comparing $S'_1(\ba)$ we find that $\ba^{+}_{k^2+k-2} = 1$ and that every remaining factor must be $\bx_2^-$. Such factors are forbidden for $k\ge 3$. When $k=2$, comparison with $S'_2(\ba)$ gives $2\ba^-_2+5=14$, which is impossible. 

Therefore, we are forced to have $\ba^{-}_{k^2+k-2}\geq 1$ and $\ba^{+}_{k^2+k-2} = 0$. 

\textbf{Claim 3}: $\ba^{-}_{k^2+k-2} = 1$. This follows from comparing $S'_1(\ba)$. Indeed, we get inequality
\begin{align}
	\ba^{-}_{k^2+k-2}\cdot  [(k^2+k-2)(k^2+k-3)-2] \leq (k^2+k-1)(k^2+k-2)-2
\end{align}
Since $k\geq 2$, this implies $\ba^{-}_{k^2+k-2} \leq \frac{9}{5} < 2$, which forces $\ba^{-}_{k^2+k-2} = 1$. 

\textbf{Claim 4}: $\ba^{+}_k = \ba^{-}_{k+1} = 1$, $\ba^{-}_k = \ba^{+}_{k+1} = 0$ and $\ba^{-}_{j} = \ba^{+}_j = 0$ for all $k+2\leq j \leq k^2+k-3$

This can be done by checking $S'_1(\ba)$. By all above we shall get
\begin{align}\label{S;ineq;4}
	& \sum_{j = k}^{k^2+k-3}\ba^{-}_{j}  \cdot \frac{j(j-1)-2}{2}  + \ba^{+}_{j}\cdot \frac{j(j+1)-2}{2}  \nonumber\\
	=& \frac{(k^2+k-2)(k^2+k-1)-2}{2} -\ba^{-}_{k^2+k-2} \cdot  \frac{(k^2+k-2)(k^2+k-3)-2}{2} \nonumber\\
	=& k^2+k-2
\end{align} 
If $\ba^{-}_j \geq 1$ or $\ba^{+}_j \geq 1$ for some $k+2\leq j \leq k^2+k-3$ then by the fact that $|\ba| \geq 3$ we get 
\begin{align}
	 \sum_{j = k}^{k^2+k-3}\ba^{-}_{j}  \cdot \frac{j(j-1)-2}{2}  + \ba^{+}_{j}\cdot \frac{j(j+1)-2}{2}   \geq \frac{k(k-1)-2}{2} + \frac{(k+2)(k+1)-2}{2} = k^2+k-1
\end{align}
A contradiction. The same contradiction arises if we assume $\ba^{+}_{k+1} \geq 1$, hence we get $\ba^{+}_{k+1} = 0$. 

Finally we look at $\ba^{-}_k, \ba^{+}_k,  \ba^{-}_{k+1}$. To that end, it will be more convenient to look at $|S'(\ba)|$. By our assumption and above results
\begin{align}
	\sum_{j=k}^{k^2+k-3} (j^2-2)(\ba^{-}_j + \ba^{+}_j) &= (k^2-2) (\ba^{-}_k + \ba^{+}_k)  + [(k+1)^2-2] \ba^{-}_{k+1}  \nonumber\\
	&= (k^2+k-1)^2- (k^2+k-2)^2 \\
	&= 2k^2+2k-3 \nonumber
\end{align}
Since $k\geq 2$ , this  equation forces $\ba^{-}_k + \ba^{+}_k = \ba^{-}_{k+1} = 1$ when $k\neq 3$. By \eqref{S;ineq;4} we immediately get $\ba^{+}_k = 1$ and $\ba^{-}_k = 0$. 
When $k=3$, there is an additional solution $\ba^{-}_3 + \ba^{+}_3 = 3$ and $\ba^{-}_4 = 0$. However, one can quickly check that this must violate \eqref{S;ineq;4} . Therefore,  we are still left with the only choice $\ba^{+}_k = 1$ and $\ba^{-}_k = 0$ for all $k\geq 2$. 
This finishes the claim.

So far, we proved that there is at most one possible $\ba$ in the summation of $\bar{\cQ}$, which corresponds to the desired term $c\bx^{+}_k \bx^{-}_{k+1}\bx^{-}_{k^2+k-2}$. This finishes the proof. 

\end{proof}

We are ready to obtain the main result of this section:
\begin{theorem}[Obstruction equation] \label{Thm;14}
For each $k\geq 2$, there exist universal polynomials $\cQ^{(7)}_k, \cQ^{(8)}_k\in \dC[\bx^*]$ depending only on $k$  satisfying:
\begin{itemize}
	\item $\cQ^{(7)}_k, \cQ^{(8)}_k$ are $S'$-homogeneous with $deg^h_{S'}(\cQ) = \Bigg(\frac{(k^2+k-1)(k^2+k-2)-2}{2},\frac{(k^2+k-1)(k^2+k)-2}{2}\Bigg)$. 
	\item $\cQ^{(7)}_k(\bx^{*})$ depends only on $\bx^{-}_k, \bx^{+}_k, \cdots, \bx^{-}_{k^2+k-3}, \bx^{+}_{k^2+k-3}$.
	\item $\cQ^{(8)}_k(\bx^{*})$ depends only on $\bx^{-}_2, \bx^{+}_2, \cdots, \bx^{-}_{k^2+k-2}, \bx^{+}_{k^2+k-2}$ and each monomial in $\cQ^{(8)}_k(\bx^{*})$ contains at least one of $\bx^{-}_2, \bx^{+}_2, \cdots, \bx^{-}_{k-1}, \bx^{+}_{k-1}$. 
\end{itemize} 
Moreover, the following obstruction equation holds for all $k\geq 2$ with $(k^2+k-1)^2\leq N$:
\begin{align}\label{f;obstruction;2}
	C^{(6)}_k\cdot \mathbf{f}^{+}_k\mathbf{f}^{-}_{k+1}\mathbf{f}^{-}_{k^2+k-2} + \cQ^{(7)}_{k}(\mathbf{f}^{*}) + \cQ^{(8)}_{k}(\mathbf{f}^{*})=0
\end{align}
where $$C^{(6)}_k = 2^{-3/2}\cdot \frac{2(k^2+k)^2(k^2+k-1)(k^2+k-2)^2(k^2+k-3)}{(k^3+k^2-2k-1)(k^3+2k^2-k-1)} >0.$$
\end{theorem}
\begin{proof}
Taking $C^{(2)}$ from Definition \ref{Def;C1C2} and define for all $i,j\geq 0$:
\begin{align}
	\cQ^{(2)}_{i,j}(\bx^{*}):= \sum_{\substack{S'(\ba) = (i,j) \\ |\ba| \geq 2}}C^{(2)}_{\ba} \prod_{l\geq2}(\bx^{-}_{l})^{\ba^{-}_l}(\bx^{+}_{l})^{\ba^{+}_l} =  \sum_{\substack{S(\ba) = (i+1,j+1) \\ |\ba| \geq 2}}C^{(2)}_{\ba}\prod_{l\geq2}(\bx^{-}_{l})^{\ba^{-}_l}(\bx^{+}_{l})^{\ba^{+}_l}
\end{align}
Then $\cQ^{(2)}_{i,j}$ is $S'$-homogeneous with $deg^h_{S'}(\cQ^{(2)}) = (i,j)$. In particular, $deg^h_{|S'|}(\cQ^{(2)}_{i,j}) = i+j$.  Moreover,  $\cQ^{(2)}_{i,j}$ contains no (usual) degree 1 term. 
Then Lemma \ref{Lem;Sdeg;1} implies that $\cQ^{(2)}_{i,j}(\bx^{*}) = \cQ^{(2)}_{i,j}(\bx^{-}_2, \bx^{+}_2,\cdots, \bx^{-}_{l-1}, \bx^{+}_{l-1})$, where $l$ is the smallest integer satisfying $l^2-2\geq i+j$. 

Next we define 
\begin{align}
	\cQ^{(3)}_{i,j}(\bx^{*}) := \sum_{\substack{S'(\ba) = (i,j) \\ |\ba| \geq 2}} \prod_{l\geq2} C^{(2)}_{\ba}\Big[-2^{-1/2}\bx^{-}_{l} + \cQ^{(12)}_{\frac{l^2-l-2}{2},\frac{l^2+l-2}{2}}(\bx^{*})\Big]^{\ba^{-}_l}\Big[-2^{-1/2}\bx^{+}_{l} + \cQ^{(12)}_{\frac{l^2+l-2}{2},\frac{l^2-l-2}{2}}(\bx^{*})\Big]^{\ba^{+}_l} 
\end{align}
By Lemma \ref{Lem;poly-comp},  $\cQ^{(3)}_{i,j}$ is also $S'$-homogeneous with $deg^h_{S'}(\cQ^{(3)}) = (i,j)$. By Lemma \ref{Lem;decomp;1}, when $(i,j) = (\frac{m^2-m-2}{2}, \frac{m^2+m-2}{2})$ or $(i,j) = (\frac{m^2+m-2}{2}, \frac{m^2-m-2}{2})$, the summation only contains $\ba$ with $|\ba|\geq 3$. By Lemma \ref{Q;deg;lem;3} there is no (usual) degree 0, 1 term in $\cQ^{(12)}_{i,j}$ for all $i+j\geq 0$ , we then conclude that there is no (usual) degree 0,1,2 terms in  $\cQ^{(3)}_{\frac{m^2-m-2}{2}, \frac{m^2+m-2}{2}}$. Moreover, the homogeneous (usual) degree 3 part of $\cQ^{(3)}_{\frac{m^2-m-2}{2}, \frac{m^2+m-2}{2}}$ is precisely 
\begin{align}
	-2^{-3/2}\sum_{\substack{S'(\ba) = (\frac{m^2-m-2}{2}, \frac{m^2+m-2}{2}) \\ |\ba| = 3}} \prod_{l\geq2} C^{(2)}_{\ba}\big(\bx^{-}_{l}\big)^{\ba^{-}_l}\big(\bx^{+}_{l}\big)^{\ba^{+}_l} 
\end{align}

Using Definition \ref{Def;bfp;1} we can simplify the obstruction equation in Theorem \ref{Thm;11} as 
\begin{align}\label{Pij,expression;21}
	\cQ^{(2)}_{\frac{m(m-1)-2}{2},\frac{m(m+1)-2}{2}}(\bp^{*}) = 0
\end{align} 
for all $m \geq 2$ with $m^2 \leq N$. Note that $\cQ^{(2)}_{\frac{m(m-1)-2}{2},\frac{m(m+1)-2}{2}}(\bx^{*})  = \cQ^{(2)}_{\frac{m(m-1)-2}{2},\frac{m(m+1)-2}{2}}(\bx^{-}_2, \bx^{+}_2,\cdots, \bx^{-}_{m-1}, \bx^{+}_{m-1}) $. 

Setting $m = k^2+k-1$. To compute \eqref{Pij,expression;21} we only to know $(\bp^{-}_2, \bp^{+}_2,\cdots, \bp^{-}_{k^2+k-2}, \bp^{+}_{k^2+k-2})$, which are all described in Proposition \ref{Prop;ptof;1}. This allows us to use Propsotion \ref{Prop;ptof;1} to  substitute $\bp^{*}$ by expressions of $\mathbf{f}^{*}$: 
\begin{align}
	\cQ^{(3)}_{\frac{(k^2+k-1)(k^2+k-2)-2}{2}, \frac{(k^2+k-1)(k^2+k)-2}{2}}(\mathbf{f}^{*}) = 0
\end{align}

Putting them together, the coefficient  $C^{(6)}_k$ of $\bx^{+}_k \bx^{-}_{k+1}\bx^{-}_{k^2+k-2}$ in $\cQ^{(3)}_{\frac{(k^2+k-1)(k^2+k-2)-2}{2}, \frac{(k^2+k-1)(k^2+k)-2}{2}}(\mathbf{x}^{*})$ must be:
\begin{align}
	C^{(6)}_k = -2^{-3/2}\cdot C^{(2)}_{\bar{\ba}}
\end{align}
where $\bar{\ba}$ a double multi-index with the only nonzero elements to be $\bar{\ba}^{+}_k = \bar{\ba}^{-}_{k+1} = \bar{\ba}^{-}_{k^2+k-2}=1$. By Theorem \ref{Thm;12} we can further compute
\begin{align}
	C^{(6)}_k = 2^{-3/2}\cdot \frac{2(k^2+k)^2(k^2+k-1)(k^2+k-2)^2(k^2+k-3)}{(k^3+k^2-2k-1)(k^3+2k^2-k-1)}
\end{align}
Finally, we apply Lemma \ref{Lem;decomp;1} and Lemma \ref{Lem;decomp;2}, the Theorem follows immediately.
\end{proof}

\section{Prescribing higher order asymptotics}\label{Sec;prescribe-higher-asymptotic}
{\subsection{Introduction}
In this section, we will always discuss general dimension $n\geq 1$.  Recall that the MCF with spherical singularity corresponds to a RMCF with graph function $v$ converging to 0 over a sphere $\mathbb{S}^n(\sqrt{2n})$.   The graph function $v$ satisfies
\begin{align}\label{v;equation;1}
	\partial_{\tau}v = L v + N(v,\nabla v, \nabla^2v)
\end{align}
where $Lv = \Delta v + v$ and $N$ is the nonlinear term. We shall abbreviate $N(v,\nabla v, \nabla^2 v)$ as $N(v)$.

Recall from  Section \ref{section;pre} that $E_k$ is the $k$-th eigenspace of the linearized operator $L = \Delta + 1$ on the sphere $\mathbb{S}^n(\sqrt{2n})$. The main result of this section is the following:
\begin{theorem}\label{Thm;prescribe-diff-asymptotic;2}
	Fix any $r > \frac{n}{2}+2, k\geq2$, given  an admissible solution $v_0$ to \eqref{v;equation;1} on $[\tau_0,\infty)$ and given $Q_k\in E_k$, set $\eta:= \frac{1}{4n}$. There exists a constant $C>0$, $\tau_1\geq \tau_0$ and another admissible solution $v_1$ of \eqref{v;equation;1} defined on $[\tau_1,\infty)$ such that 
\begin{align}
	  |e^{\lambda_k \tau}(v_1-v_0) - Q_k|_{C^0(\mathbb{S}^n)} = O( e^{-\eta \tau}).
\end{align}
\end{theorem}

Let us outline this section. In the subsection \ref{subsec;stable-manifold}, we will establish a stable manifold construction for the difference equation \eqref{w;equation;1} with a base solution $v_0$ solving \eqref{v;equation;1}, which gives existence and uniqueness of the solution $w$ to \eqref{w;equation;1} in a small ball of suitable Banach space, see Corollary \ref{Cor;existence;1}. 
Moreover, our design of the Banach space ensures that the solution $w$ satisfies a rough asymptotic decay. Next in the subsection \ref{subsec;refined-asymptotics}, we first refine the asymptotic decay rate to the optimal one: $w = e^{-\lambda_k \tau}Q_k + O(e^{-(\lambda_k + \eta)\tau})$, see Proposition \ref{Prop;fine-asymptotic;1}. Then we use a fixed point argument to show that all possible asymptotics with small $Q_k$ can be realized, see Theorem \ref{Thm;prescribe-diff-asymptotic;1}. Finally, the main Theorem \ref{Thm;prescribe-diff-asymptotic;2} follows from a scaling and shifting argument. 

This section follows the spirit of Strehlke's work \cite{str}. However, a key difficulty arises from the mismatch of asymptotic decay rates between the base solution $v_0$ and the difference solution $w$. The norm $||\cdot||_{r,\sigma}$ introduced in \cite{str} cannot handle different asymptotic rates simultaneously. In particular, the estimate of $\Pi_k D$ in the proof of Theorem 3.7 in \cite{str} breaks down. To overcome this difficulty, we introduce an additional Banach space with norm $||\cdot||_{r,\sigma,\eta}$. The extra parameter $\eta$ accommodate different decay rates and eventually allows us to piece all the estimates together.

\subsection{Setup}\label{subsec;prescribe;setup}

The following set up follows closely the framework of Strehlke \cite{str}:
\begin{itemize}
	\item Define inner product\begin{align*}
			\langle u,v\rangle = \int_{\mathbb{S}^n}uv 
		  \end{align*}
	\item Define \begin{align*}
			F_k = \bigoplus_{j=k}^{\infty} E_j
		  \end{align*}
	\item Let $\{\phi_{i,j}\}_{1\leq j \leq d_i}$ be $L^2$ orthonormal basis of $E_i$. Therefore thus $\{\phi_{i,j}\}_{i\geq 0, 1\leq j \leq d_i}$ is an orthonormal basis of $L^2(\mathbb{S}^n)$. Then for any $v\in L^2(\mathbb{S}^n)$ one can decompose $ v = \sum_{i\geq 0, 1\leq j \leq d_i} a_{i,j}\phi_{i,j}$. Clearly $||v||_{L^2}^2 = \sum_{i\geq 0, 1\leq j \leq d_i}a_{i,j}^2$. We define
		\begin{align*}
			||v||_{H^r}^2 = \sum_{\substack{i\geq 0\\ 1\leq j\leq d_i}} a_{i,j}^2 |\lambda_i|^{r} 
		\end{align*}
		Since $Ker(L)$ is empty on $\mathbb{S}^n$, the above definition is equivalent to usual $H^r$ norm. Moreover,
		\begin{align*}
			||v||_{H^r}^2 = \langle(-L)^r v, v\rangle
		\end{align*}
		when $v\in F_2 $ (which contains any $F_k$, $k\geq 2$.) 
		
	\item Define $\Pi_k, \pi_k$ to be the $L^2$ orthogonal projection onto $F_k, E_k$ resp. Indeed, it is equivalent to $H^r$ orthogonal projection for any $r\geq 0$.
	\item For $\sigma \geq 0$, define the norm $||\cdot||_{r,\sigma}$
		  \begin{align*}
		  	 ||v||_{r,\sigma} = \Big(\int_0^{\infty} ||v(s)||^2_{H^{r+1}}ds\Big)^{1/2} + \sup\limits_{s\geq 0}e^{\sigma s}||v(s)||_{H^r}
		  \end{align*}
	\item Define $X_{r,\sigma} = \{v: [0,\infty) \rightarrow H^{r+1} \ \big| \ ||v||_{r,\sigma} < \infty\}$ to be the space of functions with finite $||\cdot||_{r,\sigma}$ norm.
	\item For $p_0\in F_k$ and $v \in X_{r,\sigma}$, define 
		  \begin{align}\label{Def;T}
			T_0(v,p_0)\Big|_{\tau} = e^{L\tau}p_0 + \int_{0}^{\tau} e^{L(\tau-s)}\Pi_k(N(v(s)))ds - \int_{\tau}^{\infty} e^{L(\tau-s)}(1-\Pi_k)(N(v(s)))ds  
		  \end{align}
		Note that the integral converges, this can be seen from the spirit of the proof of Theorem \ref{Thm;contraction-map;1}.
		\item We will frequently use the following nonlinear estimate (see \cite[Lemma 3.5]{str}): if $r > \frac{n}{2}+1$, then for any $v_1, v_2\in H^r$ with $||v_1||_{H^r}\leq 1$ and $||v_2||_{H^r}\leq  1$:
			\begin{align}\label{Nonlinear;1}
			||N(v_1) - N(v_2)||_{H^{r-1}} \leq C_3\cdot \Big( ||v_1||_{H^r}  ||v_1-v_2||_{H^{r+1}} + ||v_2||_{H^{r+1}}  ||v_1-v_2||_{H^{r}} \Big)
			\end{align} 
		where $C_3$ depends only on $r$. 

\end{itemize}

\subsection{Stable manifold construction}\label{subsec;stable-manifold}
We would like to modify the Theorem 3.7 in Strehlke's work \cite{str}.  To this end, we need to define another norm $||\cdot||_{r,\sigma,\eta}$ for $0\leq \eta\leq \sigma$:
\begin{align}\label{Def;r-sigma-eta-norm}
	||v||_{r,\sigma,\eta} := \Big(\int_0^{\infty} e^{2(\sigma -\eta)s} ||v(s)||^2_{H^{r+1}}ds\Big)^{1/2} + \sup\limits_{s\geq 0}e^{\sigma s}||v(s)||_{H^r}
\end{align}
We shall also define $X_{r,\sigma,\eta}$ to be the space of functions with finite $||\cdot||_{r,\sigma,\eta}$:
\begin{align} 
	X_{r,\sigma,\eta} = \{v: [0,\infty)\rightarrow H^{r+1} \ \big| \ ||v||_{r,\sigma,\eta} < \infty\}
\end{align}

It is clear that $||\cdot||_{r,\sigma,\sigma} = ||\cdot||_{r,\sigma}$. Moreover, if $0 \leq \eta \leq \sigma$, then $||\cdot||_{r,\eta}\leq ||\cdot||_{r,\sigma}\leq ||\cdot||_{r,\sigma,\eta} $.

Since we will only be interested in the asymptotic behavior, we can define smaller class of solutions:
\begin{definition}
	We say $v$ is an admissible solution to \eqref{v;equation;1}, if $v$ converges to 0 in $C^2$ as $\tau\rightarrow\infty$.
\end{definition}

\begin{lemma}\label{Lem;admissible-norm}
	Suppose that $v$ is defined on  $[\tau_0,\infty)$, then $v$ is admissible  if and only if $||v(\cdot, \cdot + \tau_0)||_{r,\eta} < \infty$ for some $r > \frac{n}{2}+2 $ and $\eta \in (0,\lambda_2)$.
\end{lemma}
\begin{proof}
	By a time shifting, we may assume that $\tau_0 = 0$.
	
	If $||v||_{r,\eta} <\infty$ for some $\eta \in (0,\lambda_2)$ and $r > \frac{n}{2}+2 $, then by Sobolev embedding, $||v||_{C^2}\rightarrow 0$ and thus is an admissible solution. 
	
	If $v$ is admissible, then by the work of \v{S}e\v{s}um \cite{se} (see also \cite{str}), $||v(t)||_{C^l} =  O(e^{-\lambda_2 \tau})$ as $t\rightarrow\infty$ for each $l\geq 0$. Since $\eta < \lambda_2$ we have $||v||_{r,\eta}<\infty$.
\end{proof}

Our main estimate in this subsection is the following:

\begin{theorem}\label{Thm;contraction-map;1}
If $r> \frac{n}{2}+1, k\geq 2$ and $p_0\in F_k$ and if $\lambda_{k-1} < \sigma < \lambda_k$, assume that $||v_1||_{r,0} + ||v_2||_{r,0} < 1$, then for any $0<\eta \leq\sigma$,
\begin{align}
	||T_0(v_1,p_0) - T_0(v_2,p_0)||_{r,\sigma,\eta}\leq C_1\Big(||v_1||_{r,\eta}+ ||v_2||_{r,\eta}\Big)\cdot ||v_1-v_2||_{r,\sigma,\eta}
\end{align}
where  $C_1=C(n,r)\Big[(1+\sigma^{\frac{1}{2}}\eta^{-\frac{1}{2}})\Big(\frac{\lambda_k}{\lambda_k-\sigma} \Big)^{\frac{1}{2}} + (1+\lambda_k^{\frac{1}{2}}\eta^{-\frac{1}{2}})\frac{\lambda_k}{\sigma - \lambda_{k-1}} \Big]$ depends only on $r, k,\sigma,\eta$.

\end{theorem}

\begin{remark}
	The only place we need $||v_1||_{r,0} + ||v_2||_{r,0} < 1$ is  
	\begin{align*}
		||N(v_1) - N(v_2)||_{H^{r-1}}\leq C(||v_1||_{H^r} + ||v_2||_{H^r}) \cdot ||v_1 - v_2||_{H^{r+1}} + C(||v_1||_{H^{r+1}} + ||v_2||_{H^{r+1}}) \cdot ||v_1 - v_2||_{H^{r}}
	\end{align*}
	where we need to ensure that $C$ depends only on $N$ and $r$. In general, if we only have $||v_1||_{r,0} + ||v_2||_{r,0} < R$  then $C$ will depend on $R$ as well. Since the nonlinear term $N$ is fixed in our paper, we will drop the dependence on $N$ from now on.
	\end{remark}

\begin{remark}
	We only require $||v_i||_{r,0} < 1$ because we only need $H^r$ norm being bounded by 1 for each time slice. No decay is needed. For all later application, we assume smallness of $||v_i||_{r,\eta}$, so this condition will automatically be satisfied. 
\end{remark}

\begin{proof}
As in Strehlke's work \cite{str}, let
\begin{align*}
	D(\tau) := T_0(v_1,p_0) - T_0(v_2,p_0) \Big|_{\tau}=& \int_{0}^{\tau} e^{L(\tau-s)}\Pi_k\Big[N\Big((v_1(s)\Big) - N\Big(v_2(s)\Big)\Big]ds \\
	&-\int_{\tau}^{\infty} e^{L(\tau-s)}(1-\Pi_k)\Big[N\Big((v_1(s)\Big) - N\Big(v_2(s)\Big)\Big]ds 
\end{align*}

Note that 
\begin{align}
	\Pi_k D(\tau) =&  \int_{0}^{\tau} e^{L(\tau-s)}\Pi_k\Big[N\Big((v_1(s)\Big) - N\Big(v_2(s)\Big)\Big]ds \\
	(1 - \Pi_k) D(\tau) =& - \int_{\tau}^{\infty} e^{L(\tau-s)}(1-\Pi_k)\Big[N\Big((v_1(s)\Big) - N\Big(v_2(s)\Big)\Big]ds 
\end{align}
and
\begin{align}
	(\partial_{\tau} - L)(\Pi_k D) = \Pi_k\Big[N\Big((v_1(\tau)\Big) - N\Big(v_2(\tau)\Big)\Big]
\end{align}

We will need the following energy estimate (see \cite[Lemma 3.2]{str}):  for each $k\geq 2$, if $v :\RR_{\geq 0}\rightarrow H^{r+1}\cap F_k$ is a continuously differentiable path, then
\begin{align}\label{energy-estimate;1}
	\frac{1}{2}\frac{d}{d\tau}||v||_{H^r}^2 + (1-\varepsilon)||v||_{H^{r+1}}^2 \leq \frac{1}{4\varepsilon}||(\partial_{\tau} - L)v||_{H^{r-1}}^2
\end{align}
In particular, if we let $\varepsilon = \frac{\lambda_k -\sigma}{\lambda_k}$, use the fact that $||v||_{H^{r+1}}^2 \geq \lambda_k ||v||_{H^r}^2$ and integrate with integrating factor $e^{2\sigma \tau}$, we have (see \cite[Corollary 3.3]{str}):
\begin{align}\label{energy-estimate;2}
	e^{2\sigma \tau}||v(\tau)||_{H^r}^2 \leq ||\Pi_k v(0)||_{H^r}^2 + \frac{\lambda_k}{2(\lambda_k-\sigma)}\int_0^\tau e^{2\sigma s}||(\partial_{\tau} - L)v(s)||_{H^{r-1}}^2ds
\end{align}

\textbf{Estimating} $\Pi_kD$:
We need to estimate each part of the norm. Using \eqref{energy-estimate;2} with  $\Pi_k D$ in place of $v$:
\begin{align*}
	e^{2\sigma \tau}||\Pi_k D(\tau)||_{H^r}^2\leq&  ||\Pi_kD(0)||_{H^r}^2 +   \frac{\lambda_k}{2(\lambda_k-\sigma)} \int_{0}^{\tau}  e^{2\sigma s} \big\Vert (\partial_{t} - L)(\Pi_k D)  \big\Vert_{H^{r-1}}^2  ds\\
	\leq &   \frac{\lambda_k}{2(\lambda_k-\sigma)} \int_{0}^{\tau}  e^{2\sigma s} \big\Vert N\Big((v_1(s)\Big) - N\Big(v_2(s)\Big) \big\Vert_{H^{r-1}}^2  ds\\
	\leq&  \frac{C(r) \cdot \lambda_k}{(\lambda_k-\sigma)} \int_{0}^{\tau}  e^{2\sigma s}  \Big(\Vert v_1(s) \Vert_{H^{r+1}}^2 + \Vert v_2(s) \Vert_{H^{r+1}}^2\Big) \cdot \Vert v_1(s)-v_2(s)\Vert_{H^{r}}^2  ds\\
	&+ \frac{C(r) \cdot \lambda_k}{(\lambda_k-\sigma)}  \int_{0}^{\tau}  e^{2\sigma s}  \Big(\Vert v_1(s) \Vert_{H^{r}}^2 + \Vert v_2(s) \Vert_{H^{r}}^2\Big)\cdot  \Vert v_1(s)-v_2(s)\Vert_{H^{r+1}}^2  ds
\end{align*}
For the first term:
\begin{align*}
	&\int_{0}^{\tau}  e^{2\sigma s}  \Big(\Vert v_1(s) \Vert_{H^{r+1}}^2 + \Vert v_2(s) \Vert_{H^{r+1}}^2\Big) \cdot \Vert v_1(s)-v_2(s)\Vert_{H^{r}}^2  ds\\
	\leq & \Big( \sup\limits_{s\in [0,\tau]} e^{2\sigma s} \Vert v_1(s)-v_2(s)\Vert_{H^{r}}^2\Big)\cdot \int_{0}^{\tau} \Vert v_1(s) \Vert_{H^{r+1}}^2 + \Vert v_2(s) \Vert_{H^{r+1}}^2  ds\\
	\leq& ||v_1-v_2||_{r,\sigma,\eta}^2\cdot (||v_1||_{r,\eta}^2 + ||v_2||_{r,\eta}^2 ) 
\end{align*}

For the second term
\begin{align*}
	&\int_{0}^{\tau}  e^{2\sigma s}  \Big(\Vert v_1(s) \Vert_{H^{r}}^2 + \Vert v_2(s) \Vert_{H^{r}}^2\Big)\cdot  \Vert v_1(s)-v_2(s)\Vert_{H^{r+1}}^2  ds\\
	\leq&\int_{0}^{\tau}   e^{2\eta s} \Big(\Vert v_1(s) \Vert_{H^{r}}^2 + \Vert v_2(s) \Vert_{H^{r}}^2\Big)\cdot  e^{2(\sigma-\eta) s} \Vert v_1(s)-v_2(s)\Vert_{H^{r+1}}^2 ds\\
	\leq&\Big[ \sup\limits_{s\in [0,\tau]}  e^{2\eta s} \Big(\Vert v_1(s) \Vert_{H^{r}}^2 + \Vert v_2(s) \Vert_{H^{r}}^2\Big)\Big]\int_{0}^{\tau}     e^{2(\sigma-\eta) s} \Vert v_1(s)-v_2(s)\Vert_{H^{r+1}}^2 ds\\
	\leq& (||v_1||_{r,\eta} + ||v_2||_{r,\eta} )^2 \cdot ||v_1-v_2||_{r,\sigma,\eta}^2 
\end{align*}

Putting them together gives
\begin{align*}
	e^{2\sigma \tau}||\Pi_k D(\tau)||_{H^r}^2\leq  \frac{C(r) \cdot \lambda_k}{(\lambda_k-\sigma)} \cdot (||v_1||_{r,\eta} + ||v_2||_{r,\eta} )^2 \cdot ||v_1-v_2||_{r,\sigma,\eta}^2. 
\end{align*}

The next is to estimate
\begin{align*}
	\int_0^{\infty} e^{2(\sigma-\eta)s}||\Pi_k D(s)||_{H^{r+1}}^2ds
\end{align*}
We take $\varepsilon = \frac{1}{2}$ in \eqref{energy-estimate;1}:
\begin{align}\label{energy-estimate;4}
	\frac{d}{d\tau}||v||_{H^{r}}^2 + ||v||_{H^{r+1}}^2\leq ||(\partial_{\tau} -L)v||_{H^{r-1}}^2
\end{align}
This implies that
\begin{align}
	\frac{d}{d\tau}\Big(e^{2(\sigma-\eta) \tau}||\Pi_k D||_{H^{r}}^2\Big) + e^{2(\sigma-\eta) \tau}||\Pi_k D||_{H^{r+1}}^2\leq e^{2(\sigma-\eta) \tau}||(\partial_{\tau} -L)\Pi_k D||_{H^{r-1}}^2 + 2(\sigma-\eta) e^{2(\sigma-\eta) \tau}||\Pi_k D||_{H^{r}}^2 
\end{align}
Integrating we get
\begin{align}
	e^{2(\sigma-\eta) T}||\Pi_k D(T)||_{H^{r}}^2 &+ \int_0^{T}e^{2(\sigma-\eta) s}||\Pi_k D(s)||_{H^{r+1}}^2ds \nonumber \\
	&\leq  ||\Pi_k D(0)||_{H^r}^2+ \int_0^{T}e^{2(\sigma-\eta) s}||(\partial_t -L)\Pi_k D(s)||_{H^{r-1}}^2 + 2\sigma  e^{2(\sigma-\eta) s}||\Pi_k D(s)||_{H^{r}}^2ds \nonumber
\end{align}
Dropping the leftmost term and taking $T\rightarrow \infty$ we get
\begin{align}\label{energy-estimate;5}
	\int_0^{\infty}e^{2(\sigma-\eta) s}||\Pi_k D||_{H^{r+1}}^2ds \leq  ||\Pi_k D(0)||_{H^r}^2+ \int_0^{\infty}e^{2(\sigma-\eta) s}||(\partial_t -L)\Pi_k D||_{H^{r-1}}^2 + 2\sigma  e^{2(\sigma-\eta) s}||\Pi_k D||_{H^{r}}^2ds
\end{align}

Note that $\Pi_k D(0) =0$, we get 
\begin{align*}
	\int_0^{\infty} e^{2(\sigma-\eta)s}||\Pi_k D(s)||_{H^{r+1}}^2ds\leq & \int_0^{\infty}e^{2(\sigma-\eta) s}\big\Vert N\Big((v_1(s)\Big) - N\Big(v_2(s)\Big) \big\Vert_{H^{r-1}}^2  + 2\sigma  e^{2(\sigma-\eta) s}||\Pi_k D(s)||_{H^{r}}^2ds\\
\end{align*}
Now we imitate the estimate for $\int_{0}^{\tau}  e^{2\sigma s} \big\Vert N\Big((v_1(s)\Big) - N\Big(v_2(s)\Big) \big\Vert_{H^{r-1}}^2  ds$ with $\tau$ replaced by $\infty$, we get
\begin{align}
	\int_0^{\infty}e^{2(\sigma-\eta) s}\big\Vert N\Big((v_1(s)\Big) - N\Big(v_2(s)\Big) \big\Vert_{H^{r-1}}^2 ds \leq&  C(r)\cdot (||v_1||_{r,\eta} + ||v_2||_{r,\eta} )^2 \cdot ||v_1-v_2||_{r,\sigma,\eta}^2     
\end{align}

Moreover, using the estimate we just did for $e^{2\sigma \tau}||\Pi_kD(\tau)||_{H^{r}}^2$ we get
\begin{align}
	\int_0^{\infty} 2\sigma  e^{2(\sigma-\eta) s}||\Pi_k D(s)||_{H^{r}}^2ds \leq& 2\sigma\sup\limits_{s\geq0} e^{2\sigma s}||\Pi_k D(s)||_{H^{r}}^2\cdot \int_0^{\infty}e^{-2\eta s}ds \nonumber\\
	\leq& \frac{C(r) \cdot \sigma \lambda_k}{(\lambda_k-\sigma)\cdot\eta}\cdot  (||v_1||_{r,\eta} + ||v_2||_{r,\eta} )^2 \cdot ||v_1-v_2||_{r,\sigma,\eta}^2 
\end{align}

Putting them together we get
\begin{align}
	||\Pi_k D||_{r,\sigma,\eta} \leq C(r)\cdot (1+\sigma^{\frac{1}{2}}\eta^{-\frac{1}{2}})\Big(\frac{ \lambda_k}{(\lambda_k-\sigma) }\Big)^{\frac{1}{2}} \cdot ||v_1-v_2||_{r,\sigma,\eta} \cdot \Big(||v_1||_{r,\eta}+||v_2||_{r,\eta}\Big)
\end{align}

\textbf{Estimating} $(1-\Pi_k)D$: 

Since $1-\Pi_k$ projects onto a finite dimensional space on which $H^r$ norm are all equivalent, this part will be similar to Strehlke's work \cite{str}. 
The following fact will be frequently used: for all $r\geq 1, k\geq 2$ and $h\in H^{r}$, we have
\begin{align}
    ||(1-\Pi_k)h||_{H^r} \leq C(n)\cdot \lambda_k^\frac{1}{2}||(1-\Pi_k)h||_{H^{r-1}}
\end{align}
Assume that $r\geq 2$. Then 
\begin{align*}
	&e^{\sigma \tau}||(1-\Pi_k) D(\tau)||_{H^r} \\
    \leq&  C(n)\cdot \lambda_k e^{\sigma \tau}||(1-\Pi_k) D(\tau)||_{H^{r-2}} \\
	 \leq&  C(n)\cdot \lambda_k e^{\sigma \tau}\int_{\tau}^{\infty}  e^{\lambda_{k-1}(s-\tau)}\big\Vert N\Big((v_1(s)\Big) - N\Big(v_2(s)\Big)\big\Vert_{H^{r-2}} ds \\
	 \leq& C(n,r)\cdot \lambda_k e^{\sigma \tau}\int_{\tau}^{\infty}  e^{\lambda_{k-1}(s-\tau)}|| v_1(s)-v_2(s)||_{H^{r}} \Big(||v_1(s)||_{H^r} + ||v_2||_{H^r}\Big)ds \\
	 \leq& C(n,r)\cdot \lambda_k e^{\sigma \tau}\int_{\tau}^{\infty}  e^{\lambda_{k-1}(s-\tau)-\sigma s}\cdot e^{\sigma s}|| v_1(s)-v_2(s)||_{H^{r}} \cdot \Big(||v_1(s)||_{H^r} + ||v_2||_{H^r}\Big)ds\\
	 \leq& C(n,r)\cdot \lambda_k e^{\sigma \tau}\Big(\sup\limits_{s\geq 0} e^{\sigma s}||v_1(s)-v_2(s)||_{H^{r}}\Big)\cdot \Big(\sup\limits_{s\geq 0} \Big(||v_1(s)||_{H^r}+||v_2(s)||_{H^{r}}\Big)\cdot\int_{\tau}^{\infty}  e^{\lambda_{k-1}(s-\tau)-\sigma s}  ds\\
	 \leq& C(n,r)\cdot \lambda_k e^{(\sigma-\lambda_{k-1}) \tau}\cdot ||v_1-v_2||_{r,\sigma,\eta} \cdot \Big(||v_1||_{r,\eta}+||v_2||_{r,\eta}\Big)\cdot\int_{\tau}^{\infty}  e^{(\lambda_{k-1} - \sigma)s}  ds\\
	 = & \frac{C(n,r)\cdot \lambda_k}{\sigma - \lambda_{k-1}}  \cdot  ||v_1-v_2||_{r,\sigma,\eta} \cdot \Big(||v_1||_{r,\eta}+||v_2||_{r,\eta}\Big)\\
\end{align*}
This implies that 
\begin{align}
	\sup\limits_{\tau\geq 0}e^{\sigma \tau}||(1-\Pi_k) D(\tau)||_{H^r} \leq  \frac{C(n,r) \cdot \lambda_k}{\sigma - \lambda_{k-1}}  \cdot  ||v_1-v_2||_{r,\sigma,\eta} \cdot \Big(||v_1||_{r,\eta}+||v_2||_{r,\eta}\Big)
\end{align}
Moreover we get 
\begin{align*}
	e^{(\sigma -\eta) \tau}||(1-\Pi_k) D(\tau)||_{H^{r+1}} \leq&    \lambda_k^{\frac{1}{2}} e^{(\sigma -\eta) \tau}||(1-\Pi_k) D(\tau)||_{H^{r}}\\
	\leq&   \frac{C(n,r) \cdot \lambda_k^{\frac{3}{2}}}{\sigma - \lambda_{k-1}}  \cdot  e^{-\eta \tau} \cdot||v_1-v_2||_{r,\sigma,\eta} \cdot \Big(||v_1||_{r,\eta}+||v_2||_{r,\eta}\Big)
\end{align*}
Integrating against $\tau$ we get
\begin{align}
	\int_0^\infty e^{2(\sigma -\eta) \tau}||(1-\Pi_k) D(\tau)||_{H^{r+1}}^2 d\tau \leq   \frac{C(n,r) \cdot \lambda_k^3}{(\sigma - \lambda_{k-1})^2\cdot\eta}   \cdot||v_1-v_2||_{r,\sigma,\eta}^2 \cdot \Big(||v_1||_{r,\eta}+||v_2||_{r,\eta}\Big)^2
\end{align}
Putting them together we get
\begin{align}
	||(1-\Pi_k)D||_{r,\sigma,\eta} \leq  \frac{C(n,r)\cdot \lambda_k (1+\eta^{-1/2}\cdot \lambda_k^{1/2})}{\sigma - \lambda_{k-1}} \cdot ||v_1-v_2||_{r,\sigma,\eta} \cdot \Big(||v_1||_{r,\eta}+||v_2||_{r,\eta}\Big)
\end{align} 
as desired.
\end{proof}

\bigskip
Let us now consider two solutions $v_0, v$. \textbf{Fix} $v_0$ and define
\begin{align*}
	w := v - v_0
\end{align*}
Then $w$ satisfies the equation 
\begin{align}\label{w;equation;1}
	\partial_{\tau} w = Lw + N_{v_0}(w, \nabla w, \nabla^2 w)
\end{align}
where $N_{v_0}$ is the error term whose structure depends on $v_0$. We shall abbreviate $N_{v_0}(w, \nabla w, \nabla^2 w)$ as $N_{v_0}(w)$.

Remark: Unlike $N$ which only contains quadratic and higher order terms, $N_{v_0}$ has a linear term. Its coefficients involve $v_0$ and $\nabla v_0$ and decay exponentially in our application. Thus $N_{v_0}$ behaves effectively as a superlinear error, which suffices for the subsequent analysis.

By definition, one can check that
\begin{align}\label{Nonlinear;2}
	N_{v_0}(w) = N(w + v_0) - N(v_0)
\end{align}
Therefore, $N_{v_0}$ coincides with $N$ when $v_0\equiv 0$.

For $q_0\in F_k$, $v_0\in X_{r,\eta}$ and $ w\in X_{r,\sigma,\eta}$}, define 
\begin{align}\label{Def;T-v0}
	T_{v_0}(w,q_0)\Big|_{\tau} = e^{L\tau}q_0 + \int_{0}^{\tau} e^{L(\tau-s)}\Pi_k[N_{v_0}(w(s))]ds - \int_{\tau}^{\infty} e^{L(\tau-s)}(1-\Pi_k)[N_{v_0}(w(s))]ds  
\end{align}
Again the integral converges, following the spirit of the proof of Theorem \ref{Thm;contraction-map;2}.

We show that $T_{v_0}$ is a contraction map:
\begin{theorem}\label{Thm;contraction-map;2}
For each $r > \frac{n}{2}+1 , k\geq 2$,  $\lambda_{k-1} < \sigma < \lambda_k$ and $\eta \in (0,\sigma)$, we can find $0 < \varepsilon < \delta \ll 1$ and constant $C=C(n,r,k,\sigma,\eta)>0$ with the following property. Suppose that $q_0\in F_k$ with $||q_0||_{H^r} < \varepsilon$ and $||v_0||_{r,\eta} < \varepsilon$, then $T_{v_0}$ is a contraction map of the ball of radius $\delta$ in $X_{r,\sigma,\eta}$ to itself with contraction factor $1/2$. Moreover, when $||w_1||_{r,\eta} \leq \delta$ and $||w_2||_{r,\eta}\leq \delta$, we have the following inequality:
\begin{align}
	||T_{v_0}(w_1,q_0) - T_{v_0}(w_2,q_0)||_{r,\sigma,\eta}\leq C\Big(||v_0||_{r,\eta} + ||w_1||_{r,\eta}+ ||w_2||_{r,\eta}\Big)\cdot ||w_1-w_2||_{r,\sigma,\eta}
\end{align}
Indeed, we can take $C = C_1=C(n,r)\Big[(1+\sigma^{\frac{1}{2}}\eta^{-\frac{1}{2}})\Big(\frac{\lambda_k}{\lambda_k-\sigma} \Big)^{\frac{1}{2}} + (1+\lambda_k^{\frac{1}{2}}\eta^{-\frac{1}{2}})\frac{\lambda_k}{\sigma - \lambda_{k-1}} \Big]$ from Theorem \ref{Thm;contraction-map;1}, then take $\delta = (10C_1+4)^{-1}$  and $\varepsilon = (6+2\sigma^{\frac{1}{2}}\eta^{-\frac{1}{2}})^{-1}\delta$.
\end{theorem}
\begin{proof}
	The proof is straightforward adaption of Theorem \ref{Thm;contraction-map;1}.
	Plugging \eqref{Nonlinear;2} into \eqref{Def;T-v0} and use \eqref{Def;T} we find that
\begin{align}
	T_{v_0}(w,q_0)  =& e^{L\tau}q_0 +\int_{0}^{\tau} e^{L(\tau-s)}\Pi_k[N_{}((w+v_0)(s)) - N(v_0(s))]ds  \\
    &- \int_{\tau}^{\infty} e^{L(\tau-s)}(1-\Pi_k)[N((w+v_0)(s)) - N(v_0(s))]ds \nonumber\\
	=& e^{L\tau}q_0+ \int_{0}^{\tau} e^{L(\tau-s)}\Pi_k [N((w+v_0)(s))]ds - \int_{\tau}^{\infty} e^{L(\tau-s)}(1-\Pi_k)[N((w+v_0)(s))]ds \nonumber\\
	&- \int_{0}^{\tau} e^{L(\tau-s)}\Pi_k [N(v_0(s))]ds + \int_{\tau}^{\infty} e^{L(\tau-s)}(1-\Pi_k)[N(v_0(s))]ds \nonumber\\
	=& T(w + v_0 , q_0) - \int_{0}^{\tau} e^{L(\tau-s)}\Pi_k [N(v_0(s))]ds + \int_{\tau}^{\infty} e^{L(\tau-s)}(1-\Pi_k)[N(v_0(s))]ds \nonumber
\end{align}
Note that the last two terms are independent of $w$, therefore
\begin{align}
	T_{v_0}(w_1,q_0) - T_{v_0}(w_2,q_0) = T(w_1 + v_0,q_0) - T(w_2 + v_0,q_0)
\end{align}
We also note that
\begin{align}
	||T_{v_0}(0,q_0)||_{r,\sigma,\eta} &= ||e^{L\tau}q_0||_{r,\sigma,\eta}  \nonumber\\
	& = \sup_{s \geq 0} e^{\sigma s}||e^{Ls}q_0||_{H^r} + \Big(\int_{0}^{\infty} e^{2(\sigma-\eta)s}||e^{Ls}q_0||_{H^{r+1}}^2ds \Big)^{\frac{1}{2}}.
\end{align}
Since $||e^{Ls}q_0||_{H^r}$ is bounded by $e^{-\lambda_k s}$, the first term is bounded by $||q_0||_{H^r}$. To estimate the second term, we use $q_0$ in place of $\Pi_k D$ in \eqref{energy-estimate;5} (or equivalently use $e^{Ls} q_0$ in place of $v$ in \eqref{energy-estimate;4} and integrate against the integration factor $e^{2(\sigma-\eta)s}$). Since $(\partial_{\tau} - L)(e^{L\tau} q_0) \equiv 0$, we  obtain: 
\begin{align}
	\int_{0}^{\infty} e^{2(\sigma-\eta)s}||e^{L s} q_0||_{H^{r+1}}^2ds \leq& ||q_0||_{H^r}^2 + \int_0^{\infty} 2\sigma  e^{2(\sigma-\eta) s}||e^{Ls}q_0||_{H^{r}}^2ds\\
	\leq& ||q_0||_{H^r}^2 \Big(1 + \sigma\eta^{-1}\Big)
\end{align}
Putting them together we have
\begin{align}\label{Linear-term;1}
	||T_{v_0}(0,q_0)||_{r,\sigma,\eta} \leq C_2||q_0||_{H^r}
\end{align}
where $C_2 = 2 + \sigma^{\frac{1}{2}}\eta^{-\frac{1}{2}}.$

Then  we can find $0 < \varepsilon < \delta$ such that  $(2C_2+2)\varepsilon \leq  \delta \leq (10C_1+4)^{-1}$, where $C_1$ is from Theorem \ref{Thm;contraction-map;1}. By assumption $||v_0||_{r,0}\leq ||v_0||_{r,\eta}\leq \varepsilon <\delta < 1/4$, then for all $w_1,w_2$  with $||w_i||_{r,\sigma, \eta} < \delta$ ($i=1,2$), one have $||w_i||_{r, \eta} <  ||w_i||_{r, \sigma, \eta} < \delta$ and thus Theorem \ref{Thm;contraction-map;1} can be applied:
\begin{align}\label{Ineq;15}
	||T_{v_0}(w_1,q_0) - T_{v_0}(w_2,q_0)||_{r,\sigma,\eta} \leq& C_1(2||v_0||_{r,\eta} + ||w_1||_{r,\eta} + ||w_2||_{r,\eta})\cdot ||w_1-w_2||_{r,\sigma,\eta} \nonumber\\
	\leq&  4C_1\delta ||w_1-w_2||_{r,\sigma,\eta}  \nonumber\\
	\leq& \frac{1}{2}||w_1-w_2||_{r,\sigma,\eta}
\end{align}

This proves that $T_{v_0}$ is a contraction mapping on ball of radius $\delta$ in $X_{r,\sigma, \eta}$ to itself with contraction factor $1/2$.

It remains to prove that $T_{v_0}$ maps the ball of radius $\delta$ in $X_{r,\sigma, \eta}$ to itself  when $||q_0||_{H^r}<\varepsilon$. To this end, for all $w$ with $||w||_{r,\sigma,\eta} < \delta$, the above inequality gives
\begin{align}
	||T_{v_0}(w,q_0) - T_{v_0}(0,q_0)||_{r,\sigma,\eta} \leq \frac{1}{2}||w||_{r,\sigma,\eta}
\end{align}
Using \eqref{Linear-term;1} and triangle inequality:
\begin{align}\label{Ineq;16}
	||T_{v_0}(w,q_0) ||_{r,\sigma,\eta}  \leq& ||T_{v_0}(0,q_0)||_{r,\sigma,\eta} + \frac{1}{2}||w||_{r,\sigma,\eta}\nonumber\\
	\leq&  C_2||q_0||_{H^r} + \frac{1}{2}\delta\nonumber\\
	\leq& C_2\varepsilon + \frac{1}{2}\delta \nonumber\\
	<& \frac{1}{2}\delta + \frac{1}{2}\delta = \delta
\end{align}
\eqref{Ineq;15} and \eqref{Ineq;16} implies that $T_{v_0}$  is a contraction map from the ball of radius $\delta$ in $X_{r,\sigma, \eta}$ to itself.
\end{proof}

The contraction property in Theorem \ref{Thm;contraction-map;2} allows us to find a unique fixed point of the equation 
\begin{align*}
	w = T_{v_0}(w,q_0)
\end{align*} 
This is equivalently to solving the equation \eqref{w;equation;1}. Therefore, taking $\varepsilon, \delta$ as in Theorem \ref{Thm;contraction-map;2} we obtain:

\begin{corollary}\label{Cor;existence;1}
	For each $r > \frac{n}{2}+1 , k\geq 2$,  $\lambda_{k-1} < \sigma < \lambda_k$ and $\eta \in (0,\sigma)$, there exists constants $\varepsilon, \delta$ with $0<\varepsilon < \delta$ such that for all $q_0\in F_k$ with $||q_0||_{H^r}\leq \varepsilon$ and $v_0\in X_{r,\eta}$ with $||v_0||_{r,\eta} \leq \varepsilon$, there is a unique solution $w$ to \eqref{w;equation;1}  with $||w||_{r,\sigma,\eta} < \delta$ and $\Pi_k w(0) = q_0$.
\end{corollary}

\subsection{Refined asymptotics}\label{subsec;refined-asymptotics}
Fix $r> \frac{n}{2} + 2, k\geq 2$ and let 
\begin{align*}
	\eta = \frac{\lambda_2}{4} = \frac{1}{4n}, \quad \sigma = {\lambda_k - \frac{\eta}{2}}
\end{align*}

\begin{remark}
The condition $r>\frac n2+2$ ensures the Sobolev embedding
$H^{r-2}\hookrightarrow C^0$.
\end{remark}

\begin{definition} The following constants will be used throughout this subsection:
	\begin{itemize}
	\item $C_1 = C(n,r)\Big[(1+\sigma^{\frac{1}{2}}\eta^{-\frac{1}{2}})\Big(\frac{\lambda_k}{\lambda_k-\sigma} \Big)^{\frac{1}{2}} + (1+\lambda_k^{\frac{1}{2}}\eta^{-\frac{1}{2}})\frac{\lambda_k}{\sigma - \lambda_{k-1}} \Big]$ is from Theorem \ref{Thm;contraction-map;1}.
	\item $C_2 = 2 +  \sigma^{\frac{1}{2}}\eta^{-\frac{1}{2}}$. 
    \item $C_3$ is defined in $\eqref{Nonlinear;1}$.
    \item $C_{E,k} = \lambda_k$, so that $||h||_{H^r} \leq C_{E,k} ||h||_{H^{r-2}}$ for all $h\in E_k$.
	\item $\delta = \min\{(10C_1+4)^{-1}, (32C_{E,k}C_2C_3)^{-1}\eta \}$
	\item $\varepsilon = (6+2\sigma^{\frac{1}{2}}\eta^{-\frac{1}{2}})^{-1}\delta = (2+2C_2)^{-1}\delta$
\end{itemize}
\end{definition}

Recall the equation \eqref{w;equation;1}:
\begin{align*}
	\partial_{\tau} w = Lw + N_{v_0}(w)
\end{align*}
Now let us fix an $v_0 \in X_{r,\eta}$ with $||v_0||_{r,\eta}\leq \varepsilon$. This condition can be fulfilled for any admissible $v_0$ by starting $v_0$ at a later time (i.e use $v_0'$ in place of $v_0$, where $v_0'(\tau) = v_0(\tau+T)$ for some $T > 0$), see the proof of Theorem \ref{Thm;prescribe-diff-asymptotic;2}.
By Corollary \ref{Cor;existence;1}, For any $q_0\in F_k$ with $||q_0||_{H^r}\leq \varepsilon$ there is a unique solution $w\in X_{r,\sigma,\eta}$ to \eqref{w;equation;1} with $||w||_{r,\sigma,\eta} < \delta$ and $\Pi_k w(0) =q_0$. In particular for all $\tau\geq 0$
\begin{align}
	||w(\tau)||_{H^{r}}\leq Ce^{-\sigma \tau}
\end{align}
\v{S}e\v{s}um \cite{se} (see also \cite{str}), we have $||v_0(\tau)|| \leq Ce^{-\lambda_2\tau}$. 
Then using \eqref{Nonlinear;2} and \eqref{Nonlinear;1}, we can estimate
\begin{align}\label{N-v0;estimate;1}
	||N_{v_0}(w)||_{H^{r-2}}\leq C(||v_0||_{H^{r}} + ||w||_{H^r}) ||w||_{H^r} \leq C e^{-(\lambda_2 + \sigma)\tau} + Ce^{-2\sigma\tau} \leq C e^{-(\lambda_k + \eta)\tau}.
\end{align}  
In the next proposition, we reveal the sharp decay rate of $w$:

\begin{proposition}\label{Prop;fine-asymptotic;1}
There exists a unique $Q_k\in E_k$  such that
\begin{align}\label{Ineq;fine-asymptotic;2}
	||e^{\lambda_k \tau}w(\tau) - Q_k||_{H^{r-2}} =   O(e^{-\eta \tau}) \quad  \text{ as } \tau\rightarrow\infty
\end{align}
Note that $Q_k$ can possibly be 0. In particular Sobolev embedding $H^{r-2}\rightarrow C^0$ implies \textbf{pointwise} convergence:
\begin{align}
	\lim\limits_{\tau \rightarrow\infty}e^{\lambda_k \tau}w(\tau) = Q_k 
\end{align}
\end{proposition}
\begin{proof}
 Throughout this proof, $C$ denote constants depending only on $n,k,r,\sigma, \eta$ that may vary from line to line. We suppress its explicit expression (such as $C_1$ in Theorem \ref{Thm;contraction-map;1}), since this level of detail is unnecessary for the current proposition.

We need to estimate $\pi_k w, \Pi_{k+1} w $ and $(1-\Pi_k) w$ in $H^{r-2}$ norm respectively:

(1) Estimate $\pi_k w$: 

Taking $L^2$ orthogonal projection of \eqref{w;equation;1} onto $E_k$:
\begin{align*}
	\partial_{\tau} \pi_k w = -\lambda_k \pi_k w + \pi_k N_{v_0}(w)
\end{align*}
This implies that
\begin{align}
 	\frac{d}{d\tau} e^{\lambda_k \tau}\pi_k w = e^{\lambda_k \tau}\pi_k N_{v_0}(w)
\end{align}
Then for all $0 \leq \tau_1 < \tau_2$ we have
\begin{align}
	||e^{\lambda_k \tau_2}\pi_k w(\tau_2) - e^{\lambda_k \tau_1}\pi_k w(\tau_1)||_{H^{r-2}}\leq& C\int_{\tau_1}^{\tau_2} e^{\lambda_k s}||N_{v_0}(w(s))||_{H^{r-2}}ds \nonumber\\
	\leq& C\int_{\tau_1}^{\tau_2} e^{-\eta s}ds \nonumber\\
	\leq& \frac{C}{\eta}e^{-\eta \tau_1} 
\end{align}
where we used \eqref{N-v0;estimate;1} in the second line.
This shows that $e^{\lambda_k \tau}\pi_k w(\tau)$ converges in $H^{r-2}$ to a polynomial $Q_k\in E_k$ with desired estimate.

(2) Estimate $\Pi_{k+1} w$:

By the equation \eqref{w;equation;1}, we get
\begin{align}
	||\Pi_{k+1}w||_{H^{r-2}}\frac{d}{d\tau}||\Pi_{k+1}w||_{H^{r-2}}  =\frac{1}{2}\frac{d}{d\tau}||\Pi_{k+1}w||_{H^{r-2}}^2 =& \langle \Pi_{k+1} w,Lw\rangle_{H^{r-2}} +  \langle \Pi_{k+1} w,N_{v_0}w\rangle_{H^{r-2}}\nonumber\\
	 =&-||\Pi_{k+1}w||_{H^{r-1}}^2 +  \langle \Pi_{k+1} w,N_{v_0}w\rangle_{H^{r-2}} \nonumber \\
	 \leq& -\lambda_{k+1} ||\Pi_{k+1}w||_{H^{r-2}}^2 + || \Pi_{k+1}w ||_{H^{r-2}}\cdot  || N_{v_0}w||_{H^{r-2}}
\end{align}
Using \eqref{N-v0;estimate;1} and cancelling a $|| \Pi_{k+1}w ||_{H^{r-2}}$ factor gives:
\begin{align}
	\frac{d}{d\tau}||\Pi_{k+1}w||_{H^{r-2}} +\lambda_{k+1} ||\Pi_{k+1}w||_{H^{r-2}} \leq  || N_{v_0}w ||_{H^{r-2}} \leq C e^{-(\lambda_k + \eta)\tau} \nonumber
\end{align}
Integrating the above we obtain: 
\begin{align}
	e^{\lambda_{k+1}\tau}||\Pi_{k+1}w(\tau)||_{H^{r-2}} \leq  ||\Pi_{k+1}w(0)||_{H^{r-2}} + C\int_0^{\tau} e^{(\lambda_{k+1} - \lambda_k - \eta)s}ds \leq C\cdot (1 + e^{(\lambda_{k+1} - \lambda_k - \eta)\tau}) \nonumber
\end{align}
This implies that
\begin{align}
	||\Pi_{k+1}w||_{H^{r-2}}  = O(e^{-\lambda_{k+1}\tau} + e^{-(\lambda_k + \eta)\tau}) \nonumber
\end{align}
However, by the explicit formula \eqref{eigenvalue;lambda-k;1},  $\lambda_{k+1} =\lambda_k +\frac{2k+n}{2n} > \lambda_k + \eta$. Therefore, $e^{-\lambda_{k+1}\tau}$ is negligible and we obtain:
\begin{align}
	||\Pi_{k+1}w||_{H^{r-2}}  = O(e^{-(\lambda_k + \eta)\tau})
\end{align}

(3) Estimate $(1-\Pi_k)w$:

By the equation \eqref{w;equation;1} we get:
\begin{align}
	||(1-\Pi_k)w||_{L^2}\frac{d}{d\tau}||(1-\Pi_k)w||_{L^2} =& \frac{1}{2}\frac{d}{d\tau}||(1-\Pi_k)w||_{L^2}^2 = \langle (1-\Pi_{k}) w,Lw\rangle_{L^{2}} +  \langle (1-\Pi_{k})w,N_{v_0}w\rangle_{L^{2}}\nonumber\\
	 \geq& -\lambda_{k-1} ||(1-\Pi_{k})w||_{L^{2}}^2 - || (1-\Pi_{k})w ||_{L^{2}}\cdot  || N_{v_0}w||_{L^{2}}
\end{align}
Using \eqref{N-v0;estimate;1} and cancelling a $||(1- \Pi_{k})w ||_{L^{2}}$ factor gives
\begin{align}
	\frac{d}{d\tau}||(1-\Pi_k)w||_{L^2} + \lambda_{k-1} ||(1-\Pi_{k})w||_{L^{2}} \geq  -|| N_{v_0}w||_{L^{2}} \geq -Ce^{-(\lambda_k + \eta)\tau} \nonumber
\end{align}
Integrating the above inequality from $\tau$ to $\tau'$, where $\tau<\tau'$: 
\begin{align}
	e^{\lambda_{k-1}\tau'}||(1-\Pi_k)w(\tau')||_{L^2} \geq& e^{\lambda_{k-1}\tau}||(1-\Pi_k)w(\tau)||_{L^2} - \int_{\tau}^{\tau'} Ce^{(\lambda_{k-1}-\lambda_k - \eta)s}ds \nonumber\\
	\geq& e^{\lambda_{k-1}\tau}||(1-\Pi_k)w(\tau)||_{L^2} - C e^{(\lambda_{k-1}-\lambda_k - \eta)\tau} \nonumber
\end{align}
Since $w\in X_{r,\sigma,\eta}$, we have decay estimate $||w(\tau')||_{L^2}\leq C||w(\tau')||_{H^r}\leq C e^{-\sigma \tau'}$. Because $\sigma>\lambda_{k-1}$, the left-hand side converges to $0$ as
$\tau'\rightarrow\infty$. This implies that
\begin{align}
	||(1-\Pi_k)w(\tau)||_{L^2} \leq C e^{-(\lambda_k + \eta)\tau} \nonumber
\end{align}
Finally, the image of $1-\Pi_k$ is in a finite dimensional space, hence the $L^2$ norm is equivalent to $H^{r-2}$ norm:
\begin{align}
	||(1-\Pi_k)w(\tau)||_{H^{r-2}} = O (e^{-(\lambda_k + \eta)\tau})
\end{align}
If we put together (1)-(3), the proposition follows.

\end{proof}

Based on the above result, we can make sense the following definition:
\begin{definition}\label{Def;w-q0;asymptotic}
Fix any $v_0$ with $||v_0||_{r,\eta}<\varepsilon$. Then for any $q_0\in F_k$ with $||q_0||_{H^r}\leq \varepsilon$, there is a unique solution $w(t,q_0)\in X_{r,\sigma,\eta}$ to \eqref{w;equation;1} with $||w(\cdot,q_0)||_{r,\sigma,\eta} \le \delta$ and $\Pi_k w(0,q_0) =q_0$. We define
\begin{align}
	F(q_0) := \lim\limits_{t\rightarrow\infty}e^{\lambda_k \tau}w(\tau,q_0) \in E_k.
\end{align}
Note that $F(0) =0$ since $w(\cdot,0)\equiv 0$.

\end{definition}

\begin{lemma}\label{Lem;Diff-w;Diff-q}
	For $q_1,q_2\in F_k $ with $||q_i||_{H^r}\leq \varepsilon$ ($i=1,2$), we have
	\begin{align*}
		||w(\cdot, q_1)-w(\cdot,q_2)||_{r,\sigma,\eta} \leq { 2C_2}||q_2 - q_1||_{H^r}.
	\end{align*}
\end{lemma}
\begin{proof}
Since $w(\cdot, q_i) = T_{v_0}(w(\cdot,q_i), q_i)$ we find that
\begin{align*}
	w(\cdot,q_1) - w(\cdot ,q_2) =  T_{v_0}(w(\cdot, q_1),q_1) - T_{v_0}(w(\cdot, q_2),q_2) 
\end{align*} 
Moreover, by examining the definition of $T_{v_0}$ \eqref{Def;T-v0}, we can write
\begin{align*}
	T_{v_0}(w(\cdot,q_2), q_2) = T_{v_0}(w(\cdot,q_2), q_1) - e^{L\tau}q_1 + e^{L\tau}q_2
\end{align*}
Putting them together we get: 
	\begin{align*}
		 w(\cdot ,q_1) - w(\cdot ,q_2) = e^{L\tau}(q_1 - q_2) + T_{v_0}(w(\cdot, q_1),q_1) - T_{v_0}(w(\cdot, q_2),q_1) 
	\end{align*}
	Similar to \eqref{Linear-term;1}, we have 
	\begin{align}
		||e^{L\tau}(q_2 - q_1)||_{r,\sigma,\eta} \leq  C_2||q_2-q_1||_{H^r}
	\end{align}
	Then we can apply Theorem \eqref{Thm;contraction-map;2}:
	\begin{align}
		||w(\cdot,q_1) - w(\cdot , q_2)||_{r,\sigma,\eta} \leq& ||e^{L\tau}(q_2 - q_1)||_{r,\sigma,\eta} + ||T_{v_0}(w(\cdot, q_1),q_1) - T_{v_0}(w(\cdot, q_2),q_1)||_{r,\sigma,\eta}\nonumber\\
		\leq&  C_2||q_2-q_1||_{H^r} + \frac{1}{2} ||w(\cdot,q_1) - w(\cdot,q_2)||_{r,\sigma,\eta}
	\end{align}
	The assertion then follows.
\end{proof}

\begin{lemma}\label{Lem;Fk;asymptotic}
	For any $q\in F_k$, $e^{(\lambda_k + L)t}q$ converges to $\pi_k 	q$ in $H^l$ and $C^l$ for all $l\geq 0$. 
\end{lemma}

\begin{proposition}\label{Prop;Lipschitz}
	$F$ is Lipschitz continuous on the $\varepsilon$-ball (in the $H^r$ norm) of $F_k$. Moreover, $F-Id$ is a contraction map with factor $\frac{1}{2}$ on the $\varepsilon$-ball (in the $H^r$ norm) of $E_k$.
\end{proposition}
\begin{remark}
	We may possibly decrease $\varepsilon,\delta$ to get $\frac{1}{2}$ contraction factor.
\end{remark}	 
\begin{proof}
Now we consider $q_1,q_2\in F_k$ with $||q_i||_{H^r}\leq \varepsilon$ ($i=1,2$).

By definition, 
\begin{align}
	w(t,q_1) = e^{L\tau}q_1 + \int_{0}^{\tau} e^{L(\tau-s)}\Pi_k[N_{v_0}(w(s,q_1))]ds - \int_{\tau}^{\infty} e^{L(\tau-s)}(1-\Pi_k)[N_{v_0}(w(s,q_1))]ds  
\end{align}	 

Following idea in Strehlke's paper \cite[Proposition 4.9]{str}
\begin{align}
	e^{\lambda_k \tau}w(\tau,q_1) = e^{(\lambda_k + L)\tau}q_1 + \int_{0}^{\tau} e^{\lambda_k \tau + L(\tau-s)}\Pi_k[N_{v_0}(w(s,q_1))]ds - \int_{\tau}^{\infty} e^{\lambda_k \tau + L(\tau-s)}(1-\Pi_k)[N_{v_0}(w(s,q_1))]ds  
\end{align}	 
Note that
\begin{align*}
	\lambda_k \tau + L(\tau-s) = (\lambda_k + L )(\tau-s) + \lambda_k s
\end{align*}
Thus we can rewrite the above as
\begin{align}
	&e^{\lambda_k \tau}w(\tau,q_1) - e^{(\lambda_k + L)\tau}q_1 \nonumber\\
    =& \int_{0}^{\tau} e^{\lambda_k s} e^{(\lambda_k + L) (\tau-s)}\Pi_k[N_{v_0}(w(s,q_1))]ds - \int_{\tau}^{\infty} e^{\lambda_k s} e^{(\lambda_k + L) (\tau-s)}(1-\Pi_k)[N_{v_0}(w(s,q_1))]ds  
\end{align}

We shall now consider difference of two solution $w(\cdot,q_1)$ and $w(\cdot, q_2)$:
\begin{align}
	&\Big(e^{\lambda_k \tau}w(\tau,q_1) - e^{\lambda_k \tau}w(\tau,q_2)\Big) - \Big( e^{(\lambda_k + L)\tau}q_1 -  e^{(\lambda_k + L)\tau}q_2\Big) \nonumber\\
	=& \int_{0}^{\tau} e^{\lambda_k s} e^{(\lambda_k + L) (\tau-s)}\Pi_k[N_{v_0}(w(s,q_1)) - N_{v_0}(w(s,q_2))]ds \nonumber\\
	&- \int_{\tau}^{\infty} e^{\lambda_k s} e^{(\lambda_k + L) (\tau-s)}(1-\Pi_k)[N_{v_0}(w(s,q_1)) - N_{v_0}(w(s,q_2))]ds  
\end{align}

For all $l\geq 0$, note that $e^{(L+\lambda_k)(\tau-s)}$ is $H^l$ non-increasing either on $F_k$ when $s\leq \tau$, or on $F_k^{\perp} $ when $s\geq \tau$, then we can estimate the above by
\begin{align}
	&\Big\Vert \Big(e^{\lambda_k \tau}w(\tau,q_1) - e^{\lambda_k \tau}w(\tau,q_2)\Big) - \Big( e^{(\lambda_k + L)\tau}q_1 -  e^{(\lambda_k + L)\tau}q_2\Big)\Big\Vert_{H^{r-2}} \nonumber\\
	\leq& \int_{0}^{\tau} e^{\lambda_k s} ||\Pi_k[N_{v_0}(w(s,q_1)) - N_{v_0}(w(s,q_2))]||_{H^{r-2}}ds \nonumber \\
    &+ \int_{\tau}^{\infty} e^{\lambda_k s} ||(1-\Pi_k)[N_{v_0}(w(s,q_1)) - N_{v_0}(w(s,q_2))]||_{H^{r-2}}ds\\
	\leq &  \int_{0}^{\infty} e^{\lambda_k s} ||N_{v_0}(w(s,q_1)) - N_{v_0}(w(s,q_2))||_{H^{r-2}}ds \nonumber\\
	=&  \int_{0}^{\infty} e^{\lambda_k s} ||N(w(s,q_1)+v_0(s)) - N(w(s,q_2)+v_0(s))||_{H^{r-2}}ds \nonumber\\
	\leq& C_3 \int_{0}^{\infty} e^{\lambda_k s} ||w(s,q_1) - w(s,q_2)||_{H^r}(2||v_0(s)||_{H^r} + ||w(s,q_1)||_{H^r} + ||w(s,q_2)||_{H^r})ds 
\end{align}

By Lemma \ref{Lem;Diff-w;Diff-q}, for all $s\geq 0$:
\begin{align*}
	||w(s,q_1) - w(s,q_2)||_{H^r}\leq 2C_2 e^{-\sigma s} ||q_1-q_2||_{H^r}
\end{align*}
By setup and definition we also have for $s\geq 0$:
\begin{align*}
	2||v_0(s)||_{H^r} + ||w(s,q_1)||_{H^r} + ||w(s,q_2)||_{H^r}\leq  (2\varepsilon + 2\delta)e^{-\eta s}\leq 4\delta e^{-\eta s}
\end{align*}
Then
\begin{align}\label{Ineq;17}
	&\Big\Vert \Big(e^{\lambda_k \tau}w(\tau,q_1) - e^{\lambda_k \tau}w(\tau,q_2)\Big) - \Big( e^{(\lambda_k + L)\tau}q_1 -  e^{(\lambda_k + L)\tau}q_2\Big)\Big\Vert_{H^{r-2}} \nonumber\\
	\leq&  8C_2C_3\delta \int_{0}^{\infty} e^{\lambda_k s} e^{-\sigma s}||q_1-q_2||_{H^r} e^{-\eta s}ds \nonumber\\
	\leq&  16C_2C_3\eta^{-1}\delta ||q_1-q_2||_{H^r}
\end{align}
Since $(\lambda_k + L)$ is $H^l$ non-increasing on $F_k$ for all $l\geq 0$, we have 
\begin{align*}
	||e^{(\lambda_k + L)\tau}q_1 -  e^{(\lambda_k + L)\tau}q_2 ||_{H^{r-2}}\leq ||q_1 - q_2||_{H^{r-2}}\leq C||q_1 - q_2||_{H^{r}}
\end{align*} 
Putting them together:
\begin{align*}
	&|| e^{\lambda_k \tau}w(\tau,q_1) - e^{\lambda_k \tau}w(\tau,q_2)||_{H^{r-2}} \\
    \leq& ||e^{(\lambda_k + L)\tau}q_1 -  e^{(\lambda_k + L)\tau}q_2 ||_{H^{r-2}} +  16C_2C_3 \eta^{-1}\delta ||q_1-q_2||_{H^r} \leq C||q_1-q_2||_{H^r}
\end{align*}
From this and applying Sobolev embedding $H^{r-2}\rightarrow C^0$, we can reach our conclusion:
\begin{align*}
	||F(q_1) - F(q_2)||_{H^r} \leq &  C_{E,k}||F(q_1) - F(q_2)||_{H^{r-2}}\\
	 =& C_{E,k}\lim_{\tau\rightarrow\infty}|| e^{\lambda_k \tau}w(\tau,q_1) - e^{\lambda_k \tau}w(\tau,q_2)||_{H^{r-2}}\\
	 \leq& C ||q_1-q_2||_{H^r}
\end{align*}

This proves that $F$ is a Lipschitz map.

If $q_1, q_2\in E_k$ with $||q_i||_{H^r}\leq \varepsilon$, this means $q_i =\pi_k q_i$ ($i=1,2$). 
We apply Lemma \ref{Lem;Fk;asymptotic} to get:
\begin{align}
	(F-Id)(q_1)  - (F-Id)(q_2)  =&  F(q_1) - F(q_2) - (q_1 - q_2) \nonumber\\
	=& \lim_{\tau\rightarrow \infty}\Big(e^{\lambda_k \tau}w(\tau,q_1) - e^{\lambda_k \tau}w(\tau,q_2)\Big) - \Big( e^{(\lambda_k + L)\tau}q_1 -  e^{(\lambda_k + L)\tau}q_2\Big) 
\end{align}
By \eqref{Ineq;17}, Sobolev norm equivalence on $E_k$ and the choice of $\delta$, we conclude
\begin{align}
	||(F-Id)(q_1)  - (F-Id)(q_2)||_{H^r} \leq  16C_{E,k}C_2C_3 \eta^{-1}\delta ||q_1-q_2||_{H^r} \leq \frac{1}{2} ||q_1-q_2||_{H^r} 
\end{align}
 
This proves that $F-Id$ is a contraction map with contraction factor $\frac{1}{2}$ on $\varepsilon$-ball (in the $H^r$ norm) of $E_k$.  
\end{proof}

\bigskip

 The following Theorem allows us to prescribe asymptotics of $w$, the difference between solutions.

\begin{theorem}\label{Thm;prescribe-diff-asymptotic;1}
Given  a solution $v_0$ to \eqref{v;equation;1} with $||v_0||_{r,\eta} < \varepsilon$, for any $Q_k\in E_k$ with $||Q_k||_{H^r}\leq \varepsilon/3$, there exists a unique $q_k\in E_k$ with $||q_k||_{H^r}\leq \varepsilon$ and a unique solution $w$ of \eqref{w;equation;1} with $||w||_{r,\sigma,\eta} \le \delta$ such that $\Pi_k w(0) = q_k$ and
\begin{align}
	|e^{\lambda_k \tau}w - Q_k|_{C^0(\mathbb{S}^n)} =  O(e^{-\eta \tau})   \quad  \text{ as } \tau\rightarrow\infty
\end{align}
\end{theorem}
\begin{proof}
	Let us define $G: E_k \cap H^r|_{B_{\varepsilon}}\rightarrow E_k$ by
	\begin{align*}
		G(a):=  Q_k - (F(a) - a)
	\end{align*}
	By Proposition \ref{Prop;Lipschitz}, we know that $G$ is a contraction map with factor $\frac{1}{2}$. 
	
	Since $F(0) = 0$, we have $G(0) = Q_k$. Therefore, for all $||a||_{H^r} < \varepsilon$ 
	\begin{align*}
		||G(a)||_{H^r}\leq ||G(0)||_{H^r}  + \frac{1}{2}||a||_{H^r}\leq \frac{\varepsilon}{3} + \frac{\varepsilon}{2} < \varepsilon
	\end{align*}
	This means that $G$ maps  $\varepsilon$-ball (in $H^r$ norm) of $E_k$ to itself. In fact the image of $G$ is in a $\frac{5\varepsilon}{6}$-ball. By contraction property, there must be a unique fixed point $q_k$.
	
	This means  $G(q_k) = q_k$, thus $F(q_k) = Q_k$.   In other words, $|e^{\lambda_k \tau}w(\tau,q_k) - Q_k|\rightarrow 0$ as $\tau\rightarrow\infty$.
	Moreover, Proposition \ref{Prop;fine-asymptotic;1} gives finer asymptotics:
	\begin{align*}
		||e^{\lambda_k \tau}w(\tau,q_k) - Q_k||_{H^{r-2}}\leq Ce^{-\eta \tau} 
	\end{align*}
	By Sobolev embedding we would also have the  $C^0$ inequality:
	\begin{align*}
		|e^{\lambda_k \tau}w(\tau,q_k) - Q_k|_{C^0(\mathbb{S}^n)} \leq Ce^{-\eta \tau}  
	\end{align*}
	The norm bound $||w||_{r,\sigma,\eta}\leq\delta$ follows directly from Corollary \ref{Cor;existence;1}.
	 
\end{proof}
\begin{remark}
	We don't have to use contraction property to find fixed point $q_k$. Indeed, one can use Brouwer fixed point theorem because $E_k$ is finite dimensional.
	\end{remark}

We  show that the smallness assumption in the Theorem \ref{Thm;prescribe-diff-asymptotic;1}  are unnecessary and arrive in the main result of this section:
\begin{theorem*}[Restate of Theorem \ref{Thm;prescribe-diff-asymptotic;2}]
	Fix any $r> \frac{n}{2}+2, k\geq2$, given  an admissible solution $v_0$ to \eqref{v;equation;1} on $[\tau_0,\infty)$ and given $Q_k\in E_k$, set $\eta = 1/(4n)$. There exists a constant $C>0$, $\tau_1\geq \tau_0$ and another admissible solution $v_1$ of \eqref{v;equation;1} define on $[\tau_1,\infty)$ such that 
\begin{align}
	|e^{\lambda_k \tau}(v_1-v_0) - Q_k|_{C^0(\mathbb{S}^n)} = O( e^{-\eta \tau}).
\end{align}
\end{theorem*}
\begin{proof}
	If $Q_k = 0$, then we can simply take $v_1 = v_0$. So we may assume that $Q_k\neq 0$.
	
	The main idea is simply "shifting left and then right", because such operation will dilate the asymptotics and allows us to use Theorem \ref{Thm;prescribe-diff-asymptotic;1}. Now we describe the idea in detail.
	 
	By Lemma \ref{Lem;admissible-norm}, we can let $||v_0(\cdot, \cdot+\tau_0)||_{r,\eta} + ||v_0(\cdot,\cdot+\tau_0)||_{r+1,\eta} + ||Q_k||_{H^r}= R >0$. 
	Let 
	\begin{align*}
		\tau_1 =& |\tau_0| +  \eta^{-1}\ln(1 + 2R\varepsilon^{-1} \eta^{-1/2}) + \sigma^{-1} \ln(1+2R\varepsilon^{-1}) + \lambda_k^{-1}\ln(1+ 3R\varepsilon^{-1})> \tau_0\\
		\hat{Q}_k =& e^{-\lambda_k \tau_1} Q_k\\
		\hat{v}_0(\tau) =& v_0(\tau+\tau_1)
	\end{align*}
	Then $||v_0(s)||_{H^r} $ and $||v_0(s)||_{H^{r+1}}$ are both bounded by $ Re^{-\eta (s-\tau_0)}$. Hence
	\begin{align*}
		||\hat{v}_0||_{r,\eta} =& \Big(\int_{0}^{\infty} ||v_0(s+\tau_1)||^2_{H^{r+1}}ds\Big)^{1/2} + \sup\limits_{s\geq 0}e^{\eta s}||v_0(s+\tau_1)||_{H^r}\\
		=& \Big(\int_{\tau_1}^{\infty} ||v_0(s)||^2_{H^{r+1}}ds\Big)^{1/2} + e^{\eta(\tau_0 -\tau_1)}\cdot\sup\limits_{s\geq \tau_0}e^{\eta (s-\tau_0)}||v_0(s)||_{H^r}\\
		\leq& \Big(\int_{\tau_1}^{\infty} R^2e^{-2\eta (s-\tau_0)}ds\Big)^{1/2} + R e^{ -\eta (\tau_1-\tau_0)}\\
		\leq& R\eta^{-1/2}e^{-\eta (\tau_1 - \tau_0)} + Re^{-\eta(\tau_1 - \tau_0)}<\varepsilon
	\end{align*}

	We can also check that
	\begin{align*}
		||\hat{Q}_k||_{H^r}< \frac{||Q_k||_{H^r}}{3R\varepsilon^{-1}} \leq  \frac{\varepsilon}{3}
	\end{align*}

	Then we can use Theorem \ref{Thm;prescribe-diff-asymptotic;1} to get $w$ defined on $[0,\infty)$ satisfying \eqref{w;equation;1}:
	\begin{align*}
		\partial_{\tau} w = Lw + N_{\hat{v}_0}(w)
	\end{align*}
	and
	\begin{align}
	|e^{\lambda_k \tau}w - \hat{Q}_k|_{C^0(\mathbb{S}^n)}\leq C e^{-\eta \tau}
	\end{align}
	Then  
	\begin{align*}
		\hat{v}_1 = \hat{v}_0 + w
	\end{align*} 
	is defined on $[0,\infty)$ and satisfies \eqref{v;equation;1}. 	
	This in turn gives
	\begin{align*}
		|e^{\lambda_k \tau}(\hat{v}_1 - \hat{v}_0) - \hat{Q}_k|_{C^0(\mathbb{S}^n)}\leq C e^{-\eta \tau}
	\end{align*} 
	By Theorem \ref{Thm;prescribe-diff-asymptotic;1}, $\hat{v}_1$ has bounded $||\cdot||_{r,\eta}$ norm, thus is admissible by Lemma \ref{Lem;admissible-norm}.

	Finally we define $v_1(\tau) = \hat{v}_1(\tau-\tau_1)$ on $\tau\in [\tau_1,\infty)$, which is also admissible. We now check that $v_1$ is the desired function:
	\begin{align*}
		\big|e^{\lambda_k \tau}(v_1-v_0)(\tau) -Q_k\big|_{C^0(\mathbb{S}^n)} =& \big|e^{\lambda_k \tau}\hat{v}_1(\tau-\tau_1) - e^{\lambda_k \tau}\hat{v}_0(\tau-\tau_1) - e^{\lambda_k \tau_1}\hat{Q}_k\big|_{C^0(\mathbb{S}^n)}\\
		=& e^{\lambda_k \tau_1}\big|e^{\lambda_k (\tau-\tau_1)}(\hat{v}_1 - \hat{v}_0)(\tau-\tau_1) - \hat{Q}_k\big|_{C^0(\mathbb{S}^n)}\\
		\leq&  Ce^{(\lambda_k+\eta) \tau_1} \cdot e^{-\eta \tau}  
	\end{align*}
\end{proof}

\section{Construction of nonsmooth arrival time functions}\label{Sec;main-theorem-proof}

We now assemble the ingredients developed in the previous sections in order to prove the main theorems. Sections 4 and 5 show that higher regularity of the arrival time function imposes nonlinear algebraic obstruction equations on the asymptotic coefficients of the associated rescaled mean curvature flow.  Section 6, on the other hand, provides a mechanism for prescribing higher-order asymptotic modes of the rescaled flow.

Now we explain the ideas of the proof briefly. The proof is by contradiction. Assume there are too few solutions with low regularity. Then almost every higher-mode perturbation must remain smooth. Their asymptotic coefficients would have to satisfy the obstruction equations.  We choose the prescribed asymptotic data so that these equations are violated. Hence we obtain a contradiction.

We first summarize results from the previous sections. Let us begin with $\RR^{n+1}$ for all $n\geq 1$. Define 
\begin{align*}
	\cC\cA^{N}_n = \Big\{U\in C^{N}(\Omega)\big| \  U \text{ satisfies } \eqref{arrival-time-eqn-1},  U = const \text{ on } \partial \Omega, \ \Omega \subset \RR^{n+1} \text{ is bounded, open and convex} \Big\}
\end{align*} 
to be the set of all $C^{N}$ arrival time functions in bounded convex domain in $\RR^{n+1}$ satisfying Dirichlet boundary condition.

Recall that the arrival time function on any bounded convex domain is in $\cC\cA^{2}_n$ for $n\geq 2$ and in $\cC\cA^{3}_1$ when $n=1$. From now on, let us always assume that $N \geq 2$ in this section.

For any  $U\in \cC\cA^{N}_n$, it canonically  corresponds to a convex MCF $\{M_t\}_{t\in[t_0,T_0)}$ with a spherical singularity at $(x_0,T_0)$, where $x_0, t_0, T_0$ can be determined by $t_0 = \inf_{\Omega}\{U\} = U(\partial \Omega), T_0 = \sup_{\Omega}\{U\}$ and $x_0$ is the unique maximum point of $U$.    We define  the corresponding RMCF of $U$ or $\{M_t\}$ to be the RMCF  of $\{M_t\}$ center $(x_0,T_0)$ specified by \eqref{Def;RMCF;1}, denoted by  $\{\bar{M}_{\tau}\}$. Note that this forces $\bar{M}_{\tau}$ to start at $-\ln(T_0-t_0)$.  Then consider the  graph function $v$ of $\{\bar{M}_{\tau}\}$ over $\mathbb{S}^n(\sqrt{2n})$ as specified by \eqref{Def;RMCF-graph;1}. 

Conversely, given a RMCF $\{\bar{M}_{\tau}\}$, we define the corresponding MCF centered at $(x_0,T_0)$ to be 
\begin{align}\label{RMCF-MCF;2}
	M_t := x_0 + \sqrt{T_0 - t}\bar{M}_{-\ln(T_0-t)}
\end{align}
Then we can canonically corresponds an arrival time function $U$ to $\{M_t\}$. This definition allows the correspondence to be inverse to each other, namely
\begin{proposition}
	Given a convex MCF $\{M_t\}$ with spherical singularity at $(x_0, T_0)$, let the corresponding RMCF centered at $(x_0,T_0)$ be $\bar{M}_{\tau}$. Then the corresponding MCF of $\{\bar{M}_{\tau}\}$ centered at $(x_0, T_0)$ is $\{M_t\}$. 
\end{proposition}

For every solution produced below, increase the initial rescaled time and shrink the fixed-point ball, if necessary, so that the radial graph is uniformly $C^2$-close to the round sphere. It is then embedded and strictly convex; formula \eqref{RMCF-MCF;2} gives a smooth compact strictly convex flow, and its nested enclosed bodies define the corresponding arrival time.

If $U\in C^{N}$ for some $N\geq 2$, then by Section \ref{Sec;elliptic} we can write its Taylor expansion as
\begin{align}\label{Taylor;4}
	T_0 - U(x - x_0) = \frac{|x-x_0|^2}{2n} + \sum_{k=3}^{N}\frac{P_k(x-x_0)}{n} + o(|x-x_0|^{N})
\end{align} 
where $P_k$ are homogeneous degree $k$ polynomials.

By Theorem \ref{Taylor-to-asymptotic} we have the unique asymptotic expansion of $v$ as $\tau\rightarrow \infty$ :
\begin{align}\label{v;asymp;5}
	\Big| v(\cdot, \tau) - \sum^{N-2}_{l=1}f_l e^{-\frac{l \tau}{2} }\Big|_{C^0(\mathbb{S}^n(\sqrt{2n}))} = o(e^{-\frac{N-2}{2}\tau})
\end{align}
for a collection of smooth functions $f_1,...,f_{N-2}$ on $\mathbb{S}^n(\sqrt{2n})$.   For notation consistency, we manually set 
\begin{align*}
	P_0 = P_1 = P_2 \equiv 0,\quad f_0 \equiv \sqrt{2n}
\end{align*}
The collection $f_1,\cdots, f_{N-2}$ are called the \textbf{asymptotic functions} of $U$.

By Corollary \ref{fm;Lemma;2}, $f_l$ are additions and multiplications of $\frac{P_3}{|x|^2},....\frac{P_{N}}{|x|^2}$ and thus are smooth functions on $\mathbb{S}^n(\sqrt{2n})$ for each $l=1,2,...,N-2$. Moreover, $f_l$ is the restriction of degree $l$ homogeneous polynomial on $\mathbb{S}^n(\sqrt{2n})$  when $n=1$ by Corollary \ref{Cor;divisble}. Indeed this is the case for all $n\geq 2$ but we do not prove it as it will not be used in the current paper.

Recall that $\lambda_j = \frac{j(j+n-1)}{2n}-1$. Let us take all setups from \ref{subsec;prescribe;setup} in Section \ref{Sec;prescribe-higher-asymptotic}. Specially we need to adapt the notation of $E_j, \pi_j$, the $L^2$ inner product, and the $L^2$ norm on $\mathbb{S}^n({\sqrt{2n}})$. Recall that $E_j$ is the eigenspace of $\lambda_j$ and $\pi_j$ is the $L^2$-orthogonal projection onto $E_j$.

\begin{definition}\label{Def;asymp-distinct}
Given $U\in \cC\cA^{2}_n$ with a spherical singularity at $(x_0,T_0)$, a rigid motion of $U$ is $\tilde{U}(x) =  U(R(x- p)) + c$ for some $p\in \RR^{n+1}$, $R\in O(n+1)$ and $c\in \RR$. In the special case of $R = Id$  (i.e no rotation), we call it a translation.  A truncation of $U$ at $t_1\in [0,T_0)$ is $U$ restricted on $\Omega_{t_1} = \{x\in \RR^{n+1} | U(x)\geq t_1 \}$. A rescaling of $U$ is $U^{\lambda} = \lambda^2 U(\frac{x-x_0}{\lambda})$ for some $\lambda > 0$. We say that two arrival time functions $U_1, U_2 \in \cC\cA^{2}_n$ are \textbf{asymptotically distinct}, if they are different modulo rescaling, rigid motion, truncation. Otherwise, they are called \textbf{asymptotically equivalent}.
\end{definition}

Note that the spherical singularity is the only critical point of $U$ corresponding to a convex MCF. Therefore, if we are to find an rigid motion, rescaling and truncation to match two arrival time function of convex MCF, we must ensure that the spatial translation match the their respective critical point.

Let us describe how these actions will impact the arrival time function. Suppose that $U\in \cC\cA^{2}$ satisfies \eqref{Taylor;4} with Taylor polynomial $P_3,...,P_N$. Suppose that $(x_0,T_0)$ is the spherical singularity of the corresponding MCF. 
 Then Theorem \ref{Taylor-to-asymptotic}  implies that the corresponding RMCF centered at $(x_0,T_0)$ and graph function $v$ satisfies  \eqref{v;asymp;5} with asymptotic function $f_1,...f_{N-2}$. Note that  we did not assume $C^{N}$ here and \eqref{Taylor;4} is weaker than being $C^{N}$. 

The truncation of $U$ is amount to starting the corresponding MCF centered at $(x_0,T_0)$ at later time. This action will keep $f_1,...,f_{N-2}$ unchanged. Indeed, the truncation preserve the asymptotic behavior. 

The rigid motion of $U$ is amount to the space-time translation and rotation of the corresponding MCF. When coorepsonding to RMCF, we move the center point accordingly and therefore this action will keep $f_1,...,f_{N-2}$ unchanged when no rotation is involved. Indeed, the full asymptotic behavior is preserved.
In general with a rotation, there exists $R\in O(n+1)$ and $f_l$ will be changed to $f_l\circ R$ for all $1\leq l \leq N-2$. This will change $f_l$ but will keep $||f_l||_{L^2}$ unchanged. In addition, $||\pi_j(f_l)||_{L^2}$ is also invariant under rotation for all $l\geq 1$ and $j\geq 0$ since eigenspaces $E_j$ are all invariant under rotation. Indeed, this follows from the fact that $E_j$ is the space of restriction of harmonic homogeneous degree $j$ polynomial onto the sphere.

 We will explain more on the rescaling. In view of the correspondence between MCF and arrival time function, the MCF corresponding to $U^{\lambda}$ is $\{{M}^{\lambda}_t\}_{t\in [0, \lambda^{2}T_0)} = \Big\{x_0 + \lambda (M_{\lambda^{-2}t}-x_0)\Big\}_{t\in [0, \lambda^{2}T_0)}$, which is a rescaling of $M_t$ by $\lambda$ with a spherical singularity at $(x_0,\lambda^2T_0). $
 Then we can compute that the corresponding RMCF centered at $(x_0,\lambda^2 T_0)$ to be $\bar{M}^{\lambda}_{\tau} = \bar{M}_{\tau + 2\ln(\lambda)}$. Clearly $U^{\lambda}$ has the same regularity as $U$, hence the corresponding profile function $v^{\lambda}$ will satisfy the modified asymptotic:
\begin{align}\label{f;rescale}
	 \sum^{N-2}_{l=1}f_l e^{-\frac{l (\tau+2\ln \lambda)}{2} } =  \sum^{N-2}_{l=1}f_l\lambda^{-l} e^{-\frac{l \tau}{2} }
\end{align}
Alternatively, one can directly check that changing $U$ to $U^{\lambda}$ will scale $P_k$ by factor of $\lambda^{2-k}$, then Corollary \ref{fm;Lemma;2} implies that $f_l$ will be scaled by factor of $\lambda^{-l}$ for all $1\leq l\leq N-2$. Either way, we have obtained:
\begin{lemma}\label{Lem;rescale}
	 Rescaling $U$ by factor of $\lambda$ (equivalent to rescaling MCF by $\lambda$)  will scale  $f_l$ by a factor of $\lambda^{-l}$ for each $1 \leq l \leq N-2$.
\end{lemma}

\subsection[Construction of a finite parameter family of solutions]{Construction of a finite parameter family of solutions}\label{subsec;construct-I-parameter-solution}

Suppose that $U_1$ is an arrival time function of convex MCF with singularity $(x_0,T_0)$. Then the corresponding RMCF centered at $(x_0, T_0)$ has graph function $v_1$, which is an admissible solution to the graphical RMCF equation \eqref{v;equation;1} on $[\tau_0,\infty)$. 
 
Let $I\geq 2$. Fix $r>\frac{n}{2}+2$ and $R>0$. We construct an $(I-1)$-parameter collections of graphical RMCF solutions defined on $[\tau_1,\infty)$: 
\begin{align*}
	\mathbf{v}^{(R,r,v_1)}(Q_2,\ldots,Q_I), \qquad
	(Q_2,\ldots,Q_I)\in \prod_{k=2}^I \bigl(E_k\cap H^r|_{B_R}\bigr)
\end{align*}
$\tau_1$ will be specified later, but will be fully determined by $R,r,v_1,n, I$. 
We then use \eqref{RMCF-MCF;2} to convert each of them back to MCF with spherical singularity at $(x_0,T_0)$. The corresponding arrival time function is denoted by 
\begin{align*}
	\mho^{(R,r,U_1)}(Q_2,\cdots, Q_I)
\end{align*}
Here each $Q_k$ is the restriction to $\mathbb{S}^n(\sqrt{2n})$ of a homogeneous harmonic polynomial of degree $k$. Since $E_k$ is finite dimensional, each $Q_k$   may be identified with a finite tuple of real parameters. In particular, when $n=1$, the space $E_k$ has real dimension two, so $Q_k$  may be represented by either two real parameters or one complex parameter.

Note that $\mathbf{v}^{(R,r,v_1)}$ and $\mho^{(R,r,U_1)}$ also depend on the cutoff degree $I$ and the dimension $n$. We suppress this dependence because (1) the parameter domain $\prod_{k=2}^I(E_k\cap H^r|_{B_R})$ records $I$, and (2) $n$ is fixed background data.

Set
\begin{align}\label{Def;iterated-eta-sigma}
	\eta:=\frac{1}{4n}, \qquad \sigma_k:=\lambda_k-\frac{\eta}{2}, \qquad 2\leq k\leq I. 
\end{align}
Then $\lambda_{k-1}<\sigma_k<\lambda_k$ for every $2\leq k\leq I$. Take
\begin{align}
	&C_{1,k}:={}C(n,r)\left[\left(1+\sigma_k^{\frac12}\eta^{-\frac12}\right)
	\left(\frac{\lambda_k}{\lambda_k-\sigma_k}\right)^{\frac12}
	(1+\lambda_k^{\frac{1}{2}}\eta^{-\frac{1}{2}})\frac{\lambda_k}{\sigma_k - \lambda_{k-1}} \right], \label{Def;iterated-C1}\\
	&C_{2,k}:={}2+\sigma_k^{\frac12}\eta^{-\frac12}, \label{Def;iterated-C2}\\
    &C_{E,k}:={}\lambda_k \\
    &C_3 \text{ from }  \eqref{Nonlinear;1}. 
\end{align}

and take
\begin{align}\label{Def;iterated-uniform-C}
	C_1^*:=\max_{2\leq k\leq I}C_{1,k}, \qquad C_2^*:=\max_{2\leq k\leq I}C_{2,k}, \qquad C_{E}^* = \max_{2\leq k\leq I}C_{E,k}. 
\end{align}
We  choose the constants
\begin{align}
	\delta:=\min\left\{(10C_1^*+4)^{-1},\frac{\eta}{32C_{E}^*C_2^*C_3}\right\},
	\qquad \varepsilon:=(2+2C_2^*)^{-1}\delta. \label{Def;iterated-epsilon-delta}
\end{align}
These constants only depend on $I$, but they work simultaneously for all $2\leq k\leq I$ in Theorem \ref{Thm;prescribe-diff-asymptotic;1}, Proposition \ref{Prop;Lipschitz},  Lemma \ref{Lem;Diff-w;Diff-q} and Corollary \ref{Cor;existence;1}. Set
\begin{align}\label{Def;iterated-R0}
	R_1:=||v_1(\cdot,\cdot+\tau_0)||_{r,\eta}+||v_1(\cdot,\cdot+\tau_0)||_{r+1,\eta}.
\end{align}
By Lemma \ref{Lem;admissible-norm}, for every $s\geq \tau_0$,
\begin{align}\label{Ineq;iterated-v0-decay}
	||v_1(s)||_{H^r}+||v_1(s)||_{H^{r+1}}\leq R_1e^{-\eta (s-\tau_0)}.
\end{align}

We now choose a common initial time for the entire construction. Let
\begin{align}
	\tau_1:=|\tau_0| + &\eta^{-1}\ln\left(1+\frac{2R_1(1+\eta^{-\frac12})}{\varepsilon}\right)
	+\lambda_2^{-1}\ln\left(1+\frac{3R}{\varepsilon}\right) \nonumber\\
	&+\lambda_2^{-1}\ln\left(1+\frac{8(I-1)C_2^*R}{\varepsilon}\right). \label{Def;iterated-tau1}
\end{align}

Once $v_1,n,r$, and $I$ are fixed, the right hand side depends only on $R$. Define, for $\tau\geq0$:
\begin{align}\label{Def;iterated-shift}
	\hat{v}_1(\tau):=v_1(\tau+\tau_1), \qquad
\end{align}
As in the proof of Theorem \ref{Thm;prescribe-diff-asymptotic;2}, we have
\begin{align}
	||\hat{v}_1||_{r,\eta}
	={}&\left(\int_0^\infty ||v_1(s+\tau_1)||_{H^{r+1}}^2ds\right)^{\frac12}
	+\sup_{s\geq0}e^{\eta s}||v_1(s+\tau_1)||_{H^r} \nonumber\\
	={}&\left(\int_{\tau_1}^\infty ||v_1(s)||_{H^{r+1}}^2ds\right)^{\frac12}
	+e^{-\eta\tau_1}\sup_{s\geq \tau_1}e^{\eta s}||v_1(s)||_{H^r} \nonumber\\
	\leq{}&\left(\int_{\tau_1}^\infty R_1^2e^{-2\eta (s-\tau_0)}ds\right)^{\frac12}
	+R_1e^{-\eta(\tau_1-\tau_0)} \nonumber\\
	\leq{}&R_1(1+\eta^{-\frac12})e^{-\eta(\tau_1-\tau_0)}
	<\frac{\varepsilon}{2}. \label{Ineq;iterated-v0-small}
\end{align}
Further, for any fixed tuple $(Q_2,\ldots,Q_I)$ in the parameter domain, define
\begin{align}
	\hat{Q}_k:=e^{-\lambda_k\tau_1}Q_k, \qquad 2\leq k\leq I.
\end{align}
Then for every $2\leq k\leq I$,
\begin{align}
	||\hat{Q}_k||_{H^r}
	=e^{-\lambda_k\tau_1}||Q_k||_{H^r}
	\leq Re^{-\lambda_2\tau_1}
	<\frac{\varepsilon}{3}, \label{Ineq;iterated-Q-small}
\end{align}
and
\begin{align}
	4(I-1)C_2^*Re^{-\lambda_2\tau_1} < \frac{\varepsilon}{2}. \label{Ineq;iterated-sum-small}
\end{align}

We now construct the solutions inductively without any further time shift.

For each $2\leq k\leq I$, assume that $w_2,\ldots,w_{k-1}$ have been constructed and satisfy:
\begin{align}\label{Ineq;iterated-induction-hypothesis}
	\sum_{j=2}^{k-1}||w_j||_{r,\eta}
	< 4(k-2)C_2^*Re^{-\lambda_2\tau_1}.
\end{align}
Set
\begin{align}\label{Def;iterated-vk-base}
	\hat{v}_{k-1}:=\hat{v}_1+\sum_{j=2}^{k-1}w_j.
\end{align}
when $k = 2$, the collection of $w_2,...,w_{k-1}$ is empty and so the following  inequality is trivial.

By \eqref{Ineq;iterated-v0-small} and the induction hypothesis,
\begin{align}
	||\hat{v}_{k-1}||_{r,\eta}
	&\leq||\hat{v}_1||_{r,\eta}+\sum_{j=2}^{k-1}||w_j||_{r,\eta} \nonumber\\
	&\leq\frac{\varepsilon}{2}+4(k-2)C_2^*Re^{-\lambda_2\tau_1} \nonumber\\
	&\leq\frac{\varepsilon}{2}+4(I-1)C_2^*Re^{-\lambda_2\tau_1}
	<\varepsilon. \label{Ineq;iterated-base-small}
\end{align}

Consider the map $F$ in Definition \ref{Def;w-q0;asymptotic} corresponding to the base solution $\hat{v}_{k-1}$ and the eigenspace $E_k$. This will depend on $\hat{v}_{k-1}$ and so we denote $F|_{E_k} = F_{[\hat{v}_{k-1}]}$.   Together with \eqref{Ineq;iterated-Q-small}, this allows us to apply Theorem \ref{Thm;prescribe-diff-asymptotic;1} to produce a unique solution $w_k$ to \eqref{w;equation;1} defined on $[0,\infty)$ with base $\hat{v}_{k-1}$ such that 
\begin{align}
	||w_k||_{r,\sigma_k,\eta}<\delta, \qquad
	\left|e^{\lambda_k \tau}w_k(\tau)-\hat{Q}_k\right|_{C^0(\mathbb{S}^n)}=O(e^{-\eta \tau}). \label{Ineq;iterated-wk-asymp}
\end{align}
Moreover, the initial data satisfies $ \Pi_k w_k(0) = F_{[\hat{v}_{k-1}]}^{-1}(\hat{Q}_k)$. 
Then Proposition \ref{Prop;Lipschitz} gives  
\begin{align}
	||F_{[\hat{v}_{k-1}]}^{-1}(\hat{Q}_k)||_{H^r}&\leq2||\hat{Q}_k||_{H^r}, \label{Ineq;iterated-qk}
\end{align}
and  Lemma \ref{Lem;Diff-w;Diff-q} implies
\begin{align}
	||w_k||_{r,\eta}
	&\leq||w_k||_{r,\sigma_k,\eta}
	\leq2C_{2,k}||F_{[\hat{v}_{k-1}]}^{-1}(\hat{Q}_k)||_{H^r}
	\leq4C_2^*Re^{-\lambda_2\tau_1}. \label{Ineq;iterated-wk}
\end{align}
Using \eqref{Ineq;iterated-wk} and the induction hypothesis, we obtain
\begin{align}
	\sum_{j=2}^k||w_j||_{r,\eta}
	&\leq4(k-2)C_2^*Re^{-\lambda_2\tau_1}
	+4C_2^*Re^{-\lambda_2\tau_1} \nonumber\\
	&=4(k-1)C_2^*Re^{-\lambda_2\tau_1}. \label{Ineq;iterated-induction-conclusion}
\end{align}
We then define $\hat{v}_k:=\hat{v}_{k-1}+w_k.$
This completes the induction. In particular, all the shifted solutions $\hat{v}_2,\ldots,\hat{v}_I$ are defined on the same interval $[0,\infty)$.  
Finally, for every $2\leq k\leq I$, define
\begin{align}\label{Def;iterated-vk}
	v_k(\tau):=\hat{v}_k(\tau-\tau_1), \qquad \tau\geq\tau_1.
\end{align}
All the solutions $v_2,\ldots,v_I$ are defined on the common interval $[\tau_1,\infty)$.  Using \eqref{Def;iterated-shift} and the asymptotics of $w_k$, we obtain
\begin{align}
	&\left|e^{\lambda_k\tau}(v_k-v_{k-1})-Q_k\right|_{C^0(\mathbb{S}^n)} \nonumber\\
	&\quad=e^{\lambda_k\tau_1}
	\left|e^{\lambda_k(\tau-\tau_1)}w_k(\tau-\tau_1)-e^{-\lambda_k\tau_1}Q_k\right|_{C^0(\mathbb{S}^n)}
	=O(e^{-\eta\tau}). \label{v;iterated-asymp;1}
\end{align}
Since $\hat{v}_1$ and $w_2,\ldots,w_I$ all have finite $||\cdot||_{r,\eta}$ norm, each $\hat{v}_k$ and hence each $v_k$ is admissible by Lemma \ref{Lem;admissible-norm}.

Note that our process is deterministic upon fixing $R, r, v_1$ and $Q_2,\cdots, Q_I$. Therefore, we can define
\begin{align}
	\mathbf{v}^{(R,r,v_1)}(Q_2,\ldots,Q_I):=v_I. \label{Def;iterated-family}
\end{align}
Thus each parameter $Q_k$ prescribes the leading asymptotic mode of the difference introduced at the $k$-th stage.
If $Q_k=0$, then the zero difference solves the corresponding prescription problem; hence uniqueness gives $\hat{v}_k=\hat{v}_{k-1}$, and therefore $v_k=v_{k-1}$ after shifting back in time.

In the following lemmas, we record some basic facts about difference of our constructed solutions:
\begin{lemma}\label{Lem;error-estimate;0}
Suppose that $2\leq k\leq I-1$ and $(Q_2,\ldots,Q_I)$ and $(\tilde{Q}_2,\ldots, \tilde{Q}_I)$ coincide up to index $k$, i.e $Q_2 =\tilde{Q}_2,\ldots, Q_k = \tilde{Q}_k$.  Then 
\begin{align}
	 \Big|e^{\lambda_{k+1}\tau } \Big[ \mathbf{v}^{(R,r,v_1)}(\tilde{Q}_2,\ldots, \tilde{Q}_I) - \mathbf{v}^{(R,r,v_1)}(Q_2,\ldots,Q_I)\Big] - (\tilde{Q}_{k+1} - Q_{k+1})\Big|_{C^0(\mathbb{S}^n)} = O(e^{-\eta \tau}).
\end{align}
\begin{proof}
	Let $v_1, v_2,\ldots,v_I$ be the intermediate graph functions produced by the
iterative construction for $\mathbf{v}^{(R,r,v_1)}(Q_2,\ldots,Q_I)$ and  $\tilde{v}_1, \tilde{v}_2,\ldots,\tilde{v}_I$ be the intermediate graph functions produced by the
iterative construction for $\mathbf{v}^{(R,r,v_1)}(\tilde{Q}_2,\ldots,\tilde{Q}_I)$. Note that $\tilde{v}_1 = v_1$.

Since $Q_2 =\tilde{Q}_2,\ldots, Q_k = \tilde{Q}_k$, the construction process gives that $v_2 = \tilde{v}_2,\ldots, v_k = \tilde{v}_k$. 
Together with \eqref{v;iterated-asymp;1}, we have the following estimate:
\begin{align*}
	&\Big|e^{\lambda_{k+1}\tau}\Big( \mathbf{v}^{(R,r,v_1)}(\tilde{Q}_2,\ldots,\tilde{Q}_I) - \mathbf{v}^{(R,r,v_1)}(Q_2,\ldots,Q_I)\Big) - (\tilde{Q}_{k+1} - Q_{k+1})\Big|_{C^0(\mathbb{S}^n)} \\
	=&  \Big|e^{\lambda_{k+1}\tau}\sum_{i=2}^I (\tilde{v}_i - \tilde{v}_{i-1} - v_{i} + v_{i-1}) - (\tilde{Q}_{k+1} - Q_{k+1})\Big|_{C^0(\mathbb{S}^n)}\\
	\leq& \Big| e^{\lambda_{k+1}\tau}(\tilde{v}_{k+1} -\tilde{v}_{k})- e^{\lambda_{k+1}\tau}(v_{k+1} - v_k) -  (\tilde{Q}_{k+1} - Q_{k+1})\Big|_{C^0(\mathbb{S}^n)} \\ 
    &+ \sum_{i=k+2}^Ie^{\lambda_{k+1}\tau}\Big(|\tilde{v}_i - \tilde{v}_{i-1}|_{C^0(\mathbb{S}^n)}  + |v_i - v_{i-1}|_{C^0(\mathbb{S}^n)}\Big) \\
	 =& O(e^{-\eta \tau}).
\end{align*}
The assertion then follows.
\end{proof}
\end{lemma}

\begin{lemma}\label{Lem;error-estimate;1}
Suppose that $2\leq k \leq I$ and $Q_2=\cdots = Q_k = 0$. Moreover, we assume that
\begin{align}\label{U;expansion;5}
		U_1(x) = T_0 - \frac{|x-x_0|^2}{2n}  + |x-x_0|^{\frac{k(k-1)}{n}}H_k(x-x_0) + o(|x-x_0|^{\frac{k(k-1)}{n}+k}) \quad 
	\end{align}
Then 
\begin{align}
	|\mho^{(R,r,U_1)}(Q_2,\ldots,Q_I) - U_1| = o(|x-x_0|^{\frac{k(k-1)}{n} + k})
\end{align}	
\end{lemma}
\begin{remark}
	The leading order expansion of $U_1$ is not necessary. However, the current statement is sufficient for our purpose. Adding this assumption can simplify the argument.
\end{remark}
\begin{proof}
	Let  $v_1, v_2,\ldots,v_I$ be the intermediate graph functions produced by the
iterative construction.
Recall that
\begin{align*}
	v_i(\tau)=\hat v_i(\tau-\tau_1),\qquad \tau\geq\tau_1.
\end{align*}
Since $Q_i=0$  for $2\leq i\leq k$, uniqueness gives $v_i=v_{i-1}$ for these $i$. By \eqref{v;iterated-asymp;1} we get
\begin{align*}
	\big|v_I(\cdot,\tau)-v_1(\cdot,\tau)\big|_{C^0(\mathbb{S}^n)}
	&\leq\sum_{i=k+1}^I
	\big|v_i(\cdot,\tau)-v_{i-1}(\cdot,\tau)\big|_{C^0(\mathbb{S}^n)}\\
	&=O\Bigg(\sum_{i=k+1}^I e^{-\lambda_i\tau}\Bigg)
	=O(e^{-\lambda_{k+1}\tau})
	=o(e^{-\lambda_k\tau}),
\end{align*}

On the other hand, \eqref{U;expansion;5} and the argument of  Strehlke \cite[pp. 205--206]{str} tells us that:
\begin{align*}
	v_1=\frac{1}{2}
	(2n)^{\lambda_k + \frac{3}{2}-\frac{k}{2}}H_ke^{-\lambda_k\tau}
	+o(e^{-\lambda_k\tau}).
\end{align*}
Consequently, the difference estimate implies that
\begin{align*}
	v_I=\frac{1}{2}
	(2n)^{\lambda_k + \frac{3}{2}-\frac{k}{2}}H_ke^{-\lambda_k\tau}
	+o(e^{-\lambda_k\tau}).
\end{align*}
By definition,
\begin{align*}
	v_I=\mathbf{v}^{(R,r,v_1)}
	(Q_2,\ldots,Q_{I})
\end{align*}
is the graph function of the RMCF corresponding to the arrival time
function
\begin{align*}
	\mho^{(R,r,U_1)}
	(Q_2,\ldots,Q_{I}).
\end{align*}
The argument of Strehlke \cite[pp. 205--206]{str} now shows that $\mho^{(R,r,U_1)}
	(Q_2,\ldots,Q_{I})$ satisfies
\eqref{U;expansion;5}. 
The conclusion follows immediately.

\end{proof}

\subsection{Dimension 2, $n=1$}

	Let us take the complex coordinate introduced in Section \ref{Sec;elliptic}. Moreover, we adapt the decomposition $P_{i,j}$ from Section \ref{Sec;elliptic} and $f_{i,j}$ from Section \ref{Sec;parabolic-obstruction}. Recall that they correspond to the term in the form of $cz^i\Bz^j$ in $P_{i+j}, f_{i+j}$ respectively.  
	In addition, we adapt $\bp^{*}$ and $\mathbf{f}^{*}$ from  Section \ref{Sec;parabolic-obstruction}, which marks polynomials with special indices. 
	
	In this subsection, we shall keep track of the degree of the polynomials and make sure that the degree always match the subscripts.  For example, $P_l$ and $f_l$ denote polynomials of degree $l$. When restricted to the sphere $\mathbb{S}^1(\sqrt{2})$, a given  polynomial may admit different representations, since multiplying $\frac{z\Bz}{2}$ does not change the value. Among these equivalent representations, we always choose the one with the correct degree.
	
	We should remark that a polynomial $Q\in E_k$ can always be represented by a polynomial of degree $k^2-2$ on $\mathbb{S}^1(\sqrt{2})$. To see this, we can write $Q  = cz^k + \bar{c} \Bz^k$ for some $c\in \dC$, which is a polynomial of degree $k$. Then the goal is achieved by multiplying $\big(\frac{z\Bz}{2}\big)^{\frac{k^2-k-2}{2}}$.

	For smooth functions $f,f'$ on $\mathbb{S}^1(\sqrt{2})$ taking value in $\dC$, we define  $L^2$ Hermitian inner product and Hermitian norm on $\mathbb{S}^1(\sqrt{2})$ to be:
	\begin{align}
		\langle f,f'\rangle_{L^2} := \int_{\mathbb{S}^1(\sqrt{2})}f\bar{f'},\quad \quad ||f||_{L^2}^2 := \langle f,f\rangle_{L^2}
	\end{align}
	This is the natural extension of the $L^2$ inner product and norm  to complex valued functions. 
	
	For each $N\geq 4$, we find the largest integer $m$ such that $m^2\leq N$ and define $\Phi^{N}: \cC\cA^{N}_1\rightarrow \dC^{m-1}$ as follows. For each $U\in \cC\cA^{N}_1$, we find $\mathbf{f}^{-}_l, \mathbf{f}^{+}_l$ as above for all $2\leq l \leq m$. Then we let
	\begin{align}\label{Def;Phi;1}
		\Phi^{N}(U): = (\Phi^{N}_2(U), \cdots , \Phi^{N}_{m}(U))  
	\end{align}  
	where
	\begin{align}\label{Def;Phi;2}
		\Phi^{N}_l(U) := \frac{\Big\langle\mathbf{f}^{-}_l, z^{\frac{l(l-1)-2}{2}}\Bz^{{\frac{l(l+1)-2}{2}}}\Big\rangle_{L^2}}{\big\Vert z^{\frac{l(l-1)-2}{2}}\Bz^{{\frac{l(l+1)-2}{2}}}\big\Vert^2_{L^2}}  \in \dC
	\end{align}
	for each $2\leq l \leq m$.	

	 One can view $\Phi^N_l$ as the  coefficient of $z^{\frac{l(l-1)-2}{2}}\Bz^{\frac{l(l+1)-2}{2}}$ in $f_{l^2-2}$.
	 Note that we did not mention $\mathbf{f}^{+}_l$ here, because the term it is conjugate to $\mathbf{f}^{-}_l$ and provides no extra information. In particular, the quantity $\frac{\langle\mathbf{f}^{+}_l, z^{\frac{l(l+1)-2}{2}}\Bz^{{\frac{l(l-1)-2}{2}}}\rangle_{L^2}}{|| z^{\frac{l(l+1)-2}{2}}\Bz^{{\frac{l(l-1)-2}{2}}}||^2_{L^2}} $ equals $ \bar{\Phi^N_l}$, the conjugate of $\Phi^N_l$. 
	 
	 Throughout the subsection, we shall keep in mind that 
	 \begin{align}
	 	\lambda_k = \frac{k^2}{2}-1
	 \end{align}
	 as we will constantly switch between $\lambda_k$ and $\frac{k^2}{2}-1$ without referring here. 
\begin{lemma}\label{Lem;asymp-identical}
	Suppose that $n=1$.  For each $k\geq 2$,  each nonzero homogeneous harmonic polynomial $H_k$ of degree $k$ on $\mathbb R^2$, and fixed $x_0, T_0$, there are at most $2k$ asymptotically equivalent arrival time functions satisfying 
	\begin{align}\label{U;expansion;2}
		U(x) = T_0 - \frac{|x-x_0|^2}{2}  + |x-x_0|^{k(k-1)}H_k(x-x_0) + o(|x-x_0|^{k^2}) \quad 
	\end{align}
	that are different modulo truncation.
\end{lemma}

It is necessary to modulo truncation, for otherwise we can always get a continuous family of solutions truncated at different time.  
The meaning of the Lemma is that rotation and rescaling can only create $k$ different asymptotics behavior (including those hidden in $o(|x-x_0|^{k^2})$) satisfying \eqref{U;expansion;2}. This result is true only when $n=1$. 
\begin{proof}
	Suppose that we have $2k+1$ asymptotically equivalent arrival time functions $U_j$ satisfying \eqref{U;expansion;2} but they are pairwise distinct modulo translation and truncation.
	This means that there exists an open ball $B_{\epsilon} = B_{\epsilon}(0)\subset\RR^2$ centered at origin such that after translations, $U_1,\cdots, U_{2k+1}$ are always pairwise distinct on $B_{\epsilon}$.
	
	Since $U_i$ and $U_j$ are asymptotically equivalent for all $1\leq i,j\leq 2k+1$, by possibly shrinking $\varepsilon$ we can find $\lambda > 0$ and $R_j\in O(2)$  such that the following holds on $B_{\epsilon}$:
	\begin{align}\label{U;expansion;11}
		U_1(x-x_0) = T_j +\lambda^2 \Big(U_j\Big(\frac{R_j(x-x_0)}{\lambda}\Big)-T_0\Big)
	\end{align}
	
	Using the fact that $|R_jx| = |x|$, one can compute the right-hand-side as:
	\begin{align}\label{U;expansion;12}
		 T_j + \lambda^2 \Big(U_j\Big(\frac{R_j(x-x_0)}{\lambda}\Big)-T_0\Big) = T_j - \frac{|x-x_0|^2}{2} +  \lambda^{2-k^2}\cdot|x-x_0|^{k(k-1)}H_k\circ R_j + o(|x-x_0|^{k^2})
	\end{align}
	Putting together \eqref{U;expansion;2}, \eqref{U;expansion;11} and \eqref{U;expansion;12} and  we get $\lambda = 1$ (no rescaling), $T_j = T_0$ (no time shift) and $H_k = H_k\circ R_j$ on small ball $B_{\epsilon}$. 
    
    However, there are at most $2k$ elements in $O(2)$ that leave $H_k$ invariant. Actually, using polar coordinates, any $H_k$ can be written as $H_k(r,\theta) = C r^k \cos(k\theta - \phi)$ with $C \neq 0$. Its symmetry group consist of $k$ rotations and $k$ reflections.

	By Pigeonhole principle, there must exist $1 \leq j < l \leq 2k+1$ such that $R_j = R_l$. 
	
	However, \eqref{U;expansion;11} implies that $U_j = U_l$ on a small ball $B_{\epsilon}$ (after possibly shrinking $\epsilon$), this contradicts our assumption.
	
	Therefore the assumption is false and the Lemma follows.
\end{proof}

\begin{proposition}\label{Prop;nonregular-parameter-family}
Fix $R,r$, $k\geq 2$, and $I\geq k^2+k-1$. Let $U_1$ be an arrival time function of a convex mean curvature flow satisfying
\begin{align}\label{U;expansion;7}
		U_1(x) = T_0 - \frac{|x-x_0|^2}{2}  + |x-x_0|^{k(k-1)}H_k(x-x_0) + o(|x-x_0|^{k^2}) \quad 
	\end{align}
for some nonzero homogeneous harmonic polynomial $H_k$ of degree $k$ on $\mathbb R^2$. Then there exist an element $\underline{Q}_{k+1} \in E_{k+1}\cap H^r|_{B_R}$ and
a map
\begin{align*}
	\Psi_2:\prod_{i=k+1}^{k^2+k-3}\bigl(E_i\cap H^r|_{B_R}\bigr) \longrightarrow E_{k^2+k-2}\cap H^r|_{B_R}
\end{align*}
such that
\begin{align*}
	\mho^{(R,r,U_1)}\bigl(\underbracket{0,\ldots,0}_{Q_2,...,Q_k},Q_{k+1},\ldots,Q_I\bigr)
	\notin C^{(k^2+k-1)^2}
\end{align*}
for every $(Q_{k+1},\ldots,Q_I)\in
\prod_{i=k+1}^I\bigl(E_i\cap H^r|_{B_R}\bigr)\setminus \mathcal{Z}$,
where
\begin{align*}
	\mathcal{Z}:=\left\{
	(Q_{k+1},\ldots,Q_I)
	\in\prod_{i=k+1}^I\bigl(E_i\cap H^r|_{B_R}\bigr):
	\begin{aligned}
	&Q_{k+1}= \underline{Q}_{k+1} \quad\text{or}\\
	&Q_{k^2+k-2}= \Psi_2\big(Q_{k+1},\ldots, Q_{k^2+k-3}\big)
	\end{aligned}
	\right\}.
\end{align*}
\end{proposition}

\begin{proof}
Set
\begin{align*}
	m:=k^2+k-1,\qquad N:=m^2.
\end{align*}
Let $v_1$ be the graph function over $\mathbb{S}^1(\sqrt{2})$ of the
RMCF associated with the convex MCF corresponding
to $U_1$. The work of Strehlke \cite{str} indicate that the error term in \eqref{U;expansion;7} can be improved to $O(|x-x_0|^{k^2+\sigma})$ for some $\sigma > 0$ (see also \cite{se}). By the argument in \cite[pp.205-206]{str} we get:
\begin{align}\label{v;asymp;11}
	\Big|v_1 -  2^{\frac{k^2-k-1}{2}} H_ke^{-\lambda_k \tau}\Big| \leq Ce^{-(\lambda_k + \sigma)\tau}
\end{align}

Throughout the proof, all the parameter tuples $(Q_2,\ldots, Q_I)$ will satisfy $Q_2=\cdots = Q_k =0$ and we only specify $Q_{k+1},\ldots,Q_I$. Whenever $Q_{k+1},\ldots,Q_I$ are specified,  write
\begin{align*}
	\mathbf{Q}:= (Q_2,\ldots, Q_I) = (0,\ldots,0,Q_{k+1},\ldots,Q_I)
\end{align*}
as the \textbf{associated parameter tuple}. 
 Denote the corresponding arrival time
function by
\begin{align*}
	U_{\mathbf{Q}}:=
	\mho^{(R,r,U_1)}(0,\ldots,0,Q_{k+1},\ldots,Q_I),
\end{align*}
and denote the graph function of its corresponding RMCF by
\begin{align*}
	v_{\mathbf{Q}}:=
	\mathbf{v}^{(R,r,v_1)}(0,\ldots,0,Q_{k+1},\ldots,Q_I).
\end{align*}
In particular, $U_1 = \mho^{(R,r,U_1)}(0,\ldots, 0)$ and $v_1 =
	\mathbf{v}^{(R,r,v_1)}(0,\ldots,0)$.

Whenever $U_{\mathbf{Q}}\in\cC\cA_1^N$, the Taylor expansion
\eqref{Taylor;4} applies to $U_{\mathbf{Q}}$, and the asymptotic expansion
\eqref{v;asymp;5} applies to $v_{\mathbf{Q}}$. In particular, there exist
smooth functions
\begin{align*}
	f_1^{U_{\mathbf{Q}}},\ldots,f_{N-2}^{U_{\mathbf{Q}}}
\end{align*}
on $\mathbb{S}^1(\sqrt{2})$ such that
\begin{align}\label{v;asymp;parameter-family}
	\left|v_{\mathbf{Q}}-
	\sum_{l=1}^{N-2}f_l^{U_{\mathbf{Q}}}e^{-\frac{l\tau}{2}}
	\right|_{C^0(\mathbb{S}^1(\sqrt{2}))}
	=o(e^{-\frac{N-2}{2}\tau}).
\end{align}
These are the asymptotic functions of $U_{\mathbf{Q}}$ used in the
definition of $\Phi^N(U_{\mathbf{Q}})$.

Suppose that $\mathbf{Q}$ and $\widetilde{\mathbf{Q}}$ agree in all
components with indices less than $l$, where $k+1\leq l\leq m-1$.
Lemma \ref{Lem;error-estimate;0} gives
\begin{align}\label{Phi;parameter-comparison}
	\left|e^{\lambda_l\tau}(v_{\mathbf{Q}}-v_{\widetilde{\mathbf{Q}}})
	-(Q_l-\widetilde Q_l)\right|_{C^0(\mathbb{S}^1(\sqrt{2}))}
	=O(e^{-\eta\tau}).
\end{align}
Suppose in addition that $U_{\mathbf{Q}},U_{\widetilde{\mathbf{Q}}}\in\cC\cA_1^N$. Since
$2\lambda_l=l^2-2$, comparison with the unique expansions
\eqref{v;asymp;parameter-family} shows that their asymptotic functions agree through
degree $l^2-3$, while
\begin{align*}
	f^{U_{\mathbf{Q}}}_{l^2-2}-f^{U_{\widetilde{\mathbf{Q}}}}_{l^2-2}
	=Q_l-\widetilde Q_l.
\end{align*}
It follows from
\eqref{Def;Phi;2} that
\begin{align}
	\Phi_j^N(U_{\mathbf{Q}}) = \Phi_j^N(U_{\widetilde{\mathbf{Q}}}), \quad 2\leq j\leq l-1
\end{align}
Since $Q_l - \tilde{Q}_l\in E_l$, it can be written as $Q_l - \tilde{Q}_l = 2Re(cz^{\frac{l(l-1)-2}{2}}\Bz^{{\frac{l(l+1)-2}{2}}})$ for some $c\neq 0$.  By
\eqref{Def;Phi;2} again, we have
$\Phi_l^N(U_{\mathbf{Q}}) - \Phi_l^N(U_{\widetilde{\mathbf{Q}}}) = c$. As a consequence, 
\begin{align}\label{Phi;parameter-injective}
	Q_l\neq\widetilde Q_l
	\quad\Longrightarrow\quad
	\Phi_l^N(U_{\mathbf{Q}})\neq\Phi_l^N(U_{\widetilde{\mathbf{Q}}}).
\end{align}

We next verify the required nondegeneracy. By \eqref{U;expansion;7} and
the argument of Strehlke \cite[pp. 205--206]{str},
\begin{align*}
	|v_1 - 2^{\frac{k^2-k-1}{2}}H_ke^{-\lambda_k\tau}|_{C^0(\mathbb{S}^1(\sqrt{2}))} = o(e^{-\lambda_k\tau}).
\end{align*}
Since $Q_2=\cdots=Q_k=0$, Lemma \ref{Lem;error-estimate;0} applied to
$\mathbf{Q}$ and the zero parameter tuple gives
\begin{align}\label{v;diff;13}
	\big|v_{\mathbf{Q}}-v_1\big|_{C^0(\mathbb{S}^1(\sqrt{2}))}
	=o(e^{-\lambda_k\tau}).
\end{align}

Note that $0\neq H_k\in E_k$ and is in the form of $Re(cz^{\frac{k(k-1)-2}{2}}\Bz^{{\frac{k(k+1)-2}{2}}})$ for some $c\neq 0$. Then  whenever $U_{\mathbf{Q}}\in\cC\cA_1^N$, the definition \eqref{Def;Phi;2} and difference estimate \eqref{v;diff;13}  gives: 
\begin{align}\label{Phi;k;nonzero}
	\Phi_k^N(U_{\mathbf{Q}}) = 2^{\frac{k^2-k-1}{2}}c\neq 0.
\end{align}
Similarly, 
\begin{align}\label{Phi;lower-vanishing}
  \Phi_j^N(U_{\mathbf Q})=0,\qquad 2\leq j\leq k-1.
\end{align}

Consider the functions $Q_{k+1}\in E_{k+1}\cap H^r|_{B_R}$ for which there
exist parameters
\begin{align*}
	(Q_{k+2},\ldots,Q_I)
	\in\prod_{i=k+2}^I\bigl(E_i\cap H^r|_{B_R}\bigr)
\end{align*}
such that the associated parameter tuple $\mathbf{Q}$ satisfies
\begin{align*}
	U_{\mathbf{Q}}\in\cC\cA_1^N,\qquad
	\Phi_{k+1}^N(U_{\mathbf{Q}})=0.
\end{align*}
There is at most one such value by \eqref{Phi;parameter-injective}. If it
exists, denote it by $\underline{Q}_{k+1}$; otherwise, set $\underline{Q}_{k+1}=0$.

Now fix $(Q_{k+1},\ldots,Q_{m-2})$ with
$Q_{k+1}\neq \underline{Q}_{k+1}$. Consider those values of
$Q_{m-1}\in E_{m-1}\cap H^r|_{B_R}$ for which there exist parameters
\begin{align*}
	(Q_m,\ldots,Q_I)
	\in\prod_{i=m}^I\bigl(E_i\cap H^r|_{B_R}\bigr)
\end{align*}
such that the associated parameter tuple $\mathbf{Q} = (0,\ldots,0,Q_{k+1},\ldots, Q_I)$ satisfies
$U_{\mathbf{Q}}\in\cC\cA_1^N$. We claim that there is at most one such
value.

Indeed, suppose that $Q_{m-1}$ and $\widetilde Q_{m-1}$ are two such
values. Choose corresponding parameters
\begin{align*}
	(Q_m,\ldots,Q_I),\qquad
	(\widetilde Q_m,\ldots,\widetilde Q_I),
\end{align*}
and set
\begin{align*}
	\mathbf{Q}
	=(0,\ldots,0,Q_{k+1},\ldots,Q_I), \quad 
	\widetilde{\mathbf{Q}}
	=(0,\ldots,0,Q_{k+1},\ldots,Q_{m-2},
	\widetilde Q_{m-1},\widetilde Q_m,\ldots,\widetilde Q_I).
\end{align*}
Then
$U_{\mathbf{Q}},U_{\widetilde{\mathbf{Q}}}\in\cC\cA_1^N$, and the two
tuples agree through index $m-2$. Hence  Lemma \ref{Lem;error-estimate;0} implies that
\begin{align}\label{Phi;equality;11}
	\Phi_l^N(U_{\mathbf{Q}})=\Phi_l^N(U_{\widetilde{\mathbf{Q}}}),
	\qquad 2\leq l\leq m-2.
\end{align}
Moreover, \eqref{Phi;k;nonzero} and the definition of $\underline{Q}_{k+1}$ imply
\begin{align}\label{Phi;nonzero;12}
	\Phi_k^N(U_{\mathbf{Q}})\neq0,\qquad
	\Phi_{k+1}^N(U_{\mathbf{Q}})\neq0.
\end{align}

Applying Theorem \ref{Thm;14} and extracting the coefficient specified in
Lemma \ref{Lem;extraction;2}, as in the proof below, gives
\begin{align}\label{Phi;obstruction;13}
	C_k^{(6)}\overline{\Phi_k^N(U_{\mathbf{Q}})}
	\Phi_{k+1}^N(U_{\mathbf{Q}})
	\Phi_{m-1}^N(U_{\mathbf{Q}})
	+\cQ_k^{(7)}+\cQ_k^{(8)}=0,
\end{align}
where the last two terms $\cQ_k^{(7)}$ only depends on
$\Phi_k^N(U_{\mathbf{Q}}),\ldots,\Phi_{m-2}^N(U_{\mathbf{Q}})$ and their conjugates, while every monomial in $\cQ_k^{(8)}$ contains a factor among $\Phi_2^N(U_{\mathbf{Q}}),\ldots,\Phi_{k-1}^N(U_{\mathbf{Q}})$ or its conjugate. Thus $\cQ^{(8)}_k=0$ by \eqref{Phi;lower-vanishing} and $\cQ^{(7)}_k$ is determined by the coefficient through the index $m-2$. 

The same identity holds for $U_{\widetilde{\mathbf{Q}}}$.   Since $C_k^{(6)}\neq0$, 
\eqref{Phi;equality;11}, \eqref{Phi;nonzero;12} and \eqref{Phi;obstruction;13} implies that
\begin{align*}
	\Phi_{m-1}^N(U_{\mathbf{Q}})
	=\Phi_{m-1}^N(U_{\widetilde{\mathbf{Q}}}).
\end{align*}
\eqref{Phi;parameter-injective} then gives
$Q_{m-1}=\widetilde Q_{m-1}$, proving the claim.

Define $\Psi_2(Q_{k+1},\ldots,Q_{m-2})$ to be this unique value whenever it exists and
$Q_{k+1}\neq \underline{Q}_{k+1}$, and set it equal to zero otherwise.

Finally, suppose that $U_{\mathbf{Q}}\in\cC\cA_1^N$. If
$Q_{k+1}=\underline{Q}_{k+1}$, then
$(Q_{k+1},\ldots,Q_I)\in\mathcal Z$. Otherwise, the
definition of $\Psi_2$ gives
\begin{align*}
	Q_{m-1}=\Psi_2(Q_{k+1},\ldots,Q_{m-2}),
\end{align*}
and again $(Q_{k+1},\ldots,Q_I)\in\mathcal Z$. Thus every parameter tuple producing a
$C^N$ arrival time function lies in $\mathcal Z$. Equivalently, every
tuple outside $\mathcal Z$ produces an arrival time function that is not
in $C^N$.
\end{proof}

\begin{remark}
	The choice of $\Psi_2$ may not be unique and we do not have any regularity about them. In fact, it can happen that all parameters produces nonsmooth solution. In that case, the choice of $\Psi_2$ is arbitrary. However, if $\Psi_2$ happens to be differentiable, then $\mathcal{Z}$ is nothing but a union of submanifolds of complex codimension 1 (or equivalently real codimension 2) in the full space. Nevertheless, $\mathcal Z$ is still small in the set-theoretical sense. 
\end{remark}

\begin{theorem}
Suppose that $n=1$. For each $k\geq 2$ and each nonzero homogeneous harmonic polynomial $H_k$ of degree $k$ on $\mathbb R^2$,  there exists uncountably many asymptotically distinct arrival time functions $U$ that are not in  $C^N$ satisfying:
\begin{align}\label{U;expansion;6}
		U(x) = T_0 - \frac{|x-x_0|^2}{2}  + |x-x_0|^{k(k-1)}H_k(x-x_0) + o(|x-x_0|^{k^2}) \quad 
	\end{align}
where $N = (k^2+k-1)^2$. In particular, there exists uncountably many asymptotically distinct arrival time functions $U$ that are not $C^{25}$. 
 
\end{theorem}
\begin{proof}
Set
\begin{align*}
	m:=k^2+k-1
\end{align*}
By the work of Strehlke \cite{str}, there exists an arrival time function
$U_1$ of a convex mean curvature flow satisfying \eqref{U;expansion;6}.

Fix $R > 0, r > \frac{n}{2}+2$, fix an integer $J\geq1$, and let $I=m+J-1$. Then by subsection \ref{subsec;construct-I-parameter-solution}, we obtain a collection of solutions $\mho^{(R,r,U_1)}$. Take $\underline{Q}_{k+1}, \Psi_2$ and $\mathcal{Z}$ predicted by Proposition \ref{Prop;nonregular-parameter-family}. 
Choose $Q_2^* = \cdots = Q_{k}^* = 0 $ and $Q_{k+2}^* = \cdots = Q_{m-2}^* = 0$, then choose
\begin{align*}
	&Q_{k+1}^*\in E_{k+1}\cap H^r|_{B_R} \setminus \{\underline{Q}_{k+1}\}\\
	&Q_{m-1}^*\in E_{m-1}\cap H^r|_{B_R}\setminus\{\Psi_2(Q_{k+1}^*,\ldots , Q_{m-2}^*)\}.
\end{align*}
We now allow
\begin{align*}
	(Q_{m},\ldots,Q_{m+J-1})
	\in\prod_{i=m}^{m+J-1}\bigl(E_i\cap H^r|_{B_R}\bigr)
\end{align*}
to vary freely. By our choice of $Q_i^*$ and definition of $\mathcal{Z}$, every tuple of free parameters $(Q_{m},\ldots,Q_{m+J-1})$ ensures
\begin{align*}
	(Q_2^*,\ldots,Q_{m-1}^*,Q_{m},\ldots,Q_{m+J-1}) \not\in \mathcal{Z}
\end{align*}
Proposition
\ref{Prop;nonregular-parameter-family} therefore gives
\begin{align}\label{Ineq;parameter-family-nonregular}
	\mho^{(R,r,U_1)}
	(Q_2^*,\ldots,Q_{m-1}^*,Q_{m},\ldots,Q_{m+J-1})\notin C^N
\end{align}
for every choice of the $J$ free parameters $Q_{m},\ldots,Q_{m+J-1}$.

Since $Q_2^* = \cdots = Q_{k}^* = 0 $, we can apply Lemma \ref{Lem;error-estimate;1} to deduce that each $\mho^{(R,r,U_1)}
	(Q_2^*,\ldots,Q_{m-1}^*,Q_{m},\ldots,Q_{m+J-1})$ satisfies \eqref{U;expansion;6}.

For each $1\leq i\leq J$, choose an
$H^r$-unit vector $e_{m+i-1}\in E_{m+i-1}$ and restrict
\begin{align*}
	Q_{m+i-1}=t_i e_{m+i-1},\qquad t_i\in(-R,R).
\end{align*}
Then $(t_1,\ldots,t_J)\in(-R,R)^J$ parametrizes a collection of
non-$C^N$ solutions satisfying \eqref{U;expansion;6}.

By Lemma \ref{Lem;error-estimate;0}, different parameters give different solutions modulo truncation, Lemma \ref{Lem;asymp-identical} then implies that asymptotically equivalent solutions are represented by at most $2k$ parameters. In particular, this means that we have uncountably many asymptotically distinct solutions. 
\end{proof}

\subsection{Higher dimensions $n\geq 2$}

The higher dimensional case is fairly easy to handle. In contrast in the $n=1$ case, we do not track the degree of $f_1,\cdots, f_{N-2}$ appeared in \eqref{v;asymp;5}, Instead, we mostly treat them as smooth functions over $\mathbb{S}^n$. In particular, we do not state and use that $f_l$ is a restriction of degree $l$ polynomial.

Before proving the Theorem, we state one more Lemma. 
\begin{lemma}\label{Lem;poly;1}
	Suppose that $P$ and $P'$ are nonzero homogeneous polynomials of degree $k$ and $k'$ and that $P =P'$ on a round sphere centered at origin. Then $k - k'$ is an even number. 
\end{lemma}
\begin{proof}
    The proof is direct by comparing the values at antipodal points on the sphere.
\end{proof}
We finish the proof of the case $n\ge 2$ for the main Theorem \ref{thm:general}.

\begin{theorem*}
		Suppose that $n\geq 2$.  For each $k\geq 2$ and $0\neq H_k\in E_k$, there exist uncountably many asymptotically distinct arrival time functions $U$ that are not $C^N$ and  satisfy:
	\begin{align}\label{U;expansion;21}
		U(x) = T_0 - \frac{|x-x_0|^2}{2n}  + |x-x_0|^{\frac{k(k-1)}{n}}H_k(x-x_0) + o(|x-x_0|^{\frac{k(k-1)}{n}+k}) \quad 
	\end{align}
	Here $N=N(k)$ is defined as follows. Let $k_1$ be the smallest integer satisfying  $k_1 \geq k$ and $\frac{k_1(k_1-1)}{2n}\not\in \mathbb{Z}$. Then we define $N$ to be the smallest integer such that $N\geq  \frac{k_1(k_1-1)}{n} + k_1$.
\end{theorem*}

\begin{proof}
For each $U$ satisfying \eqref{U;expansion;21}, the corresponding MCF has a spherical singularity at $(x_0,T_0)$ and we can find the graph function $v$ of the corresponding RMCF centered at $(x_0,T_0)$.  Then the argument in \cite[pp.205-206]{str} implies that $U$ satisfies \eqref{U;expansion;21} if and only if :
\begin{align}\label{v;asymp;23}
	\Big|v -  \frac{(2n)^{\lambda_k +\frac{3}{2}-\frac{k}{2}}}{2} H_ke^{-\lambda_k \tau}\Big|_{C^0(\mathbb{S}^n(\sqrt{2n}))} = o(e^{-\lambda_k\tau}).
\end{align} 

By the work of Strehlke \cite{str}, there exists an arrival time function
$U_1$ of a convex mean curvature flow satisfying \eqref{U;expansion;21}. Then there is a corresponding MCF. Let $v_1$ be the graph function of the corresponding RMCF satisfying \eqref{v;asymp;23}. 
We aim to find a $J$-parameter collection of required solutions that are pairwise asymptotically distinct. 
 
Fix $R>0, r>\frac{n}{2} + 2$ and let $I = k_1 + J$.   Take all constants $\sigma_2\ldots, \sigma_I, \eta, \varepsilon,\delta, C_1^*, C_2^*$ as in \eqref{Def;iterated-eta-sigma} -- \eqref{Def;iterated-epsilon-delta}. Take $R_1:=||v_1(\cdot,\cdot+\tau_0)||_{r,\eta}+||v_1(\cdot,\cdot+\tau_0)||_{r+1,\eta}$ and 
\begin{align}
	\tau_1:=& |\tau_0| + \eta^{-1}\ln\left(1+\frac{4R_1(1+\eta^{-\frac12})}{\varepsilon}\right)
	+\lambda_2^{-1}\ln\left(1+\frac{4R\cdot J}{\varepsilon}\right). \label{Def;iterated-tau1;2}+ \lambda_2^{-1}\ln\left(1+\frac{8C_2^*R}{\varepsilon}\right)
\end{align}

We first choose arbitrary $H^r$-unit vectors $\varphi_i \in E_i$ for $k_1\leq i \leq k_1+ J$.

Then take $\varphi_{k_1}$ to be any $H^r$-unit vector in $E_{k_1}$. 
Let $\hat{v}_1(\tau) = v_1(\tau+\tau_1)$ for all $\tau \in [0,\infty)$. One can estimate as in the proof of Theorem \ref{Thm;prescribe-diff-asymptotic;2} or as \eqref{Ineq;iterated-v0-small}: 
\begin{align}
	||\hat{v}_1||_{r,\eta}
	={}&\left(\int_0^\infty ||v_1(s+\tau_1)||_{H^{r+1}}^2ds\right)^{\frac12}
	+\sup_{s\geq0}e^{\eta s}||v_1(s+\tau_1)||_{H^r} \nonumber\\
	\leq &R_1(1+\eta^{-\frac12})e^{-\eta(\tau_1-\tau_0)}
	<\frac{\varepsilon}{4}. \label{Ineq;iterated-v1-small}
\end{align}

Since $\Pi_i$ is $H^r$-nonincreasing map for all $i\geq 0$, we have
\begin{align}\label{Initial-data;11}
	||\Pi_{k_1+1} \hat{v}_1(\cdot, 0) ||_{H^r} \leq ||\hat{v}_1(\cdot, 0) ||_{H^r} <\frac{\varepsilon}{4}
\end{align}

\noindent\textbf{Step1: construction of solutions.} 
We construct a $J+1$-parameter collection of solutions $\mathbf{v}_{[t_{k_1}, t_{k_1+1}, \ldots, t_{k_1+J}]}$ with parameters 
$(t_{k_1}, t_{k_1+1}, \ldots, t_{k_1+J}) \in  [0,R]^{J+1}$.

Take any $(t_{k_1}, t_{k_1+1}, \ldots, t_{I}) \in [0,R]^{J+1}$. We first handle the parameter $t_{k_1}$.   Since $||t_{k_1}e^{-\lambda_{k_1}\tau_1}\varphi_{k_1}||_{H^r} \leq e^{-\lambda_{k_1}\tau_1} R < \varepsilon$,  Corollary \ref{Cor;existence;1} allows us to produce a unique solution $w_{k_1}$ to \eqref{w;equation;1} defined on $[0,\infty)$ with base $\hat{v}_{1}$ such that 
\begin{align}
	||w_{k_1}||_{r,\sigma_{k_1},\eta}<\delta, \qquad \Pi_{k_1}w_{k_1}(0) = t_{k_1}e^{-\lambda_{k_1}\tau_1}\varphi_{k_1}
\end{align}
Moreover, we can take the map $F$ as in Definition \ref{Def;w-q0;asymptotic}, which is denoted by $F_{[\hat{v}_1]}$ here. By definition, 
\begin{align}\label{w;asymp;2}
	\left|e^{\lambda_{k_1}\tau}w_{k_1}(\tau)-F_{[\hat{v}_1]}(t_{k_1}e^{-\lambda_{k_1}\tau_1}\varphi_{k_1})\right|_{C^0(\mathbb{S}^n)}=o(1). 
\end{align} 
Proposition \ref{Prop;Lipschitz} gives that $F_{[\hat{v}_1]}$ is injective on $E_{k_1}\cap H^r|_{B_{\varepsilon}}$. In particular, the input $t_{k_1}e^{-\lambda_{k_1}\tau_1}\varphi_{k_1}$ is always in the domain $E_{k_1}\cap H^r|_{B_{\varepsilon}}$ whenever  $t_{k_1}\in [0,R]$ and $F_{[\hat{v}_1]}$ maps 0 to 0.

By Lemma \ref{Lem;Diff-w;Diff-q} we have $||w_{k_1}||_{r,\sigma_{k_1}, \eta} \leq 2C_2 ||t_{k_1}e^{-\lambda_{k_1}\tau_1}\varphi_{k_1}||_{H^r} \leq 2C_2^* Re^{-\lambda_{k_1}\tau_1} < \frac{\varepsilon}{4}$. 
Therefore, 
\begin{align}\label{Initial-data;12}
	||\Pi_{k_1+1}w_{k_1}(\cdot, 0)||_{H^r}\leq ||w_{k_1}(\cdot, 0)||_{H^r}\leq ||w_{k_1}||_{r,\eta} < \frac{\varepsilon}{4}
\end{align}
Set $\hat{v}_2 = \hat{v}_1 + w_{k_1}$. Then by \eqref{Ineq;iterated-v1-small}, \eqref{Initial-data;11} and \eqref{Initial-data;12}:
\begin{align}
	&\Vert\hat{v}_2\Vert_{r,\eta} \leq \Vert\hat{v}_1\Vert_{r,\eta} + \Vert w_{k_1}\Vert_{r,\eta}  \leq  \frac{\varepsilon}{4} + \frac{\varepsilon}{4} \leq \frac{\varepsilon}{2} \\
	&\big\Vert  -\Pi_{k_1+1} \hat{v}_2(\cdot, 0) + e^{-\sigma_{k_1+1}\tau_1}\sum_{i=k_1+1}^{I} t_{i} \varphi_i \big\Vert_{H^r} \leq \frac{\varepsilon}{2}  +  e^{-\lambda_{2}\tau_1 }R\sqrt{J} < \varepsilon
\end{align}
These two estimates allows us to use Corollary \ref{Cor;existence;1} to find a unique solution $w$ defined on $[0,\infty)$ solving \eqref{w;equation;1}  with base $\hat{v}_2$ and satisfies 
\begin{align}\label{w;proj;2}
	||w||_{r,\sigma_{k_1+1}, \eta} < \delta \quad  \text{ and } \quad \Pi_{k_1 +1}w(\cdot,0) = -\Pi_{k_1+1} \hat{v}_2(\cdot, 0) + e^{-\sigma_{k_1+1}\tau_1}\sum_{i=k_1+1}^{I} t_{i} \varphi_i
\end{align}
Moreover, by Proposition \ref{Prop;fine-asymptotic;1} and Sobolev embedding we get 
\begin{align}\label{w;asymp;3}
	\left|e^{\lambda_{k_1}\tau}w(\tau)\right|_{C^0(\mathbb{S}^n)}=O(e^{(\lambda_{k_1}- \lambda_{k_1+1})\tau}) = o(1). 
\end{align}
Finally we define for all $\tau \in [\tau_1 ,\infty)$:
\begin{align}\label{time-shift;5}
	\mathbf{v}_{[t_{k_1}, t_{k_1+1}, \ldots, t_{I}]} (\tau):= (\hat{v}_2+ w) (\tau - \tau_1) = (\hat{v}_1 + w_{k_1} + w)(\tau - \tau_1)
\end{align}
Let us use \eqref{RMCF-MCF;2} to convert the RMCF with graph $\mathbf{v}_{[t_{k_1}, t_{k_1+1}, \ldots, t_{I}]} (\tau)$ back to the MCF with singularity $(x_0,T_0)$. The corresponding arrival time function is denoted by $\mho_{[t_{k_1}, t_{k_1+1}, \ldots, t_{I}]}$.

By \eqref{v;asymp;23}, \eqref{w;asymp;2}, \eqref{w;asymp;3} and \eqref{time-shift;5}, the following two asymptotics hold if $k_1 > k$, or if $k_1 = k$ and $t_{k_1} = 0$: 
\begin{align}\label{v;asymp;24}
	\Big|\mathbf{v}_{[t_{k_1}, t_{k_1+1}, \ldots, t_{I}]} (\tau) -  \frac{(2n)^{\lambda_k +\frac{3}{2}-\frac{k}{2}}}{2} H_ke^{-\lambda_k \tau}\Big|_{C^0(\mathbb{S}^n(\sqrt{2n}))} = o(e^{-\lambda_k\tau})
\end{align} 
\begin{align}\label{U;expansion;22}
		\mho_{[t_{k_1}, t_{k_1+1}, \ldots, t_{I}]}(x) = T_0 - \frac{|x-x_0|^2}{2n}  + |x-x_0|^{\frac{k(k-1)}{n}}H_k(x-x_0) + o(|x-x_0|^{\frac{k(k-1)}{n}+k})  
\end{align}

\noindent\textbf{Step2: $C^N$ solutions are rare.}
Let us consider those values of $t_{k_1} \in [0,R]$ for which there exists parameters $(t_{k_1+1},\ldots, t_{I}) \in [0,R]^J$ such that $\mho_{[t_{k_1}, t_{k_1+1}, \ldots, t_{I}]} $ is in  $ C^N$. We claim that there is at most one such $t_{k_1}$, which is denoted by $t_{k_1}^*$. In the case that no such value exists, we simply fix $t_{k_1}^*$ to be arbitrary number in $(0,R]$.

Suppose that the claim is not true, then there exists two such values $t_{k_1} \neq \tilde{t}_{k_1}\in [0,R]$. Choose the   corresponding remaining parameters $(t_{k_1+1},\ldots, t_{I})$ and $(\tilde{t}_{k_1+1},\ldots, \tilde{t}_{I})\in [0,R]^J$ such that $\mho^{[1]}:= \mho_{[t_{k_1},\ldots, t_{I}]}$ and $\mho^{[2]}:=\mho_{[t_{k_1},\tilde{t}_{k_1+1},\ldots, \tilde{t}_{I}]}$ are both in $C^N$. 
Denote their corresponding RMCF graph function to be $\mathbf{v}_{[t_{k_1}, t_{k_1+1}\ldots, t_{I}]} := \mathbf{v}^{[1]}$ and $\mathbf{v}_{[{t}_{k_1}, \tilde{t}_{k_1+1}, \ldots, \tilde{t}_{I}]}:= \mathbf{v}^{[2]}$ respectively. 
Then \eqref{Taylor;4} holds for $\mho^{[1]}, \mho^{[2]}$ and \eqref{v;asymp;5} holds for  $\mathbf{v}^{[1]}$ and $\mathbf{v}^{[2]}$, namely: 
\begin{align}\label{Taylor;11}
	T_0 - \mho^{[\alpha]}(x - x_0) = \frac{|x-x_0|^2}{2n} + \sum_{k=3}^{N}\frac{(P^{[\alpha]})_k(x-x_0)}{n} + o(|x-x_0|^{N})\qquad \alpha=1,2
\end{align} 
\begin{align}\label{v;asymp;21}
	\Big| v^{[\alpha]}(\cdot, \tau) - \sum^{N-2}_{l=1}(f^{[\alpha]})_l e^{-\frac{l \tau}{2} }\Big|_{C^0(\mathbb{S}^n(\sqrt{2n}))} = o(e^{-\frac{N-2}{2}\tau}) \quad  \quad \text{ as }\tau \rightarrow \infty\qquad \alpha=1,2
\end{align}
where $(P^{[\alpha]})_k$ are homogenous degree $k$ polynomials and $(f^{[\alpha]})_k$ are smooth functions on $\mathbb{S}^n(\sqrt{2n})$ for all $3\leq k \leq N$ and $\alpha = 1,2$. Moreover, $P^{[\alpha]}$ and $f^{[\alpha]}$ satisfies the relation specified in Corollary \ref{fm;Lemma;1} and \ref{fm;Lemma;2}. 

Before comparing $P^{[\alpha]}$ and $f^{[\alpha]}$, we need more information on the asymptotics. Let
\begin{align}
	Q^{diff} = e^{\lambda_{k_1}\tau_1}\Big(F_{[\hat{v}_1]}(t_{k_1}e^{-\lambda_{k_1}\tau_1}\varphi_{k_1}) -F_{[\hat{v}_1]}(\tilde{t}_{k_1}e^{-\lambda_{k_1}\tau_1}\varphi_{k_1})\Big)
\end{align}
By our assumption, $t_{k_1} \neq \tilde{t}_{k_1}\in [0,R]$, which means $t_{k_1}e^{-\lambda_{k_1}\tau_1}\varphi_{k_1}$ and $ \tilde{t}_{k_1}e^{-\lambda_{k_1}\tau_1}\varphi_{k_1}$ are both in $ E_{k_1}\cap H^r|_{B_{\varepsilon}}$. Since $F_{[\hat{v}_1]}$ is injective on $ E_{k_1}\cap H^r|_{B_{\varepsilon}}$, this means
\begin{align}
	Q^{diff} \neq 0
\end{align}

Note that \eqref{time-shift;5} can be rewritten as $\mathbf{v}_{[t_{k_1}, t_{k_1+1}, \ldots, t_{I}]} (\tau)=  v_1(\tau) + ( w_{k_1} + w)(\tau - \tau_1)$.
Together with \eqref{v;asymp;23},  \eqref{w;asymp;2} and \eqref{w;asymp;3} we get: 
\begin{align}\label{v;asymp;15}
	&\Big|e^{\lambda_{k_1}\tau} \Big( \mathbf{v}^{[1]}(\tau) -  \mathbf{v}^{[2]}(\tau)\Big) - Q^{diff}\Big|   \\
	=& e^{\lambda_{k_1}\tau_1}\Big|e^{\lambda_{k_1}(\tau-\tau_1)} \Big( \mathbf{v}_{[t_{k_1}, t_{k_1+1}, \ldots, t_{I}]}(\tau) -  \mathbf{v}_{[{t}_{k_1}, \tilde{t}_{k_1+1}, \ldots, \tilde{t}_{I}]}(\tau)\Big)  -  e^{-\lambda_{k_1}\tau_1}Q^{diff}\Big| = o(1) \nonumber
\end{align}

It can be quickly checked that $N-2\geq 2\lambda_{k_1}$ by our definition. Then using  \eqref{v;asymp;21}, \eqref{v;asymp;15} and that $Q^{diff} \neq 0$, we get: 
\begin{itemize}
	\item $2\lambda_{k_1}\in \mathbb{Z}$
	\item $(f^{[1]})_l = (f^{[2]})_l $ for all $1\leq l \leq 2\lambda_{k_1} - 1$
	\item $(f^{[1]})_{2\lambda_{k_1}} = (f^{[2]})_{2\lambda_{k_1}}  +  Q^{diff}$
\end{itemize} 

Using Corollary \ref{fm;Lemma;1} with $m = 2\lambda_{k_1}$ we get:  
\begin{align}
	(P^{[1]})_{2\lambda_{k_1}+2} - (P^{[2]})_{2\lambda_{k_1}+2} = -(2n)^{\frac{1}{2}} Q^{diff} \quad \text{ on } \mathbb{S}^n(\sqrt{2n})
\end{align}
the denominator disappeared because $|x|^2 = 2n$  on $\mathbb{S}^n(\sqrt{2n})$. In particular, $Q^{diff}$ coincide with a homogeneous degree $2\lambda_{k_1}+2$ polynomial on $\mathbb{S}^n(\sqrt{2n})$. 
On the other hand, any element in $E_{k_1}$, including $Q^{diff}$, must be equal to a homogeneous degree $k_1$ polynomial on  $ \mathbb{S}^n(\sqrt{2n})$. Lemma \ref{Lem;poly;1} then forces $k_1$ and $2\lambda_{k_1}+2$ to differ by an even number. However, $2\lambda_{k_1} + 2 - k_1  = 2(\lambda_{k_1} - \frac{k_1}{2}) + 2$ can not be an even number since we assumed $\lambda_{k_1} - \frac{k_1}{2}\not\in \mathbb{Z}$. This causes a contradiction. The claim is then justified. 

\begin{remark}
    We might already reached a contradiction when we conclude that $2\lambda_{k_1}\in \mathbb{Z}$, because $2\lambda_{k_1}$ is likely not an integer, especially when dimension is high. If this happens,  one can simply skip the afterward  argument. 
\end{remark}

When $k_1 =k$, we claim that $t_{k_1}^{*} \neq 0$. Indeed, it suffices to show that $\mho_{[0, t_{k_1+1}, \ldots, t_{I}]}(x) \not\in C^N$. If this is not the case, then it admits Taylor expansion $\mho_{[0, t_{k_1+1}, \ldots, t_{I}]}(x) = \sum_{k=0}^{N} P_k(x-x_0) +o(|x-x_0|^N)$, where $P_k$ is homogenous polynomial of degree $k$. Since $k = k_1$ and hence $\frac{k(k-1)}{n}$ is not an even number, \eqref{U;expansion;22}  is not compatible with the Taylor expansion, a contradiction.

\noindent\textbf{Step3: Asymptotic distinction.}
 
For each fixed $t_{k_1} \in [0,R]$,  we claim that they are pairwise asymptotically distinct in terms of the remaining parameters $(t_{k_1+1},\ldots, t_{I})$.

If this is not the case, then there exists two parameter tuples $(t_{k_1+1},\ldots, t_{k_i+J})$ and  $(\tilde{t}_{k_1+1},\ldots, \tilde{t}_{k_i+J})$ in $[0,R]^J$, a rotation $\mathcal{R}  \in O(n+1)$,  $\rho \in \RR_{+}$  and  $\tau_2 \geq \tau_1$ such that 
$\mathbf{v}_{[t_{k_1}, t_{k_1+1}, \ldots, t_{k_1+J}]}(\tau) = (\mathbf{v}_{[t_{k_1}, \tilde{t}_{k_1+1},\ldots, \tilde{t}_{k_i+J}]}\circ \mathcal{R})(\tau + 2\ln\rho)$ holds on $[\tau_2,\infty)$. By \eqref{v;asymp;24}, this implies that $H_k = \rho^{-2\lambda_k} H_k\circ \mathcal{R}$. Comparing the $L^2$ norm we get $\rho = 1$.  
They are both solutions to \eqref{v;equation;1}, since \eqref{v;equation;1} is rotational invariant.  Then backward uniqueness \cite{huang} implies that
\begin{align}
	\mathbf{v}_{[t_{k_1}, t_{k_1+1}, \ldots, t_{I}]} = \mathbf{v}_{[t_{k_1}, \tilde{t}_{k_1+1},\ldots, \tilde{t}_{I}]}\circ \mathcal{R}
\end{align}
 for all $\tau \in [\tau_1, \infty)$. Since $\Pi_i$ commutes with $\mathcal{R}$ for all $i\geq 0$, we can consequently get: 
\begin{align}
	\Pi_{k_1+1} \mathbf{v}_{[t_{k_1},t_{k_1+1}, \ldots, t_{I}]}(\cdot, \tau_1) =\Big(\Pi_{k_1+1} \mathbf{v}_{[t_{k_1}, \tilde{t}_{k_1+1},\ldots, \tilde{t}_{I}]}(\cdot,\tau_1) \Big)\circ \mathcal{R}
\end{align}
Together with \eqref{w;proj;2} and \eqref{time-shift;5}, we obtain:
\begin{align}
	 e^{-\sigma_{k_1+1}\tau_1}\sum_{i=k_1+1}^{I} t_{i} \varphi_i =  e^{-\sigma_{k_1+1}\tau_1}\sum_{i=k_1+1}^{I} \tilde{t}_{i} \varphi_i  \circ \mathcal{R}
\end{align}
Taking further projection $\pi_i$ for each $k_1+1\leq i \leq I$ we get
\begin{align}
	t_i \varphi_i = \tilde{t}_i \varphi_i\circ \mathcal{R}
\end{align}

Comparing the $L^2$ norm we get $t_i = \tilde{t}_i \geq 0$. But this forces  $(t_{k_1+1},\ldots, t_{k_i+J})=(\tilde{t}_{k_1+1},\ldots, \tilde{t}_{k_i+J})$, a contradiction.

\noindent\textbf{Completion of proof.}
We take the collection of solution $\mathbf{v}_{[t_{k_1},t_{k_1+1},\ldots,t_{I}]}$ constructed in Step 1 with \\  $(t_{k_1},t_{k_1+1},\ldots,t_{I}) \in [0,R]^{J+1}$. Then we take  $t_{k_1}^*\in [0,R]$ from Step 2 and fix a number  $t_{k_1}^{**} \in [0,R]\setminus\{t_{k_1}^*\}$.  In the case of $k_1 = k$, we set $t_{k_1}^{**} = 0$, which is allowed by Step 2. 
Then we let the collection of solutions to be
\begin{align}
	\mho_{[t_{k_1}^{**}, t_{k_1+1}, \ldots, t_{I}]}(x)
\end{align} 
with $J$-free parameters $t_{k_1+1}, \ldots, t_{I} \in [0,R]^J$. 
By the choice of $t_{k_1}^{**}$ and Step2, each solution in the collection is not in $C^N$. By Step 3, these $J$-parameter collection of solutions are pairwise asymptotically distinct. 
In particular there are uncountably many of them. 
Finally, by the choice of $t_{k_1}^{**}$ and by  \eqref{U;expansion;22} in Step 1, each $\mho_{[t_{k_1}^{**}, t_{k_1+1}, \ldots, t_{I}]}(x) $ satisfies \eqref{U;expansion;21}. the Theorem then follows.

\end{proof}


\appendix
\section{Proof of Corollaries \ref{fm;Lemma;1} and \ref{fm;Lemma;2}} 

Recall that we take $P_0, P_1, P_2, P_3,...,P_{N}$ and $f_0, f_1 ,...,f_{N-2}$ from Theorem \ref{Taylor-to-asymptotic}. Keep in mind the special values: $P_0 = P_1\equiv 0, P_2= \frac{|x|^2}{2}$ in $\mathbb{R}^{n+1}$ and  $f_0 = (2n)^{1/2}$ on $\mathbb{S}^n((2n)^{1/2})$.  Throughout this appendix, the polynomials $P_j$ are understood to be restricted to $\mathbb S^n((2n)^{1/2})$.  On this sphere, $f_0=(2n)^{1/2}$ and $2P_2=2n$.  We use the corrected recurrence \eqref{fm;induction;3}, whose spherical-radius factor is
$(2n)^{1-k/2}$.

\subsection{Proof of Corollary \ref{fm;Lemma;1}}\label{App: A1} The proof is similar to Corollary \ref{fm;Lemma;2}.

Define universal polynomials $\cQ^{(10)}_m(\bx_1,\cdots,\bx_{m-1},\bx_m)$ and $ \cQ^{(11)}_m(\bx_1,\cdots,\bx_{m-1})$ depending only on $m\geq 0$ and dimension $n$ in the following inductive way. First let 
\begin{align}
	 \cQ^{(10)}_0 \equiv \frac{1}{2}, \qquad\cQ^{(10)}_1(\bx_1) \equiv -(2n)^{-1/2}\bx_1.
\end{align}

\begin{align}
	 \cQ^{(10)}_{m+1}(\bx_1,\ldots,\bx_{m+1}) =& -(2n)^{-1/2}\bx_{m+1} -\sum_{\substack{\alpha_1+\cdots+\alpha_k+k=m+3\\ 2\leq k\leq m+2\\0\leq\alpha_1,\ldots,\alpha_k\leq m}} (2n)^{-k/2}\cQ^{(10)}_{k-2}(\bx_1,\ldots,\bx_{k-2})
	\bx_{\alpha_1}\cdots\bx_{\alpha_k}.
\end{align}

Once $\cQ^{(10)}_m$ is defined for all $m\geq 0$, we can define
\begin{align}
	 \cQ^{(11)}_{m}(\bx_1,\ldots,\bx_{m-1}) =& -\sum_{\substack{\alpha_1+\cdots+\alpha_k+k=m+2\\ 2\leq k\leq m+1\\0\leq\alpha_1,\ldots,\alpha_k\leq m-1}} (2n)^{-k/2}\cQ^{(10)}_{k-2}(\bx_1,\ldots,\bx_{k-2}) \bx_{\alpha_1}\cdots\bx_{\alpha_k}.
\end{align}
where we set $\bx_0$ to be 1. One can check that
\begin{align}\label{Def;Q10;2}
	 \cQ^{(10)}_{m}(\bx_1,\ldots,\bx_m) =-(2n)^{-1/2}\bx_m+\cQ^{(11)}_m(\bx_1,\ldots,\bx_{m-1}),
	\qquad m\geq1.
\end{align}

For each monomial $C\bx_{i_1}\cdots \bx_{i_p}$ in  $\cQ^{(10)}_m$ and $\cQ^{(11)}_m$, if we count the sum of indices $i_1+\cdots + i_p$, our definition forces them to be equal to $m$. Moreover, $\cQ^{(11)}_m$ does not have degree 1 term. This will give the desired form \eqref{Q11;def}. It remains to establish \eqref{fm;formula;3} by induction.

By formula \eqref{fm;induction;3} we have
\begin{align}
	 f_1=-(2n)^{-1/2}P_3.
\end{align}	
 Thus $\frac{P_3}{2P_2}=-(2n)^{-1/2}f_1$, which proves the case $m=1$.

Suppose that  \eqref{fm;formula;3} is proved for $1,..., m$.  We may assume $m+1\leq N-2$, otherwise there is nothing else to prove.  By \eqref{Def;Q10;2}, 
\begin{align*}
	\frac{P_{i}}{2P_2} = \cQ^{(10)}_{i-2}(f_1,\cdots, f_{i-2})
\end{align*} 
for $i=3,\cdots m+2$. Then we use \eqref{fm;induction;3} to obtain:
\begin{align}
	 f_{m+1} =& -(2n)^{1/2} \frac{P_{m+3}}{2P_2}  -\sum_{\substack{|\alpha|+k=m+3\\2\leq k\leq m+2\\ 
     0\leq\alpha_1,\ldots,\alpha_k\leq m}} (2n)^{-(k-1)/2}\frac{P_k}{2P_2} f_{\alpha_1}\cdots f_{\alpha_k}.
\end{align}
Using induction hypothesis, we can rewrite the above as:
\begin{align}
	 \frac{P_{m+3}}{2P_2} =& -(2n)^{-1/2}f_{m+1}
	-\sum_{\substack{|\alpha|+k=m+3\\2\leq k\leq m+2\\
	0\leq\alpha_1,\ldots,\alpha_k\leq m}} (2n)^{-k/2}\cQ^{(10)}_{k-2}(f_1,\ldots,f_{k-2}) f_{\alpha_1}\cdots f_{\alpha_k}\\
	 =& -(2n)^{-1/2}f_{m+1} +\cQ^{(11)}_{m+1}(f_1,\ldots,f_m).
\end{align}

This proves \eqref{fm;formula;3} for $m+1$, which finishes the induction step.

\subsection{Proof of Corollary \ref{fm;Lemma;2}}{\label{App: A2}} 
Corollary \ref{fm;Lemma;2} states that one can write $f_m$ directly in terms of $P_k$ (restricted on $\mathbb{S}^n((2n)^{1/2})$) with universal coefficients. 

Define universal polynomials $\cQ^{(20)}_m(\bx_1,\cdots,\bx_{m-1},\bx_m)$ and $ \cQ^{(21)}_m(\bx_1,\cdots,\bx_{m-1})$ depending only on $m\geq 0$ and dimension $n$ in the following inductive way. First let 
\begin{align}
	 \cQ^{(20)}_0\equiv (2n)^{1/2},\qquad \cQ^{(20)}_1(\bx_1)\equiv-\frac{(2n)^{1/2}}2\bx_1,\qquad \cQ^{(21)}_0\equiv\cQ^{(21)}_1\equiv0.
\end{align}
If we have already defined $\cQ^{(20)}_0, \cdots, \cQ^{(20)}_{m}$ and $\cQ^{(21)}_0, \cdots, \cQ^{(21)}_{m}$, then we define
\begin{align}
	 \cQ^{(21)}_{m+1}(\bx_1,\ldots,\bx_m) :=& -\frac12 \sum_{\substack{|\alpha|+k-2=m+1\\2\leq k\leq m+2\\
	0\leq\alpha_1,\ldots,\alpha_k\leq m}} (2n)^{(1-k)/2}\bx_{k-2}\prod_{l=1}^k \cQ^{(20)}_{\alpha_l}(\bx_1,\ldots,\bx_{\alpha_l}).
\end{align}
\begin{align}
	 \cQ^{(20)}_{m+1}(\bx_1,\ldots,\bx_{m+1}) :=& -\frac{(2n)^{1/2}}2\bx_{m+1} +\cQ^{(21)}_{m+1}(\bx_1,\ldots,\bx_m)\\ 
     =& -\frac12
	\sum_{\substack{|\alpha|+k-2=m+1\\2\leq k\leq m+3\\
	0\leq\alpha_1,\ldots,\alpha_k\leq m}} (2n)^{(1-k)/2}\bx_{k-2}\prod_{l=1}^k\cQ^{(20)}_{\alpha_l}(\bx_1,\ldots,\bx_{\alpha_l}).
\end{align}
where we set $\bx_0$ to be 1. Every factor $\cQ^{(20)}_{\alpha_l}$ on the right has $\alpha_l\leq m$ and has therefore already been defined; the possible term with $k=m+3$ is precisely the displayed linear term in $\bx_{m+1}$.  Thus the induction is
well posed.

\noindent\textbf{Remark:} Indeed, we are only defining $\cQ^{(20)}$ by induction. Once all $\cQ^{(20)}$ are defined, $\cQ^{(21)}$ only differs from $\cQ^{(20)}$ by a computable degree 1 monomial. 

For each monomial $C\bx_{i_1}\cdots \bx_{i_p}$ in  $\cQ^{(20)}_m$ and $\cQ^{(21)}_m$, if we count the sum of indices $i_1+\cdots + i_p$, our definition forces them to be equal to $m$. Moreover, $\cQ^{(21)}_m$ does not have degree 1 term. This will give the desired form \eqref{Q21;def}. It remains to establish \eqref{fm;formula;4} by induction.

By formula \eqref{fm;induction;3} we have
\begin{align}
	 f_1=-(2n)^{-1/2}P_3 =-\frac{(2n)^{1/2}}2\frac{P_3}{P_2}.
\end{align}	
This proves the case $m=1$. Suppose that the assertion is proved for $1,..., m$. We may assume $m+1\leq N-2$, otherwise there is nothing else to prove.  
Then by \eqref{fm;induction;3}: 
\begin{align}\label{fm;formula;5}
	 f_{m+1} =& -\frac{(2n)^{1/2}}2\frac{P_{m+3}}{P_2} -\frac12\sum_{\substack{|\alpha|+k-2=m+1\\2\leq k\leq m+2\\
	0\leq\alpha_1,\ldots,\alpha_k\leq m}} (2n)^{(1-k)/2}\frac{P_k}{P_2}f_{\alpha_1}\cdots f_{\alpha_k}.
\end{align}
 
By the induction hypothesis and definition of $\cQ^{(20)}$ and $\cQ^{(21)}$ and the fact that $P_0 = P_1\equiv 0$ we have 
\begin{align}
	 f_i=-\frac{(2n)^{1/2}}2\frac{P_{i+2}}{P_2} +\cQ^{(21)}_i\!\left(\frac{P_3}{P_2},\ldots, \frac{P_{i+1}}{P_2}\right) =\cQ^{(20)}_i\!\left(\frac{P_3}{P_2},\ldots, \frac{P_{i+2}}{P_2}\right),\qquad1\leq i\leq m.
\end{align}
We can also manually check that 
\begin{align}
	 f_0=(2n)^{1/2}=\cQ^{(20)}_0.
\end{align}
Plugging into \eqref{fm;formula;5} and using definition of $\cQ^{(20)}, \cQ^{(21)}$ we get 
\begin{align}
	 f_{m+1} =& -\frac{(2n)^{1/2}}2\frac{P_{m+3}}{P_2} -\frac12\sum_{\substack{|\alpha|+k-2=m+1\\2\leq k\leq m+2\\ 
     0\leq\alpha_1,\ldots,\alpha_k\leq m}} (2n)^{(1-k)/2}\frac{P_k}{P_2} \prod_{l=1}^k\cQ^{(20)}_{\alpha_l}\!\left( \frac{P_3}{P_2},\ldots,\frac{P_{\alpha_l+2}}{P_2}\right)\\
	 =& -\frac{(2n)^{1/2}}2\frac{P_{m+3}}{P_2} +\cQ^{(21)}_{m+1}\!\left(\frac{P_3}{P_2},\ldots, \frac{P_{m+2}}{P_2}\right)\\
	 =& \cQ^{(20)}_{m+1}\!\left( \frac{P_3}{P_2},\ldots,\frac{P_{m+3}}{P_2}\right).
\end{align}
This establishes the induction step. The Corollary is then proved.

\section{Explicit computations and examples}\label{app:compute}
\subsection{Computation of example, 1}
We compute the case $k=2$. Note that $\mathbf{K} = k^2+k = 6$. 

Let us define $\ba^{1},\ba^{2}, \ba^{3}$ by their nonzero elements, where the we write $S(\ba^i)$ under each $\ba^i$ for reader's convenience:

$\underset{(3,1)}{(\ba^1)^{+}_2} = 1;  \quad \underset{(3,6)}{(\ba^2)^{-}_3} = 1; \quad  \underset{(6,10)}{(\ba^3)^{-}_4} = 1; $ 

Note that by induction base we have $C^{(1)}_{\ba^1} = C^{(1)}_{\ba^2} = C^{(1)}_{\ba^3} = 1$.

Define $\ba := \ba^{1} + \ba^{2} + \ba^{3}$.
Note that $S(\ba) = S(\ba^1 + \ba^2 + \ba^3) = (10,15)$ and
\begin{align*}
	S(\ba^1+ \ba^2) = (5,6),\quad  S(\ba^1+ \ba^3) = (8,10), \quad S(\ba^2+ \ba^3) = (8,15), 
\end{align*}
Then we can compute
\begin{align*}
	C^{(2)}_{\ba} =& 4C^{(s4)}_{S(\ba^{1} + \ba^{2}), \ S(\ba^{3})}C^{(1)}_{\ba^{1} + \ba^{2}} + 4C^{(s4)}_{S(\ba^{1} + \ba^{3}), \ S(\ba^{2})}C^{(1)}_{\ba^{1} + \ba^{3}} + 4C^{(s4)}_{S(\ba^{2} + \ba^{3}), \ S(\ba^{1})}C^{(1)}_{\ba^{2} + \ba^{3}} + 4C^{(s5)}_{S(\ba^1), S(\ba^2), S(\ba^3)} 
\end{align*}
Using Definition \ref{Def;C3C4} and \ref{Def;C1C2} we have: 
\begin{align*}
	C^{(1)}_{\ba^i+\ba^j} = \frac{C^{(2)}_{\ba^i+\ba^j}}{C^{(3)}_{S(\ba^i+\ba^j)}} = \frac{4C^{(s4)}_{S(\ba^i),\ S(\ba^j)}}{C^{(3)}_{S(\ba^i+\ba^j)}}
\end{align*} 
Therefore, we reached the final expansion:
\begin{align*}
	C^{(2)}_{\ba}=& \frac{4C^{(s4)}_{S(\ba^{1} + \ba^{2}), \ S(\ba^{3})}\cdot 4C^{(s4)}_{S(\ba^{1}), \ S(\ba^{2})}}{C^{(3)}_{S(\ba^1+\ba^2)}} + \frac{4C^{(s4)}_{S(\ba^{1} + \ba^{3}), \ S(\ba^{2})}\cdot 4C^{(s4)}_{S(\ba^{1}), \ S(\ba^{3})}}{C^{(3)}_{S(\ba^1+\ba^3)}}+ \frac{4C^{(s4)}_{S(\ba^{2} + \ba^{3}), \ S(\ba^{1})}\cdot 4C^{(s4)}_{S(\ba^{2}), \ S(\ba^{3})}}{C^{(3)}_{S(\ba^2+\ba^3)}} \\
	  &+4C^{(s5)}_{S(\ba^1), S(\ba^2), S(\ba^3)} 
\end{align*}
Plugging in the numbers:
\begin{align*}
	C^{(2)}_{\ba} =& \frac{16C^{(s4)}_{(5,6),(6,10)}C^{(s4)}_{(3,1),(3,6)}}{C^{(3)}_{(5,6)}}+\frac{16 C^{(s4)}_{(8,10),(3,6)} C^{(s4)}_{(3,1),(6,10)}}{C^{(3)}_{(8,10)}}+ \frac{16 C^{(s4)}_{(8,15),(3,1)} C^{(s4)}_{(3,6),(6,10)}}{C^{(3)}_{(8,15)}} + 4C^{(s5)}_{(3,1),(3,6),(6,10)}
\end{align*}

One can check that
\begin{align*}
	 C^{(s4)}_{(5,6),(6,10)} =& C^{(4)}_{(5,6),(6,10)}+C^{(4)}_{(6,10),(5,6)} =  -130; \quad C^{(s4)}_{(3,1),(3,6)} = C^{(4)}_{(3,1),(3,6)} + C^{(4)}_{(3,6),(3,1)} = 33; \quad C^{(3)}_{(5,6)} = 10\\
\end{align*}
Similarly we have
\begin{align*}
	&C^{(s4)}_{(8,10),(3,6)} = -138; \quad C^{(s4)}_{(3,1),(6,10)} = 72; \quad C^{(3)}_{(8,10)} = 14\\
	&C^{(s4)}_{(8,15),(3,1)} = 227; \quad C^{(s4)}_{(3,6),(6,10)} = -138; \quad C^{(3)}_{(8,15)} = -26\\
	&C^{(s5)}_{(3,1),(3,6),(6,10)} = -312
\end{align*}
Plugging all in we get
\begin{align*}
	C^{(2)}_{\ba} =& \frac{16\cdot (-130)\cdot 33}{10} + \frac{16\cdot (-138)\cdot 72}{14} + \frac{16\cdot 227\cdot (-138)}{-26}  + 4\cdot(-312)  \\
	=& -6864 - \frac{79488}{7} + \frac{250608}{13} - 1248 = -\frac{17280}{91}
\end{align*}

This recovers \eqref{Comp;C2;3} with $k=2$ (and $\mathbf{K}=6$).

\subsection{Computation of example, 2}
We consider the case $k=3$. Note that $\mathbf{K} = k^2+k = 12$. 

Let $\ba^{1},\ba^{2}, \ba^{3}$ be defined by their only nonzero elements: $(\ba^1)^{+}_3 = 1;  \  (\ba^2)^{-}_4 = 1; \   (\ba^3)^{-}_{10} = 1$ and let $\ba = \ba^1 + \ba^2 + \ba^3$. 
Then   
\begin{align*}
	S(\ba^1) = (6,3); \quad S(\ba^2) = (6,10), \quad  S(\ba^3) = (45,55) 
\end{align*}
Moreover
\begin{align*}
	S(\ba^1+ \ba^2) = (11,12),\quad  S(\ba^1+ \ba^3) = (50,57), \quad S(\ba^2+ \ba^3) = (50,64),  \quad S(\ba) = (55,66).
\end{align*}

Plugging into \eqref{Comp;C2;1} we get:
\begin{align}\label{Comp;C2;13}
	C^{(2)}_{\ba} =& \frac{16C^{(s4)}_{(11,12),(45,55)}C^{(s4)}_{(6,3),(6,10)}}{C^{(3)}_{(11,12)}}+\frac{16 C^{(s4)}_{(50,57),(6,10)} C^{(s4)}_{(6,3),(45,55)}}{C^{(3)}_{(50,57)}} \\
    &+\frac{16 C^{(s4)}_{(50,64),(6,3)} C^{(s4)}_{(6,10),(45,55)}}{C^{(3)}_{(50,64)}} + 4C^{(s5)}_{(6,3),(6,10),(45,55)}
\end{align}

Using Definition \ref{Def;C3C4} one can check that:
\begin{align*}
	&C^{(s4)}_{(11,12),(45,55)} =  -1705; \quad C^{(s4)}_{(6,3),(6,10)} = 138; \quad C^{(3)}_{(11,12)} = 22\\
	&C^{(s4)}_{(50,57),(6,10)} = -2158; \quad C^{(s4)}_{(6,3),(45,55)} = 1605; \quad C^{(3)}_{(50,57)} = 58\\
	&C^{(s4)}_{(50,64),(6,3)} = 2910; \quad C^{(s4)}_{(6,10),(45,55)} = -2280; \quad C^{(3)}_{(50,64)} = -82\\
	&C^{(s5)}_{(6,3),(6,10),(45,55)} = -42600
\end{align*}
Plugging all in \eqref{Comp;C2;13} we get
\begin{align*}
	C^{(2)}_{\ba} =& \frac{16\cdot (-1705)\cdot 138}{22} + \frac{16\cdot (-2158)\cdot 1605}{58} + \frac{16\cdot 2910\cdot (-2280)}{-82}  + 4\cdot(-42600)  \\
	=& -171120 - \frac{27708720}{29} + \frac{53078400}{41} - 170400 = -\frac{2851200}{1189}
\end{align*} 

This recovers \eqref{Comp;C2;3} with $k=3$ (and $\mathbf{K}=12$).

\end{document}